\documentclass[11pt,english]{article}
\usepackage[T1]{fontenc}
\usepackage[utf8]{inputenc}
\usepackage{geometry}
\usepackage{color}
\usepackage{babel}
\usepackage{array}
\usepackage{multirow}
\usepackage{amsmath}
\usepackage{amsthm}
\usepackage{amssymb}
\usepackage{stackrel}
\usepackage{graphicx}
\usepackage[all]{xy}
\PassOptionsToPackage{normalem}{ulem}
\usepackage{ulem}
\usepackage[unicode=true,pdfusetitle,
 bookmarks=true,bookmarksnumbered=false,bookmarksopen=false,
 breaklinks=false,pdfborder={0 0 1},backref=page,colorlinks=true]
 {hyperref}
\hypersetup{
 linkcolor=blue,citecolor=blue}

\makeatletter

\providecommand{\tabularnewline}{\\}

\numberwithin{equation}{section}
\numberwithin{figure}{section}
\theoremstyle{plain}
\newtheorem{thm}{\protect\theoremname}[section]
\theoremstyle{plain}
\newtheorem{assumption}[thm]{\protect\assumptionname}
\theoremstyle{plain}
\newtheorem{prop}[thm]{\protect\propositionname}
\theoremstyle{definition}
\newtheorem{defn}[thm]{\protect\definitionname}
\theoremstyle{definition}
\newtheorem{example}[thm]{\protect\examplename}
\theoremstyle{plain}
\newtheorem{cor}[thm]{\protect\corollaryname}
\theoremstyle{remark}
\newtheorem{rem}[thm]{\protect\remarkname}
\theoremstyle{plain}
\newtheorem{lem}[thm]{\protect\lemmaname}
\theoremstyle{plain}
\newtheorem{conjecture}[thm]{\protect\conjecturename}

\usepackage{babel}

\usepackage{ytableau}

\usepackage{multirow}

\usepackage{appendix}

\makeatother

\providecommand{\assumptionname}{Assumption}
\providecommand{\conjecturename}{Conjecture}
\providecommand{\corollaryname}{Corollary}
\providecommand{\definitionname}{Definition}
\providecommand{\examplename}{Example}
\providecommand{\lemmaname}{Lemma}
\providecommand{\propositionname}{Proposition}
\providecommand{\remarkname}{Remark}
\providecommand{\theoremname}{Theorem}

\begin{document}
\global\long\def\emb{\mathbb{E}^{\mathrm{emb}}}%
\global\long\def\F{\mathbb{\mathbb{\mathbf{F}}}}%
 
\global\long\def\rk{\mathbb{\mathrm{rk}}}%
 
\global\long\def\crit{\mathbb{\mathrm{Crit}}}%
 
\global\long\def\Hom{\mathrm{Hom}}%
 
\global\long\def\defi{\stackrel{\mathrm{def}}{=}}%
 
\global\long\def\tr{{\cal T}r }%
 
\global\long\def\id{\mathrm{id}}%
 
\global\long\def\Aut{\mathrm{Aut}}%
 
\global\long\def\wl{w_{1},\ldots,w_{\ell}}%
 
\global\long\def\alg{\le_{\mathrm{alg}}}%
 
\global\long\def\ff{\stackrel{*}{\le}}%
 
\global\long\def\chimax{\chi^{\mathrm{max}}}%
 
\global\long\def\fix{\mathrm{fix}}%
 
\global\long\def\cyc{\mathrm{cyc}}%
 
\global\long\def\R{{\cal R}}%
 
\global\long\def\remb{{\cal R}^{\mathrm{emb}}}%
 
\global\long\def\H{{\cal H}}%
 
\global\long\def\lm{\mathfrak{lm}}%
 
\global\long\def\asyexp{\stackrel{\mathrm{asym.~exp.}}{\sim}}%
 
\global\long\def\plab{\pi_{1}^{\mathrm{lab}}}%
 
\global\long\def\chigrp{\chi^{\mathrm{grp}}}%
 
\global\long\def\tsg{\widetilde{\Sigma_{g}}}%
 
\global\long\def\C{{\cal C}}%
 
\global\long\def\core{\mathrm{Core}}%

\title{Local Statistics of Random Permutations from Free Products}
\author{Doron Puder~~~~~~Tomer Zimhoni}
\maketitle
\begin{abstract}
Let $\alpha$ and $\beta$ be uniformly random permutations of orders
$2$ and $3$, respectively, in $S_{N}$, and consider, say, the permutation
$\alpha\beta\alpha\beta^{-1}$. How many fixed points does this random
permutation have on average? The current paper studies questions of
this kind and relates them to surprising topological and algebraic
invariants of elements in free products of groups. 

Formally, let $\Gamma=G_{1}*\ldots*G_{k}$ be a free product of groups
where each of $G_{1},\ldots,G_{k}$ is either finite, finitely generated
free, or an orientable hyperbolic surface group. For a fixed element
$\gamma\in\Gamma$, a $\gamma$-random permutation in the symmetric
group $S_{N}$ is the image of $\gamma$ through a uniformly random
homomorphism $\Gamma\to S_{N}$. In this paper we study local statistics
of $\gamma$-random permutations and their asymptotics as $N$ grows.
We first consider $\mathbb{E}\left[\fix_{\gamma}\left(N\right)\right]$,
the expected number of fixed points in a $\gamma$-random permutation
in $S_{N}$. We show that unless $\gamma$ has finite order, the limit
of $\mathbb{E}\left[\fix_{\gamma}\left(N\right)\right]$ as $N\to\infty$
is an integer, and is equal to the number of subgroups $H\le\Gamma$
containing $\gamma$ such that $H\cong\mathbb{Z}$ or $H\cong C_{2}*C_{2}$.
Equivalently, this is the number of subgroups $H\le\Gamma$ containing
$\gamma$ and having (rational) Euler characteristic zero. We also
prove there is an asymptotic expansion for $\mathbb{E}\left[\fix_{\gamma}\left(N\right)\right]$
and determine the limit distribution of the number of fixed points
as $N\to\infty$. These results are then generalized to all statistics
of cycles of fixed lengths.
\end{abstract}
\tableofcontents{}

\section{Introduction\label{sec:Introduction}}

Let us begin with a special case of the problem we study in this paper.
Let $\alpha$ and $\beta$ be uniformly random permutations of orders
$2$ and $3$, respectively, in $S_{N}$, or, almost equivalently,
uniformly random permutations among all those satisfying $\alpha^{2}=1$
and $\beta^{3}=1$. Consider random permutations formed by some fixed
word in $\alpha$ and $\beta$, e.g., the random permutation $\alpha\beta\alpha\beta^{-1}$.
This random permutation can also be described as the image of the
element $xyxy^{-1}$ of $\Gamma=\left\langle x,y\,\middle|\,x^{2},y^{3}\right\rangle \cong C_{2}*C_{3}\cong\mathrm{PSL}(2,\mathbb{Z}$)
through a uniformly random homomorphism to $S_{N}$. This paper studies
the local statistics of such random permutations and shows that their
limit distributions (as $N\to\infty$) can be completely extracted
from certain algebraic and topological invariants of the corresponding
element ($xyxy^{-1}$ in the above example) in the group $\Gamma$.

More generally, given a f.g.~(finitely generated) group $\Gamma$,
the set $\Hom\left(\Gamma,S_{N}\right)$ of group homomorphisms from
$\Gamma$ to the symmetric group $S_{N}$ is finite, and is a natural
object of study, being the set of all permutation-representations
(actions) of $\Gamma$ on a set of size $N$. This set also lies in
one-to-one correspondence with all $N$-sheeted covering spaces of
a ``nice'' topological space\footnote{For this correspondence, $X$ needs to be connected, locally path-connected
and semilocally simply-connected. Moreover, $X$ is equipped with
a basepoint $x_{0}\in X$, the group $\Gamma$ is identified with
$\pi_{1}\left(X,x_{0}\right)$, and $N$-sheeted covering spaces $\rho\colon\tilde{X}\to X$
are equipped with a bijection between $\left\{ 1,\ldots,N\right\} $
and the fiber $\rho^{-1}\left(x_{0}\right)$. See \cite[pp.~68-70]{hatcher2005algebraic}.} $X$ with fundamental group $\Gamma$. The set $\Hom\left(\Gamma,S_{N}\right)$
also shows up in the study of residual properties of $\Gamma$, of
its profinite topology, of its subgroup growth and so on.

In this paper we study $\Hom\left(\Gamma,S_{N}\right)$ where $\Gamma$
is a free product of finite, free, and (orientable) hyperbolic surface
groups. Namely,
\begin{assumption}
\label{assu:Gamma}Throughout this paper, we let 
\begin{equation}
\Gamma=G_{1}*\ldots*G_{k}\label{eq:Gamma}
\end{equation}
for some $k\in\mathbb{Z}_{\ge1}$, and for every $i=1,\ldots,k$,
the group $G_{i}$ is either a finite group, a f.g.~free group, or
the fundamental group $\Lambda_{g}\cong\left\langle a_{1},b_{1},\ldots,a_{g},b_{g}\,\middle|\,\left[a_{1},b_{1}\right]\cdots\left[a_{g},b_{g}\right]\right\rangle $
of a closed orientable surface of genus $g\ge2$. Denote by $m\left(\Gamma\right)\in\mathbb{Z}_{\ge1}$
the $\mathrm{lcm}$ of the orders of the finite factors in \eqref{eq:Gamma}
(in particular, $m=1$ if and only if $\Gamma$ is torsion free).
\end{assumption}

The case of $k=1$, namely, when $\Gamma$ is simply a finite, free
or surface group, was studied in previous works mentioned below and
on which we build upon in the current paper. Indeed, the main innovation
of the current paper is in treating non-trivial free products. Non-trivial
free products in our setting include the modular group $\mathrm{PSL}_{2}\left(\mathbb{Z}\right)\cong C_{2}*C_{3}$
(we denote by $C_{r}$ the cyclic group of order $r$) and all its
f.g.~subgroups, as well as many other f.g.~orientable Fuchsian groups:
those with parabolic or hyperbolic boundary generators (see \cite[pp.~553]{LiebeckShalev}). 

The mere number of homomorphisms $\Gamma\to S_{N}$ is well understood
-- see Section \ref{subsec:Related-works}. Another natural question
is whether a uniformly random action of $\Gamma$ on $\left\{ 1,\ldots,N\right\} $
is transitive, or, equivalently, if a random $N$-cover of a corresponding
space is connected. Here, known results are striking: in many of the
cases covered by our setting, the image of a random homomorphism $\varphi\colon\Gamma\to S_{N}$
is not only a.a.s.\footnote{We write a.a.s., or asymptotically almost surely, to describe an event
which has probability tending to $1$ as the implied parameter ($N$
in the current case) tends to infinity.}~a transitive subgroup of $S_{N}$, but actually a.a.s.~contains
the alternating group $A_{N}$. This is true for non-abelian free
groups by the famous result of Dixon \cite{dixon1969probability}
that two uniformly random permutations a.a.s.~generate $A_{N}$ or
$S_{N}$. It is true for hyperbolic surface groups and for free products
of cyclic groups which are Fuchsian by \cite[Thm.~1.12]{LiebeckShalev}.
Of course, adding free factors to $\Gamma$ can only enlarge the image
of a random homomorphism. If $\Gamma$ is a finite group, $\Gamma\cong\mathbb{Z}$
or $\Gamma\cong C_{2}*C_{2}$, it is known (and easy) that the image
of a random homomorphism to $S_{N}$ is not a.a.s.~transitive. It
is probable that in all remaining cases\footnote{The remaining cases are non-trivial free products where all factors
are finite groups or $\mathbb{Z}$, with at most one cyclic factor
or precisely two cyclic factors both of which are $C_{2}$ (excluding,
of course, the group $C_{2}*C_{2}$ itself).}, the image of a random homomorphism to $S_{N}$ should also contain
$A_{N}$ a.a.s., but we do not know of a reference. 

\subsection{Fixed points in a $\gamma$-random permutation}

In this paper, however, our focus is different. We fix an element
$\gamma\in\Gamma$ and consider its image through a uniformly random
homomorphism $\varphi\colon\Gamma\to S_{N}$. We call the resulting
random permutation a $\gamma$-random permutation. In the topological
setting, the image of $\gamma$ corresponds to the structure of the
lifts of the corresponding closed curve in the space $X$ to a random
$N$-sheeted cover. We concentrate on the \emph{local} statistics
of a $\gamma$-random permutation: the distribution of the number
of cycles of given fixed lengths. 

We begin by presenting our results for the distribution of the number
of fixed points, and later generalize to cycles of arbitrary fixed
lengths. Denote by $\fix_{\gamma}\left(N\right)$ the random variable
that counts the number of fixed points of a $\gamma$-random permutation
in $S_{N}$.

There is a clear distinction between torsion elements and elements
of infinite order. Any non-trivial torsion element $\gamma$ of $\Gamma$
is conjugate into one of the finite factors (the infinite factors
are torsion-free). So the statistics of a $\gamma$-random permutation
only depend on the particular factor it is conjugate into. In this
case, the following proposition readily follows from results in \cite{muller1997finite}: 
\begin{prop}
\label{prop:exp-torsion-element}Let $\gamma\in\Gamma$ have finite
order and let $\left|\gamma\right|$ denote its order. Then
\begin{equation}
\mathbb{E}\left[\fix_{\gamma}\left(N\right)\right]=N^{1/\left|\gamma\right|}+O\left(N^{1/\left(2\left|\gamma\right|\right)}\right).\label{eq:E=00005Bfix=00005D of torsion}
\end{equation}
\end{prop}

For instance, Example \ref{exa:x^2  in C_4} explains why $\mathbb{E}\left[\fix_{\gamma}\left(N\right)\right]=N^{1/2}+N^{1/4}+O\left(1\right)$
for $\gamma=x^{2}\in\Gamma=C_{4}=\left\langle x\right\rangle $. (More
general statistics of $\gamma$-random permutations when $\gamma$
has finite order can be derived from \cite{muller2010statistics}.)

The picture is completely different for elements of infinite order.
Consider first the case where $\Gamma\cong\mathbb{Z}=\left\langle x\right\rangle $
and $\gamma=x^{q}$. A $x^{q}$-random permutation is simply the $q$-power
of a uniformly random permutation, and the local statistics here are
well-understood: as $N\to\infty$ they converge in distribution to
a sum of suitable independent Poisson variables -- see \cite{diaconis1994eigenvalues}.
In particular, $\mathbb{E}\left[\fix_{\gamma}\left(N\right)\right]$
converges to $d\left(q\right)$, the number of positive divisors of
$q$. Nica showed in \cite{nica1994number} that the same is true
for elements of a free group: if $\Gamma$ is a f.g.~free group and
$1\ne\gamma\in\Gamma$, write $\gamma=\gamma_{0}^{~q}$ with $q\in\mathbb{Z}_{\ge1}$
and $\gamma_{0}\in\Gamma$ a non-power. Then $\fix_{\gamma}\left(N\right)$
converges in distribution, as $N\to\infty$, to the same sum of Poissons
as $x^{q}\in\mathbb{Z}$ does. In particular, the limit distribution
depends only on $q$ and not on $\gamma_{0}$.

The case of orientable surface groups was recently studied by Magee
and the first author in \cite{magee2020asymptotic}. While the presence
of a relation makes the analysis in this case by far more complicated
than in free groups, it is nevertheless shown in \cite{magee2020asymptotic}
that Nica's results about free group elements hold in surface groups
as well. In particular, for $g\ge2$ and $1\ne\gamma\in\Lambda_{g}$,
if we write $\gamma=\gamma_{0}^{~q}$ with $\gamma_{0}$ a non-power
and $q\in\mathbb{Z}_{\ge1}$, then $\fix_{\gamma}\left(N\right)$
converges in distribution, as $N\to\infty$, to the same sum of Poissons
as $x^{q}\in\mathbb{Z}$ does\footnote{\label{fn:To-be-precise MP20}To be precise, this result is not stated
explicitly in \cite{magee2020asymptotic}. The paper \cite{magee2020asymptotic}
is long as is and its main feature is the development of a new representation-theoretic
method to compute integrals over $\Hom\left(\Lambda_{g},S_{N}\right)$.
To keep that paper to a manageable size, it states explicitly only
the result that $\mathbb{E}\left[\fix_{\gamma}\left(N\right)\right]\stackrel{N\to\infty}{\to}d\left(q\right)$.
However, the stronger results about the limit distributions of $\fix_{\gamma}\left(N\right)$
and other local statistics follow readily from \cite{magee2020asymptotic}.
At any rate, the proofs we give in the current paper heavily rely
on \cite{magee2020asymptotic} and encompass, as a special case, the
case of $\gamma\in\Lambda_{g}$.}.

An interesting twist arises when one considers groups with torsion,
and, in particular, free products of finite groups, as in the current
paper. It turns out that the property of an element $\gamma\in\Gamma$
which determines the local statistics of a $\gamma$-random permutation
in the limit is not only whether it is a power and the value of the
exponent, but rather, \emph{the array of subgroups of Euler characteristic
zero containing it} \emph{and its powers}. To explain this phenomenon,
let us first recall what the Euler characteristic is for the groups
in play in this paper. 
\begin{defn}[Euler Characteristic of groups]
\label{def:EC} The (rational) Euler characteristic of a group $\Gamma$,
denoted $\chi\left(\Gamma\right)$, is a rational number defined for
groups with a finite index subgroup of finite homological type --
see \cite[Sec.~IX.7]{brown1982cohomology}. For the sake of the current
paper, it is enough to mention that 
\begin{itemize}
\item For a finite group $G$, $\chi\left(G\right)=\frac{1}{\left|G\right|}$.
\item For a rank-$r$ free group, $\chi\left(\F_{r}\right)=1-r$.
\item For a surface group $\Lambda_{g}\cong\left\langle a_{1},b_{1},\ldots,a_{g},b_{g}\,\middle|\,\left[a_{1},b_{1}\right]\cdots\left[a_{g},b_{g}\right]\right\rangle $,
$\chi\left(\Lambda_{g}\right)=2-2g$.
\item If $G_{1}$ and $G_{2}$ have a well-defined Euler characteristic,
then so does $G_{1}*G_{2}$, and
\[
\chi\left(G_{1}*G_{2}\right)=\chi\left(G_{1}\right)+\chi\left(G_{2}\right)-1.
\]
\end{itemize}
\end{defn}

So, for example, $\chi\left(\mathrm{PSL}_{2}\left(\mathbb{Z}\right)\right)=\chi\left(C_{2}*C_{3}\right)=\frac{1}{2}+\frac{1}{3}-1=-\frac{1}{6}$.
By Kurosh subgroup theorem, if $\Gamma$ is as in \eqref{eq:Gamma},
then every subgroup of $\Gamma$ is a free product of (conjugates
of) subgroups of the factors of $\Gamma$ together with, possibly,
a free group factor. As every subgroup of a free group is free, and
every subgroup of $\Lambda_{g}$ $\left(g\ge2\right)$ is either free
or $\Lambda_{h}$ for some $h\ge g$ (e.g.~\cite{scott1978subgroups}),
we get that every f.g.~subgroup of $\Gamma$ is, too, of the form
\eqref{eq:Gamma}, and, in particular, has a well-defined EC (Euler
characteristic) as in Definition \ref{def:EC}. Note that when restricting
to the groups considered in this paper, the only groups with positive
EC are finite groups, and the only groups with EC zero are $\mathbb{Z}$
and $C_{2}*C_{2}$. 

For $\Gamma$ as in \eqref{eq:Gamma} and $\gamma\in\Gamma$, denote
\begin{equation}
\H_{\gamma}\defi\left\{ H\le\Gamma\,\middle|\,\gamma\in H~\mathrm{and}~\chi\left(H\right)=0\right\} .\label{eq:H_gamma}
\end{equation}
Equivalently, this is the set of subgroups of $\Gamma$ containing
$\gamma$ which are isomorphic to $\mathbb{Z}$ or to $C_{2}*C_{2}$.
It is not hard to show (and see Corollary \ref{cor:H_gamma is finite})
that this set is finite for every non-torsion $\gamma\in\Gamma$.
\begin{thm}
\label{thm:limit expectation of fix}Let $\Gamma$ be as in \eqref{eq:Gamma}
and $\gamma\in\Gamma$ have infinite order. Then 
\begin{equation}
\mathbb{E}\left[\fix_{\gamma}\left(N\right)\right]\stackrel{N\to\infty}{\to}\left|\H_{\gamma}\right|.\label{eq:subgroups of EC 0}
\end{equation}
More precisely, writing $m=m\left(\Gamma\right)$ as in Assumption
\ref{assu:Gamma}, we have
\[
\mathbb{E}\left[\fix_{\gamma}\left(N\right)\right]=\left|\H_{\gamma}\right|+O\left(N^{-1/m}\right).
\]
\end{thm}

In Table \ref{tab:examples for limit of E=00005Bfix=00005D} we illustrate
this result with some concrete examples. This generalizes the above-mentioned
results in free groups and surface groups, as these groups are torsion
free and have no embedded copies of $C_{2}*C_{2}$. Thus, in this
case ${\cal H}_{\gamma}$ contains only infinite cyclic groups: if
$\gamma=\gamma_{0}^{~q}$ with $\gamma_{0}$ a non-power, this set
is ${\cal H}_{\gamma}=\left\{ \left\langle \gamma_{0}^{~d}\right\rangle \,\middle|\,1\le d|q\right\} $
(see Footnote \ref{fn:cyclic subgroups containing gamma in free and surface groups}
for some details).

\begin{table}
\begin{centering}
\begin{tabular}{|>{\centering}p{3.5cm}|c|c|c|}
\hline 
$\Gamma$ & $\gamma$ & $\lim_{N\to\infty}\mathbb{E}\left[\fix_{\gamma}\left(N\right)\right]$ & ${\cal H}_{\gamma}$\tabularnewline
\hline 
\hline 
$C_{2}*C_{2}=\left\langle x\right\rangle *\left\langle y\right\rangle $ & \multirow{5}{*}{$\left[x,y\right]$} & $5$ & $\left\langle \gamma\right\rangle ,\left\langle xy\right\rangle ,\left\langle x,yxy\right\rangle ,\left\langle xyx,y\right\rangle ,\Gamma$\tabularnewline
\cline{1-1} \cline{3-4} \cline{4-4} 
$C_{2}*C_{q}=\left\langle x\right\rangle *\left\langle y\right\rangle $,
$2<q$ &  & $2$ & $\left\langle \gamma\right\rangle ,\left\langle x,yxy^{-1}\right\rangle $\tabularnewline
\cline{1-1} \cline{3-4} \cline{4-4} 
$C_{m}*C_{q}=\left\langle x\right\rangle *\left\langle y\right\rangle $,
$2<m\le q$ &  & $1$ & $\left\langle \gamma\right\rangle $\tabularnewline
\cline{1-1} \cline{3-4} \cline{4-4} 
$\F_{2}=\mathbb{Z}*\mathbb{Z}=\left\langle x\right\rangle *\left\langle y\right\rangle $ &  & $1$ & $\left\langle \gamma\right\rangle $\tabularnewline
\cline{1-1} \cline{3-4} \cline{4-4} 
$\Lambda_{2}=\left\langle x,y,z,t\,\middle|\,\left[x,y\right]\left[z,t\right]\right\rangle $ &  & $1$ & $\left\langle \gamma\right\rangle $\tabularnewline
\hline 
$C_{2}*C_{2}=\left\langle x\right\rangle *\left\langle y\right\rangle $ & $\left(xy\right)^{3}$ & $6$ & $\left\langle \gamma\right\rangle ,\left\langle xy\right\rangle ,\left\langle x,yxyxy\right\rangle ,\left\langle yxy,xyx\right\rangle ,\left\langle y,xyxyx\right\rangle ,\Gamma$\tabularnewline
\hline 
$C_{3}*C_{4}=\left\langle x\right\rangle *\left\langle y\right\rangle $ & $\left[x,y^{2}\right]$ & $2$ & $\left\langle \gamma\right\rangle ,\left\langle xy^{2}x^{-1},y^{2}\right\rangle $\tabularnewline
\hline 
\end{tabular}
\par\end{centering}
\caption{This table illustrates Theorem \ref{thm:limit expectation of fix}
and gives the limit value of $\mathbb{E}\left[\protect\fix_{\gamma}\left(N\right)\right]$
as $N\to\infty$ for various infinite-order elements $\gamma$ in
various groups. The limit is the number of subgroups $H\le\Gamma$
with $\chi\left(H\right)=0$ containing $\gamma$, and their full
list in each case is given in the rightmost column. \label{tab:examples for limit of E=00005Bfix=00005D}}
\end{table}

One can give a unified statement encompassing both Proposition \ref{prop:exp-torsion-element}
and Theorem \ref{thm:limit expectation of fix}: for every $\gamma\in\Gamma$
\begin{equation}
\mathbb{E}\left[\fix_{\gamma}\left(N\right)\right]=c_{\gamma}\cdot N^{1/\left|\gamma\right|}\left(1+O\left(N^{-1/m}\right)\right),\label{eq:leading order of E=00005Bfix=00005D for torsion and non-torsion}
\end{equation}
where $\left|\gamma\right|$ is the order of $\gamma$, $\frac{1}{\infty}\defi0$,
$c_{\gamma}$ is the number of subgroups $H\le\Gamma$ containing
$\gamma$ and of Euler characteristic $\frac{1}{\left|\gamma\right|}$,
and $m=m\left(\Gamma\right)$. Indeed, \eqref{eq:leading order of E=00005Bfix=00005D for torsion and non-torsion}
coincides with Theorem \ref{thm:limit expectation of fix} when $\left|\gamma\right|=\infty$.
If $\gamma$ is a torsion element, then $\left\langle \gamma\right\rangle $
is the sole subgroup of EC $\frac{1}{\left|\gamma\right|}$ containing
$\gamma$, and so $c_{\gamma}=1$. Both bounds on the error term ---
$O\left(N^{1/\left|\gamma\right|-1/m}\right)$ in \eqref{eq:leading order of E=00005Bfix=00005D for torsion and non-torsion}
and $O\left(N^{1/\left(2\left|\gamma\right|\right)}\right)$ in \eqref{eq:E=00005Bfix=00005D of torsion}
--- hold in this case. See Section \ref{subsec:sub-coverings-finite-groups}
for details.

In fact, the role of EC of subgroups in local statistics of random
homomorphisms $\Gamma\to S_{N}$ goes much further. Roughly, for a
natural choice of a nice space $X_{\Gamma}$ with fundamental group
$\Gamma$, let $p\colon X\to X_{\Gamma}$ be an arbitrary covering
space, let $Y\subseteq X$ be a compact subspace, and for simplicity
assume that $Y$ is connected. Let $\plab\left(Y\right)\le\Gamma$
be the (conjugacy class of the) subgroup corresponding to $Y$, namely,
this is $\plab\left(Y\right)\defi p_{*}\left(\pi_{1}\left(Y\right)\right)\le\pi_{1}\left(X_{\Gamma}\right)=\Gamma$.
Then the average number of embeddings of $Y$ in a random $N$-cover
of $X_{\Gamma}$, or more precisely the average number of injective
lifts of $p|_{Y}\colon Y\to X_{\Gamma}$ to a random $N$-cover of
$X_{\Gamma}$, is of order $N^{\chi\left(\plab\left(Y\right)\right)}$.
The precise statement is given in Theorem \ref{thm:E^emb is N^chi and asym expansion}
below (and see Remark \ref{rem:point of confusion}). This result
is an important ingredient in the proof of Theorem \ref{thm:limit expectation of fix}
and the other main results.

The same method we use to prove Theorem \ref{thm:limit expectation of fix}
can be used to compute the limit of all moments of $\fix_{\gamma}\left(N\right)$
and, by the method of moments, prove the following.
\begin{thm}
\label{thm:limit distribution of fix}Let $\Gamma$ be as in \eqref{eq:Gamma}
and $\gamma\in\Gamma$ have infinite order. Let $H_{1},\ldots,H_{t}$
be \emph{representatives of the conjugacy classes of subgroups }represented
in ${\cal H}_{\gamma}$. For $i=1,\ldots,t$ let $\alpha_{i}=\left|\left\{ {\cal H}_{\gamma}\cap H_{i}^{~\Gamma}\right\} \right|$
be the number of conjugates of $H_{i}$ in ${\cal H}_{\gamma}$, and
let $\beta_{i}\defi\left[N_{\Gamma}\left(H_{i}\right):H_{i}\right]$
be the index of $H_{i}$ in its normalizer. Then 
\begin{equation}
\fix_{\gamma}\left(N\right)\stackrel[N\to\infty]{\mathrm{dis}}{\longrightarrow}\sum_{i=1}^{t}\alpha_{i}\beta_{i}Z_{1/\beta_{i}},\label{eq:limit distr of fix}
\end{equation}
where $Z_{\lambda}\sim\mathrm{Poi}\left(\lambda\right)$ (a random
variable with Poisson distribution with parameter $\lambda$), $Z_{1/\beta_{1}},\ldots,Z_{1/\beta_{t}}$
are $t$ independent variables, and ``$\stackrel{\mathrm{dis}}{\longrightarrow}$''
denotes convergence in distribution.
\end{thm}

\begin{example}
Consider the penultimate element from Table \ref{tab:examples for limit of E=00005Bfix=00005D}:
$\gamma=\left(xy\right)^{3}\in\Gamma=C_{2}*C_{2}=\left\langle x\right\rangle *\left\langle y\right\rangle $.
In this case, the elements of ${\cal H}_{\gamma}$ belong to four
different conjugacy classes: $\left\{ \left\langle \gamma\right\rangle \right\} $,
$\left\{ \left\langle xy\right\rangle \right\} $, $\left\{ \Gamma\right\} $
and\linebreak{}
$\left\{ \left\langle x,yxyxy\right\rangle ,\left\langle yxy,xyx\right\rangle ,\left\langle y,xyxyx\right\rangle \right\} $,
so $t=4$, $\alpha_{1}=\alpha_{2}=\alpha_{3}=1$ and $\alpha_{4}=3$.
In addition, $\left\langle \gamma\right\rangle \trianglelefteq\Gamma$
and $\left\langle xy\right\rangle \trianglelefteq\Gamma$ and so $\beta_{1}=\left[\Gamma:\left\langle \gamma\right\rangle \right]=6$,
$\beta_{2}=\left[\Gamma:\left\langle xy\right\rangle \right]=2$ and
$\beta_{3}=\left[\Gamma:\Gamma\right]=1$. Finally, $N_{\Gamma}\left(\left\langle x,yxyxy\right\rangle \right)=\left\langle x,yxyxy\right\rangle $
and so $\beta_{4}=1$. Hence in this case
\[
\fix_{\gamma}\left(N\right)\stackrel[N\to\infty]{\mathrm{dis}}{\longrightarrow}6Z_{1/6}+2Z_{1/2}+Z_{1}+3Z_{1}
\]
(here the last two $Z_{1}$'s are two distinct, independent Poisson
variables with parameter $1$ each).
\end{example}

Given a non-torsion $\gamma\in\Gamma$, the set $\H_{\gamma}$ can
be generated by following the procedure\footnote{In short, this procedure involves constructing a $1$-dimensional
``sub-cover'' corresponding to $\gamma$, producing all surjective
morphisms from it (namely, construction all 'sub-covers' which are
its quotients, with the map between them), and recognizing the quotients
with labeled fundamental group of Euler characteristic zero. See Sections
\ref{sec:The-space-X_Gamma its covers and sub-covers} and \ref{sec:limit-distr-and-asymptotic-expansion of fix}
for details.} in the proof of Theorem \ref{thm:limit expectation of fix} in Section
\ref{sec:limit-distr-and-asymptotic-expansion of fix}.

As a special case, we retrieve the known results when $\Gamma$ is
free (originally due to Nica \cite{nica1994number}) or a hyperbolic
orientable surface group (due to Magee-Puder \cite{magee2020asymptotic}
-- and see Footnote \ref{fn:To-be-precise MP20}). Recall that in
these cases, if $\gamma=\gamma_{0}^{~q}\in\Gamma$ with $\gamma_{0}$
a non-power, then ${\cal H}_{\gamma}=\left\{ \left\langle \gamma_{0}^{~d}\right\rangle \,\middle|\,1\le d|q\right\} $.
Moreover, $N_{\Gamma}\left(\left\langle \gamma_{0}^{~d}\right\rangle \right)=\left\langle \gamma_{0}\right\rangle $
and $\left\langle \gamma_{0}\right\rangle $ is malnormal. Thus Theorem
\ref{thm:limit distribution of fix} translates to the following.
\begin{cor}
\cite{nica1994number,magee2020asymptotic}\label{cor:limit dsitribution for free and surface}
Assume that $\Gamma$ is either free or a hyperbolic orientable surface
group (so $\Gamma=\Lambda_{g}=\left\langle a_{1},b_{1},\ldots,a_{g},b_{g}\,\middle|\,\left[a_{1},b_{1}\right]\cdots\left[a_{g},b_{g}\right]\right\rangle $
with $g\ge2$). Let $1\ne\gamma=\gamma_{0}^{~q}\in\Gamma$ with $\gamma_{0}$
a non-power and $q\in\mathbb{Z}_{\ge1}$. Then 
\[
\fix_{\gamma}\left(N\right)\stackrel[N\to\infty]{\mathrm{dis}}{\longrightarrow}\sum_{1\le d|q}dZ_{1/d}.
\]
\end{cor}

The following quantitative version of the residual finiteness of $\Gamma$,
follows from Theorem \ref{thm:limit distribution of fix} by a simple
application of the Markov inequality (and see \cite[Sec.~1.4]{magee2020asymptotic}
for some background).
\begin{cor}
\label{cor:quantitative RF}Given a non-torsion element $\gamma\in\Gamma$
and $r\in\mathbb{Z}_{\ge1}$,
\[
\frac{\left|\left\{ \varphi\in\Hom\left(\Gamma,S_{N}\right)\,\middle|\,\varphi\left(\gamma\right)\ne\id\right\} \right|}{\left|\Hom\left(\Gamma,S_{N}\right)\right|}\ge1-\frac{c_{r}\left(\gamma\right)}{N^{r}}-O\left(\frac{1}{N^{r+1/m}}\right),
\]
where $c_{r}\left(\gamma\right)=\mathbb{E}\left[\left(\sum\alpha_{i}\beta_{i}Z_{1/\beta_{i}}\right)^{r}\right]$,
and $\alpha_{i}$ and $\beta_{i}$ are the parameters from Theorem
\ref{thm:limit distribution of fix}.
\end{cor}

(The corollary follows from the fact that $\mathbb{E}\left[\fix_{\gamma}\left(N\right)^{r}\right]=c_{r}\left(\gamma\right)+O\left(N^{-1/m}\right)$,
which follows from Theorem \ref{thm:limit distribution of fix}, Equation
\eqref{eq:r-th moment by lifts} and Theorem \ref{thm:E^emb is N^chi and asym expansion}.)
\begin{rem}
\label{rem:Z^2}It is not clear to us to what extent the results in
this paper can be extended to more general f.g.~groups. There are
certainly groups which behave very differently. As an example, consider
the genus-1 surface group $\mathbb{Z}^{2}=\left\langle x,y\,\middle|\,\left[x,y\right]\right\rangle $.
The image of $x$ in a uniformly random homomorphism $\mathbb{Z}^{2}\to G$
to some finite group $G$ is a uniformly random element in a uniformly
random conjugacy class. So if $\varphi\colon\mathbb{Z}^{2}\to S_{N}$
is uniformly random, $\varphi\left(x\right)$ has the cycle structure
of a uniformly random conjugacy class. In particular, $\mathbb{E}\left[\fix_{x}\left(N\right)\right]$
is the average number of rows of length one in a uniformly random
Young diagram with $N$ blocks. It it not hard to see that this number
is 
\begin{equation}
\mathbb{E}\left[\fix_{x}\left(N\right)\right]=\frac{p\left(0\right)+p\left(1\right)+\ldots+p\left(N-1\right)}{p\left(N\right)},\label{eq:E=00005Bfix=00005D in Z^2}
\end{equation}
where $p$ is the partition function. This number is of order $\sqrt{N}$.
In fact, $\fix_{x}\left(N\right)\cdot\frac{\pi}{\sqrt{6N}}$ converges
in distribution to the exponential distribution with expectation 1
-- see \cite[Thm.~2.1]{fristedt1993structure}. Notice there are
infinitely many EC-zero subgroups containing $x$: $\left\{ \left\langle x,y^{j}\right\rangle \,\middle|\,j\in\mathbb{Z}_{\ge0}\right\} $.
See also Section \ref{sec:Open-questions}.
\end{rem}

\subsection{Asymptotic expansion of $\mathbb{E}\left[\protect\fix_{\gamma}\left(N\right)\right]$}

When $\Gamma$ is free and $\gamma\in\Gamma$, it is not hard to show
that $\mathbb{E}\left[\fix_{\gamma}\left(N\right)\right]$ is given
by a rational function in $N$ for every large enough $N$ (see \cite{nica1994number,linial2010word}).
For example, for $\gamma=\left[x,y\right]\in\F_{2}=\F\left(x,y\right)$
we have $\mathbb{E}\left[\fix_{\gamma}\left(N\right)\right]=\frac{N}{N-1}$
for every $N\ge2$. Such a clean result does not hold for the other
groups we consider here. Yet, asymptotic expansion, in the form of
rational or ``fractional rational'' approximation, does exist. 
\begin{defn}[Asymptotic expansion]
\label{def:asymptotic expansion} Let $f\colon\mathbb{Z}_{\ge0}\to\mathbb{R}$.
Let $k_{1}>k_{2}>\ldots$ be a decreasing sequence of real numbers
and $a_{k_{1}},a_{k_{2}},\ldots$ a sequence of real numbers. We say
that $f$ has asymptotic expansion given by $a_{k_{1}},a_{k_{2}},\ldots$
and denote
\[
f\left(N\right)\asyexp a_{k_{1}}N^{k_{1}}+a_{k_{2}}N^{k_{2}}+a_{k_{3}}N^{k_{3}}+\ldots,
\]
or simply $f\left(N\right)\asyexp\sum_{j=0}^{\infty}a_{k_{j}}N^{k_{j}}$,
if for every $\ell\in\mathbb{Z}_{\ge1}$ we have
\[
f\left(N\right)=a_{k_{1}}N^{k_{1}}+a_{k_{2}}N^{k_{2}}+\ldots+a_{k_{\ell}}N^{k_{\ell}}+O\left(N^{k_{\ell+1}}\right).
\]
\end{defn}

The most recent development here is the easiest to state: 
\begin{thm}
\cite[Thm.~1.1]{magee2020asymptotic}\label{thm:MP - rational approx}
For any $\gamma\in\Lambda_{g}$ there are rational numbers $a_{i}=a_{i}\left(\gamma\right)$
for $i=1,0,-1,-2,\ldots$ such that 
\begin{equation}
\mathbb{E}\left[\fix_{\gamma}\left(N\right)\right]\asyexp a_{1}N+a_{0}+a_{-1}N^{-1}+a_{-2}N^{-2}+\ldots.\label{eq:asymptotic expansion is surface groups}
\end{equation}
\end{thm}

The case of finite groups has a long history. The expected number
of fixed points is intimately related to the size of $\Hom\left(G,S_{N}\right)$:
indeed, if $\left\langle x\right\rangle =C_{q}$ is a cyclic group,
then $\mathbb{E}\left[\fix_{x}\left(N\right)\right]=N\cdot\frac{\left|\Hom\left(C_{q},S_{N-1}\right)\right|}{\left|\Hom\left(C_{q},S_{N}\right)\right|}$.
Already in 1951 it was conjectured by Chowla, Herstein and Moore \cite{chowla1951recursions}
that $\frac{\left|\Hom\left(C_{2},S_{N}\right)\right|}{\left|\Hom\left(C_{2},S_{N-1}\right)\right|}$
has asymptotic expansion of the form $N^{1/2}+A+BN^{-1/2}+CN^{-1}+DN^{-3/2}+\ldots$,
a conjecture proven slightly later by Moser and Wyman \cite{moser1955solutions}.
After many milestones, a complete solution for arbitrary finite groups
was given by Müller in 1997.
\begin{thm}
\cite[Thm.~6]{muller1997finite}\label{thm:Muller} Let $G$ be a
finite group of order $m\ge2$. Then there are rational numbers\footnote{The statement of Theorem 6 in \cite{muller1997finite} does not explicitly
specify that the coefficients $Q_{i}$ are rational - the rationality
is explicit only when $1\le i\le m+3$, in which case concrete formulas
are given. However, the rationality of $Q_{i}$ for all $i$ does
follow from the proof and was verified via personal communication
with the author of \cite{muller1997finite}.} $Q_{t}=Q_{t}\left(G\right)$ for $t=-1/m,-2/m,\ldots$ such that
\begin{equation}
\frac{\left|\Hom\left(G,S_{N}\right)\right|}{\left|\Hom\left(G,S_{N-1}\right)\right|}\asyexp N^{1-1/m}\cdot\left\{ 1+Q_{-1/m}N^{-1/m}+Q_{-2/m}N^{-2/m}+\ldots\right\} .\label{eq:mullers expansion}
\end{equation}
\end{thm}

Müller's result can be translated into a similar asymptotic expansion
for $\mathbb{E}\left[\fix_{\gamma}\left(N\right)\right]$ whenever
$\gamma$ is an element of a finite group (see Section \ref{subsec:sub-coverings-finite-groups}
below). In the current paper we rely on Theorems \ref{thm:MP - rational approx}
and \ref{thm:Muller} in order to generalize these results to arbitrary
free products as in \eqref{eq:Gamma}.
\begin{thm}
\label{thm:asmptotic expansion} Let $\Gamma$ be a free product and
let $m=m\left(\Gamma\right)$ as in \eqref{eq:Gamma}. Then for every
$\gamma\in\Gamma$ there are rational numbers $a_{t}=a_{t}\left(\gamma\right)$
for $t=1,\frac{m-1}{m},\frac{m-2}{m},\ldots,\frac{1}{m},0,-\frac{1}{m},\ldots$
so that 
\begin{equation}
\mathbb{E}\left[\fix_{\gamma}\left(N\right)\right]\asyexp a_{1}N+a_{1-1/m}N^{1-1/m}+a_{1-2/m}N^{1-2/m}+\ldots.\label{eq:asym expansion of E=00005Bfix=00005D}
\end{equation}
\end{thm}

The leading non-vanishing term of \eqref{eq:asym expansion of E=00005Bfix=00005D}
is given by Proposition \ref{prop:exp-torsion-element} and Theorem
\ref{thm:limit expectation of fix}. The value of the second non-zero
term in \eqref{eq:asym expansion of E=00005Bfix=00005D}, or, similarly,
the order of $\mathbb{E}\left[\fix_{\gamma}\left(N\right)\right]-N^{1/\left|\gamma\right|}$,
may encode additional group-theoretic information about $\gamma$:
see Conjecture \ref{conj:pi-conjecture}. 

\subsection{Asymptotic independence of different elements of $\Gamma$}

One may also consider joint local statistics of different elements
in $\Gamma$. We state our result for two elements, although it easily
generalizes to any finite set of elements. Two variables with parameter
$N$ are asymptotically independent if they have a joint limit distribution
as $N\to\infty$ and the limit is that of two independent random variables. 
\begin{thm}
\label{thm:asymptotic independence}Let $\Gamma$ be as in \eqref{eq:Gamma},
let $m=m\left(\Gamma\right)$ and let $\gamma_{1},\gamma_{2}\in\Gamma$
have infinite order. Then the following three conditions are equivalent:
\begin{enumerate}
\item $\fix_{\gamma_{1}}\left(N\right)$ and $\fix_{\gamma_{2}}\left(N\right)$
are asymptotically independent as $N\to\infty$.
\item \label{enu:gamma1 and gamma2 are not conjugate in a EC=00003D0 subgroup}$\gamma_{1}$
and $\gamma_{2}$ are \textbf{not} both conjugate into the same Euler-Characteristic-zero
subgroup of $\Gamma$. 
\item $\mathbb{E}\left[\fix_{\gamma_{1}}\left(N\right)\cdot\fix_{\gamma_{2}}\left(N\right)\right]=\mathbb{E}\left[\fix_{\gamma_{1}}\left(N\right)\right]\cdot\mathbb{E}\left[\fix_{\gamma_{2}}\left(N\right)\right]+O\left(N^{-1/m}\right)$.
\end{enumerate}
\end{thm}

In concrete terms, the condition in item \ref{enu:gamma1 and gamma2 are not conjugate in a EC=00003D0 subgroup}
translates in our settings to that $\left(i\right)$ the non-power
root of $\gamma_{1}$ is not conjugate to the non-power root of $\gamma_{2}$
nor of $\gamma_{2}^{-1}$, and $\left(ii\right)$ $\gamma_{1}$ and
$\gamma_{2}$ do not have conjugates belonging to the same subgroup
isomorphic to $C_{2}*C_{2}$.

\subsection{Statistical asymptotics of cycles of bounded lengths}

The techniques used to study the asymptotic distribution of $\fix_{\gamma}\left(N\right)$,
the number of fixed points in a $\gamma$-random permutation, can
also be used to analyze the asymptotic distribution of $\cyc_{\gamma,L}\left(N\right)$:
the number of $L$-cycles for any fixed $L$ (in particular, $\cyc_{\gamma,1}=\fix_{\gamma}$).
In addition, they lead to the asymptotic joint distribution of $\cyc_{\gamma,1},\cyc_{\gamma,2},\ldots,\cyc_{\gamma,L}$.
To state the results, define, for every $\gamma\in\Gamma$ and $L\in\mathbb{Z}_{\ge1}$,
a set analogous to $\H_{\gamma}$ from \eqref{eq:H_gamma} :
\begin{equation}
\H_{\gamma,L}\defi\left\{ H\le\Gamma\,\middle|\,\begin{gathered}\gamma^{L}\in H,~\chi\left(H\right)=0\\
\forall1\le L'<L~~\gamma^{L'}\notin H
\end{gathered}
\right\} .\label{eq:H_gamma,L}
\end{equation}
So $\H_{\gamma,L}$ is the set of EC-zero subgroups of $\Gamma$ containing
$\gamma^{L}$ but not any smaller positive power of $\gamma$. Note
that $\H_{\gamma,1}=\H_{\gamma}$ and $\H_{\gamma^{L}}=\bigsqcup_{1\le d|L}\H_{\gamma,d}$.
We summarize these results in the following theorem generalizing Theorems
\ref{thm:limit expectation of fix} and \ref{thm:limit distribution of fix}.
\begin{thm}
\label{thm:asymptotic number of cyclces of bounded size}Let $\Gamma$
be as in \eqref{eq:Gamma}, $m=m\left(\Gamma\right)$ and $\gamma\in\Gamma$
have infinite order, and fix $L\in\mathbb{Z}_{\ge1}$. Then
\begin{enumerate}
\item \label{enu:limit of Exp=00005Bcyc_L=00005D}We have
\begin{equation}
\mathbb{E}\left[\cyc_{\gamma,L}\left(N\right)\right]=\frac{1}{L}\left|\H_{\gamma,L}\right|+O\left(N^{-1/m}\right).\label{eq:limit of E=00005Bcyc_L=00005D}
\end{equation}
\item \label{enu:limit of distribution of cyc_L}Let $H_{1},\ldots,H_{t}$
be \emph{representatives of the conjugacy classes of subgroups }represented
in ${\cal H}_{\gamma,L}$. For $i=1,\ldots,t$ let $\alpha_{i}=\left|\left\{ {\cal H}_{\gamma_{,L}}\cap H_{i}^{~\Gamma}\right\} \right|$
be the number of conjugates of $H_{i}$ in ${\cal H}_{\gamma,L}$,
and let $\beta_{i}\defi\left[N_{\Gamma}\left(H_{i}\right):H_{i}\right]$
be the index of $H_{i}$ in its normalizer. Then, 
\[
\cyc_{\gamma,L}\left(N\right)\stackrel[N\to\infty]{\mathrm{dis}}{\longrightarrow}\frac{1}{L}\sum_{i=1}^{t}\alpha_{i}\beta_{i}Z_{1/\beta_{i}},
\]
(as in Theorem \ref{thm:limit distribution of fix}, $Z_{\lambda}\sim\mathrm{Poi}\left(\lambda\right)$,
the random variable $Z_{1/\beta_{1}},\ldots,Z_{1/\beta_{t}}$ are
independent, and ``$\stackrel{\mathrm{dis}}{\longrightarrow}$''
denotes convergence in distribution).
\item \label{enu:cyc_i and cyc_j asymptoticallly independent}The variables
$\cyc_{\gamma,1}\left(N\right),\cyc_{\gamma,2}\left(N\right),\ldots,\cyc_{\gamma,L}\left(N\right)$
are asymptotically independent. In particular, for $L_{1}\ne L_{2}$,
\[
\mathbb{E}\left[\cyc_{\gamma,L_{1}}\left(N\right)\cdot\cyc_{\gamma,L_{2}}\left(N\right)\right]=\mathbb{E}\left[\cyc_{\gamma,L_{1}}\left(N\right)\right]\cdot\mathbb{E}\left[\cyc_{\gamma,L_{2}}\left(N\right)\right]+O\left(N^{-1/m}\right).
\]
\end{enumerate}
\end{thm}

In the case of free groups, parts \ref{enu:limit of Exp=00005Bcyc_L=00005D}
and \ref{enu:limit of distribution of cyc_L} of Theorem \ref{thm:asymptotic number of cyclces of bounded size}
recover the full result of Nica \cite{nica1994number} which determines
the limit distribution of $\cyc_{\gamma,L}\left(N\right)$ when $\gamma$
is an element of a free group and shows the limit depends only on
$q$ where $\gamma=\gamma_{0}^{~q}$ with $\gamma_{0}$ a non-power
as above (see also \cite[Thm.~25]{linial2010word} and \cite[Thm.~1.3]{hanany2020word}).

\subsection{Overview of the paper\label{subsec:Overview}}

\subsubsection*{Outline of the proof of the main results}

Let us explain the ideas behind the proofs of the main results. First
we construct a CW-complex, denoted $X_{\Gamma}$, which is a graph
of spaces (in the sense of Scott and Wall \cite{scott1979topological})
with fundamental group $\Gamma$. The space $X_{\Gamma}$ consists
of a star with a central vertex $o$ and, for every free factor $G_{i}$
of $\Gamma$ in \eqref{eq:Gamma}, an edge $e_{i}$ with one end at
$o$ and the other the basepoint of some pointed CW-complex $X_{G_{i}}$
representing $G_{i}$. For $G$ a f.g.~free group, $X_{G}$ is a
bouquet of circles; for $G=\Lambda_{g}$ a surface group, $X_{G}$
is a pointed, genus-$g$ orientable surface with a given CW-structure
specified below; and for $G$ finite, $X_{G}$ is some finite presentation
$2$-complex of $G$. Clearly, $\pi_{1}\left(X_{\Gamma},o\right)\cong\Gamma$.
See Figure \ref{fig:X_Gamma}. 

Every covering space $p\colon\hat{X}\to X_{\Gamma}$ inherits a CW-structure
from $X_{\Gamma}$. Let $Y\subseteq\hat{X}$ be a sub-complex of $\hat{X}$
with finitely many cells (so if some open cell belongs to $Y$, then
so do all the cells of smaller dimension it is attached to). We call
such a sub-complex a compact sub-cover of $X_{\Gamma}$. It is equipped
with the restriction of the covering map $p=p|_{Y}\colon Y\to X_{\Gamma}$.
The main technical result of this paper is the following.
\begin{quote}
Let $p\colon Y\to X_{\Gamma}$ be a connected compact sub-cover of
$X_{\Gamma}$. Let $\plab\left(Y\right)\defi p_{*}\left(\pi_{1}\left(Y\right)\right)\le\pi_{1}\left(X_{\Gamma}\right)=\Gamma$
be the corresponding conjugacy class of subgroups of $\Gamma$, and
let $\chigrp\left(Y\right)\defi\chi\left(\plab\left(Y\right)\right)$.
Then the average number of \emph{injective lifts} of $Y$ to a random
$N$-cover of $X_{\Gamma}$, denoted $\emb_{Y}\left(N\right)$, satisfies
\begin{equation}
\emb_{Y}\left(N\right)=N^{\chigrp\left(Y\right)}\left(a_{0}\left(Y\right)+O\left(N^{-1/m}\right)\right).\label{eq:E^emb of Y - rough statement of the main technical result}
\end{equation}
\end{quote}
Here $m=m\left(\Gamma\right)$ and $a_{0}\left(Y\right)\in\mathbb{Z}_{\ge1}$
is a positive integer. Moreover, in many important cases $a_{0}\left(Y\right)=1$.
A more precise statement is given in Theorem \ref{thm:E^emb is N^chi and asym expansion}
below and applies to compact sub-covers which are not necessarily
connected. We first prove \eqref{eq:E^emb of Y - rough statement of the main technical result}
for sub-covers of $X_{G}$ for each factor $G$ of $\Gamma$. This
part is straightforward when $G$ is free, it relies on \cite{muller1997finite}
when $G$ is finite, and on \cite{magee2020asymptotic} when $G=\Lambda_{g}$.
We then integrate these results to obtain \eqref{eq:E^emb of Y - rough statement of the main technical result}
for arbitrary sub-covers of $X_{\Gamma}$. 

To analyze $\mathbb{E}\left[\fix_{\gamma}\left(N\right)\right]$ for
some $\gamma\in\Gamma$, recall that $\gamma$ corresponds to some
loop $\overline{\gamma}\colon\left(S^{1},1\right)\to\left(X_{\Gamma},o\right)$,
and we may assume that the image of $\overline{\gamma}$ is a combinatorial
closed path in the $1$-skeleton of $X_{\Gamma}$. Given $\varphi\colon\Gamma\to S_{N}$,
the fixed points of $\varphi\left(\gamma\right)$ are in bijection
with the lifts of $\overline{\gamma}$ to the $N$-sheeted cover $\pi\colon X_{\varphi}\to X_{\Gamma}$
corresponding to $\varphi$. 
\[
\xymatrix{ & X_{\varphi}\ar[d]^{\pi}\\
S^{1}\ar[r]^{\overline{\gamma}}\ar@{-->}[ru]^{\hat{\gamma}} & X_{\Gamma}
}
\]
So we analyze the number of such lifts of $\overline{\gamma}$ into
a random $N$-cover of $X_{\Gamma}$. In every such lift $\text{\ensuremath{\hat{\gamma}}}\colon\left(S^{1},1\right)\to\left(X_{\varphi},y\right)$,
the image $\hat{\gamma}\left(S^{1}\right)$ is a subcomplex of (the
$1$-skeleton of) $X_{\varphi}$ and in particular a sub-cover $Y$
of $X_{\Gamma}$. As $\overline{\gamma}$ is a finite path, there
are finitely many such sub-covers. 
\[
\xymatrix{ & \left(Y,y\right)\ar[d]^{\pi}\\
\left(S^{1},1\right)\ar[r]^{\overline{\gamma}}\ar@{->>}[ru]^{\hat{\gamma}} & \left(X_{\Gamma},o\right)
}
\]
Denote by ${\cal R}_{\gamma}$ the finite set of all such possible
surjective lifts $\hat{\gamma}\colon S^{1}\twoheadrightarrow Y$ to
sub-covers. Such a set is called a resolution in the terminology of
\cite{magee2020asymptotic}. Figure \ref{fig:resolution} illustrates
such a resolution for an element of $C_{2}*C_{4}$. We obtain 
\begin{equation}
\mathbb{E}\left[\fix_{\gamma}\left(N\right)\right]=\sum_{Y\in\R_{\gamma}}\emb_{Y}\left(N\right),\label{eq:E=00005Bfix=00005D =00003D sum E^emn=00005BY=00005D}
\end{equation}
and using \eqref{eq:E^emb of Y - rough statement of the main technical result}
deduce that
\[
\mathbb{E}\left[\fix_{\gamma}\left(N\right)\right]=\sum_{Y\in{\cal R}_{\gamma}}N^{\chigrp\left(Y\right)}\left(a_{0}\left(Y\right)+O\left(N^{-1/m}\right)\right).
\]
It is clear that for every $Y\in{\cal R}_{\gamma}$, we have $\gamma\in\plab\left(Y,y\right)$.
It is not hard to show that if $\left|\gamma\right|=\infty$, then
the subgroups in ${\cal H}_{\gamma}$ are precisely the subgroups
$\plab\left(Y,y\right)$ for $Y\in{\cal R}_{\gamma}$ with $\chi\left(\plab\left(Y\right)\right)=0$.
Moreover, in all these elements of the resolution, $a_{0}\left(Y\right)=1$.
This leads to Theorem \ref{thm:limit expectation of fix}.

Theorem \ref{thm:asmptotic expansion} about the asymptotic expansion
of $\mathbb{E}\left[\fix_{\gamma}\left(N\right)\right]$ is proven
along the same lines: we first prove asymptotic expansion of $\emb_{Y}\left(N\right)$
for sub-covers of $X_{G}$ separately for every factor $G$ of $\Gamma$
(again, heavily relying on \cite{muller1997finite,magee2020asymptotic}),
then establish this expansion for arbitrary sub-covers of $X_{\Gamma}$,
and finally use \eqref{eq:E=00005Bfix=00005D =00003D sum E^emn=00005BY=00005D}
to establish the sought-after result of Theorem \ref{thm:asmptotic expansion}.

The proofs of the remaining results use similar techniques combined
with the method of moments. In particular, to establish Theorem \ref{thm:limit distribution of fix}
about the limit distribution of $\fix_{\gamma}\left(N\right)$, we
study the moments $\mathbb{E}\left[\fix_{\gamma}\left(N\right)^{r}\right]$
for every $r\in\mathbb{Z}_{\ge1}$ by constructing a resolution for
the union of $r$ disjoint copies of $\overline{\gamma}$. Every element
in this resolution corresponds to a finite multiset of f.g.~subgroups
of $\Gamma$. The limit $\lim_{N\to\infty}\mathbb{E}\left[\fix_{\gamma}\left(N\right)^{r}\right]$
is given by the number of elements in this resolution corresponding
to multisets of subgroup with total EC zero. 

\subsubsection*{Paper organization }

After mentioning some related works in Section \ref{subsec:Related-works},
we formally construct the graph of spaces $X_{\Gamma}$ and introduce
the notions of sub-covers and resolutions in Section \ref{sec:The-space-X_Gamma its covers and sub-covers}.
Section \ref{sec:Sub-coverings of X_G} studies sub-covers of a vertex-space
$X_{G}$ of $X_{\Gamma}$, and analyzes them separately for every
type of group $G$: free groups, finite groups, and surface groups.
In particular, it proves our main technical result, Theorem \ref{thm:E^emb is N^chi and asym expansion},
for all such sub-covers, and proves Proposition \ref{prop:exp-torsion-element}
concerning torsion elements of $\Gamma$. Then, Section \ref{sec:arbitrary sub-coverings}
incorporates the results from Section \ref{sec:Sub-coverings of X_G}
to prove Theorem \ref{thm:E^emb is N^chi and asym expansion} for
arbitrary sub-covers. In Section \ref{sec:limit-distr-and-asymptotic-expansion of fix}
we complete the proof of Theorems \ref{thm:limit expectation of fix},
\ref{thm:limit distribution of fix} and \ref{thm:asmptotic expansion},
and in Section \ref{sec:Asymptotic-independence-and-small-cycles}
of Theorems \ref{thm:asymptotic independence} and \ref{thm:asymptotic number of cyclces of bounded size}.
We end in Section \ref{sec:Open-questions} with two intriguing open
questions which arise from our results.

\subsubsection*{Notation}

We denote by $\left(N\right)_{t}$ the falling factorial, also known
as Pochhammer symbol,
\[
\left(N\right)_{t}\defi N\left(N-1\right)\cdots\left(N-t+1\right).
\]
We write $\left[N\right]$ for the set $\left\{ 1,\ldots,N\right\} $.
The notation $\asyexp$ marks asymptotic expansion and is defined
in Definition \ref{def:asymptotic expansion}. Many repeating notions
are formally defined in Section \ref{sec:The-space-X_Gamma its covers and sub-covers},
such as sub-covers (Definition \ref{def:sub-covering}), resolutions
(Definition \ref{def:resolution}), and $\mathbb{E}_{Y}\left(N\right)$
and $\emb_{Y}\left(N\right)$ (Definition \ref{def:E_Y and E^emb_Y}).

\subsection{Related works \label{subsec:Related-works}}

\subsubsection*{The number of homomorphisms $\Gamma\to S_{N}$}

As mentioned above, the size of $\Hom\left(\Gamma,S_{N}\right)$ is
well understood. Clearly, $\left|\Hom\left(\Gamma,S_{N}\right)\right|=\prod_{i=1}^{k}\left|\Hom\left(G_{i},S_{N}\right)\right|$.
If $G=\F_{r}$ is a rank-$r$ free group, then $\left|\Hom\left(\F_{r},S_{N}\right)\right|=\left|S_{N}\right|^{r}=\left(N!\right)^{r}$.
If $G$ is finite, an asymptotic formula for $\left|\Hom\left(G,S_{N}\right)\right|$
is given in \cite[Thm.~5]{muller1997finite}. Finally, if $G=\Lambda_{g}$
is a genus-$g$ surface group, then $\left|\Hom\left(G,S_{N}\right)\right|$
is related to the ``zeta function'' of irreducible representations
of $S_{N}$ and is equal to $\left|S_{N}\right|^{2g-1}\left(2+O\left(N^{2-2g}\right)\right)$
(see \cite{hurwitz1902ueber,Lulov96,LiebeckShalev}). 

\subsubsection*{Random Belyi surfaces}

In \cite{gamburd2006poisson}, Gamburd studies random Belyi surfaces
glued from $N$ ideal hyperbolic triangles. Closed paths in these
surfaces are related to closed paths in the dual graph, which is cubical.
This cubical graph is completely determined by the cyclic order of
half-edges around every vertex, so a permutation $\beta\in S_{3N}$
consisting of $N$ $3$-cycles, and the perfect matching of the half
edges which creates the edges, so a permutation $\alpha\in S_{3N}$
consisting of $1.5N$ transpositions. Gamburd analyzes the random
permutation $\alpha\beta$ when $\alpha$ and $\beta$ are chosen
uniformly at random, and proves it converges to the uniform distribution
in total variation distance as $N\to\infty$. In fact, he proves the
same when $3$ is replaced by an arbitrary $k\in\mathbb{N}$ \cite[Thm.~4.1]{gamburd2006poisson}.
Although similar, note that this model is different than ours, as
$\alpha$ and $\beta$ are not allowed to have fixed points.

\subsubsection*{Measures induced by elements of finitely generated groups}

There has been an extensive study of ``word measures'' -- measures
induced by elements of free groups -- on various families of groups.
As mentioned above, the asymptotics of word measures on $S_{N}$ were
studied in \cite{nica1994number}. A precise result about the leading
term of $\mathbb{E}\left[\fix_{\gamma}\left(N\right)\right]-1$ (for
$\gamma\in\F$) was found in \cite{puder2014primitive,PP15} -- see
Section \ref{sec:Open-questions}, and more general results about
all stable characters of $S_{N}$ were more recently established in
\cite{hanany2020word}. Additional works studied word measures on
$U\left(N\right)$, $O\left(N\right)$, $Sp\left(N\right)$, $\mathrm{GL}_{N}\left(\mathbb{F}_{q}\right)$
and generalized symmetric groups -- see \cite[Sec.~1.6]{hanany2020word}
for a short survey. As for measures induced by elements of surface
groups, aside for the above mentioned work \cite{magee2020asymptotic},
the recent works \cite{magee2021surface-unitary1,magee2021surface-unitary2}
study measures induced by elements of $\Lambda_{g}$ on $U\left(N\right)$
and establish results about the expected trace: its limit and its
asymptotic expansion. Finally, Baker and Petri show in \cite{BakerPetri}
how one can use results as in the current paper about measures induced
on $S_{N}$ by elements of the free product $C_{p_{1}}*\ldots*C_{p_{k}}$
of cyclic group, in order to study such measures induced on $S_{N}$
by elements of the group
\[
\Gamma_{p_{1},\ldots,p_{k}}\defi\left\langle x_{1},\ldots,x_{m}\,\middle|\,x_{1}^{p_{1}}=x_{2}^{p_{2}}=\ldots=x_{k}^{p_{k}}\right\rangle .
\]
(Note that $C_{p_{1}}*\ldots*C_{p_{k}}$ is a quotient with kernel
$\mathbb{Z}$ of $\Gamma_{p_{1},\ldots,p_{k}}$ by the additional
relation $x_{1}^{p_{1}}=1$.) See also \cite{hanany2020some} for
a general discussion regarding ``profinitely rigid'' elements in
finitely generated groups and the relation to measures induced by
such elements on finite groups.

\subsubsection*{Spectral gap}

Some of the works in this line of research are motivated, inter alia,
by questions about expansion and spectral gap of random objects. For
$r\in\mathbb{Z}_{\ge2}$, a random $2r$-regular graph on $N$ vertices
can be obtained as a random $N$-cover of the bouquet of one vertex
and $r$ loops, thus corresponding to a random homomorphism $\F_{r}\to S_{N}$.
Relying on the trace method, word measures on $S_{N}$ can thus be
used to show that random graphs are expanders. Indeed, this is the
arranging idea in many works on the subject, starting from \cite{broder1987second}
(and also in \cite{friedman2008proof,puder2015expansion}). This also
stands in the background for works about expansion of more general
random Schreier graphs of $S_{N}$ \cite{friedman1998action,hanany2020word}.
Similarly, using Selberg's trace formula, the results of \cite{magee2020asymptotic}
about measures induced on $S_{N}$ by elements of $\Lambda_{g}$,
the fundamental group of a hyperbolic surface, were used in \cite{magee2020random}
to yield results about spectral gap in random covers of closed hyperbolic
surfaces. Results about $\Lambda_{g}$ stated in the current paper
are used in the recent paper \cite{naud2022random}, studying other
statistics of the spectrum of the Laplacian on random covers of surfaces.
Simpler techniques can also show that these random objects Benjamini-Schramm
converge to the corresponding universal cover: see \cite[Sec.~1.5]{magee2020asymptotic}
for random surfaces and \cite{BakerPetri} for random covers of torus-knot
complements.

\subsubsection*{Core graphs and sub-covers}

The notion of a sub-cover in this paper is very much related to Stallings
core graphs. The original Stallings core graphs \cite{stallings1983topology}
correspond to subgroups of a free group and they are only slightly
generalized in sub-covers of $X_{G}$ where $G$ is a free group.
Bass \cite{bass1993covering} extends Stallings' theory to a very
general theory about geometric presentation of subgroups of the fundamental
group of a graph of groups. Other authors developed a more specialized
version for more specialized families of groups such as amalgams of
finite groups \cite{markus2007stallings} or the mere modular group
$\mathrm{PSL}_{2}\left(\mathbb{Z}\right)\cong C_{2}*C_{3}$ \cite{bassino2021statistics}.
Some of the analysis in this paper is inspired by ideas in these works.
We heavily rely here also on a theory of core surfaces developed for
surface groups in \cite{MPcore,magee2020asymptotic}.

\subsection*{Acknowledgements}

We thank Nir Avni, Michael Magee, Chen Meiri, Thomas Müller and Ron
Peled for beneficial discussions. Michal Buran carried out the computation
leading to an estimate of $\mathbb{E}\left[\fix_{\left[a,b\right]}\left(N\right)\right]$
where $\left[a,b\right]\in\Lambda_{2}=\left\langle a,b,c,d\,\middle|\,\left[a,b\right]\left[c,d\right]\right\rangle $,
an estimate used in Section \ref{sec:Open-questions}. This project
has received funding from the European Research Council (ERC) under
the European Union’s Horizon 2020 research and innovation programme
(grant agreement No 850956).

\section{The space $X_{\Gamma}$, its covers and sub-covers\label{sec:The-space-X_Gamma its covers and sub-covers}}

\subsection{The graph of spaces representing $\Gamma$ \label{subsec:The-graph-of-spaces}}

Let us formally construct the space $X_{\Gamma}$ mentioned in Section
\ref{subsec:Overview}. If $\left(X,x_{0}\right)$ and $\left(Y,y_{0}\right)$
are two pointed connected CW-complexes (or nice enough topological
spaces), Seifert-Van-Kampen Theorem guarantees that their wedge sum
at the points $x_{0}$ and $y_{0}$ has fundamental group isomorphic
to the free product $\pi_{1}\left(X,x_{0}\right)*\pi_{1}\left(Y,y_{0}\right)$.
Relying on this basic fact, we construct a CW-complex whose fundamental
group is $\Gamma$ from smaller CW-complexes with fundamental groups
$G_{1},\ldots,G_{k}$ (we continue using here the notation from Assumption
\ref{assu:Gamma}). To get a clearer picture which is somewhat easier
to work with, we use a star-graph instead of a single wedge point.
This leads to the following construction of $X_{\Gamma}$ as a graph
of spaces. 

For every $i=1,\ldots,k$, let $X_{G_{i}}$ be a CW-complex with a
marked vertex (in fact, a single vertex) $v_{i}$, so that $\pi_{1}\left(X_{G_{i}},v_{i}\right)\cong G_{i}$.
Moreover, the edges ($1$-cells) of $X_{G_{i}}$ are directed and
labeled by a fixed set of generators of $G_{i}$. The construction
of $X_{G_{i}}$ is as follows:
\begin{itemize}
\item If $G_{i}=\F_{r}$ is a rank-$r$ free group, we fix a basis $B=\left\{ b_{1},\ldots,b_{r}\right\} $
and let $X_{G_{i}}$ be a bouquet made of one vertex, (named $v_{i}$)
and $r$ directed loops labeled\footnote{To make the notation complete, here and for the other types of factors,
one needs the notation of the generators to formally reflect the factor
they generate, for example $b_{1}^{i},\ldots,b_{r}^{i}$, as there
may be more than one factor which is a free group (or more than one
factor which is a surface group and so on). However, we prefer to
keep the notation a bit simpler and have the factor be understood
from the context. } $b_{1},\ldots,b_{r}$. This defines an isomorphism $\pi_{1}\left(X_{G_{i}},v_{i}\right)\cong\F_{r}$.
\item If $G_{i}$ is a finite group, we let $X_{G_{i}}$ be some finite
presentation complex of $G_{i}$: given some finite presentation $\left\langle S\,\middle|\,R\right\rangle $
of $G_{i}$, the complex $X_{G_{i}}$ is made of the vertex $v_{i}$
together with a directed loop for every $s\in S$ and a $2$-cell
attached to the loops for every $r\in R$. We fix some isomorphism
$\left\langle S\,\middle|\,R\right\rangle \cong G_{i}$ and for every
$s\in S$, label the $s$-loop by its image in $G_{i}$ through this
isomorphism. For example, if $G_{i}=C_{q}=\left\langle a\right\rangle $
is a cyclic group, then $G_{i}\cong\left\langle s\,\middle|\,s^{q}\right\rangle $
and so $X_{G_{i}}$ may consist of a single vertex, a single directed
loop labeled '$a$', and a single $2$-cell whose boundary wraps around
the $a$-loop $q$ times.
\item If $G_{i}=\Lambda_{g}$ is a surface group, we let $X_{G_{i}}$ be
a genus-$g$ orientable surface with a CW-structure obtained from
gluing the sides of a $4g$-gon according to the pattern\linebreak{}
$a_{1},b_{1},a_{1}^{-1},b_{1}^{-1},\ldots,a_{g},b_{g},a_{g}^{-1},b_{g}^{-1}$.
So $X_{G_{i}}$ consists of a single vertex $v_{i}$, $2g$ directed
loops labeled $a_{1},b_{1},\ldots,a_{g},b_{g}$ and a single $2$-cell.
The elements $a_{1},\ldots,b_{g}$ are elements of $\Lambda_{g}$
given by an isomorphism $\Lambda_{g}\cong\left\langle a_{1},\ldots,b_{g}\,\middle|\,\left[a_{1},b_{1}\right]\cdots\left[a_{g},b_{g}\right]\right\rangle $.
\end{itemize}
Finally, $X_{\Gamma}$ consists of the complexes $X_{G_{1}},\ldots,X_{G_{k}}$
together with a vertex $o$ and edges $e_{1},\ldots,e_{k}$ so that
$e_{i}$ connects $o$ and $v_{i}$. We have $\pi_{1}\left(X_{\Gamma},o\right)\cong\Gamma$.
Moreover, the labels of the $1$-cells inside $X_{G_{1}},\ldots,X_{G_{k}}$
form a generating set for $\Gamma$, and every word in this generating
set corresponds to a closed loop in $X_{\Gamma}^{\left(1\right)}$,
the $1$-skeleton of $X_{\Gamma}$, based at $o$: simply follow the
$1$-cells according to the given word (transversing the corresponding
edge backwards if the generator comes with a negative exponent), and
going through $o$ and $e_{1},\ldots,e_{k}$ to pass from one $X_{G_{i}}$
to another. This is illustrated in Figure \ref{fig:X_Gamma}.

\begin{figure}
\begin{centering}
\includegraphics[viewport=0bp 0bp 407.907bp 324.071bp,scale=0.7]{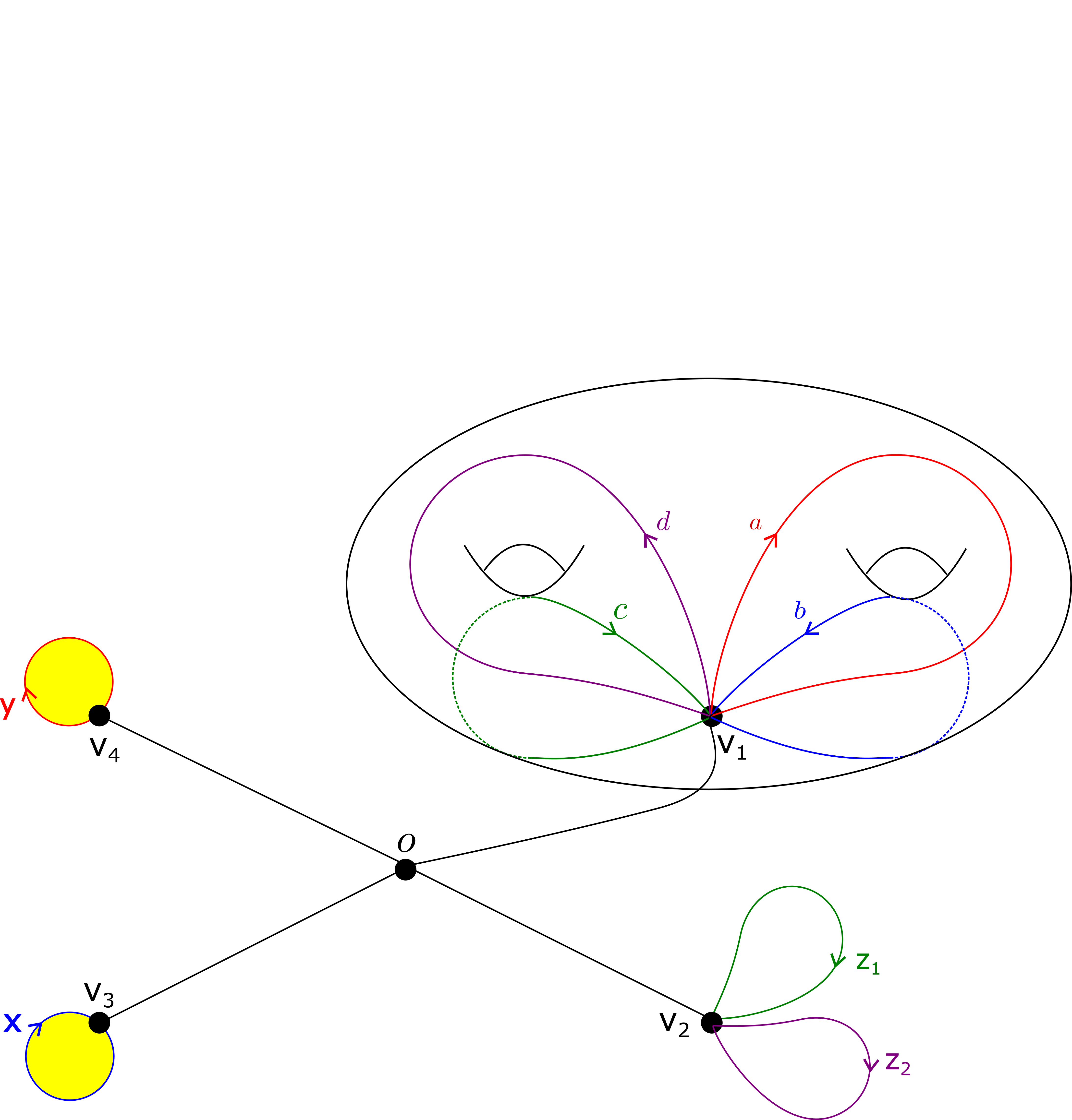}
\par\end{centering}
\caption{\label{fig:X_Gamma}The graph of spaces $X_{\Gamma}$ when $\Gamma=\Lambda_{2}*\protect\F_{2}*C_{2}*C_{4}$.
The middle vertex $o$ is connected by four edges to the four basepoints
of the spaces $X_{\Lambda_{2}}$, $X_{\protect\F_{2}}$, $X_{C_{4}}$
and $X_{C_{2}}$. The space $X_{\Lambda_{2}}$ is a genus-2 surface
consisting of the vertex $v_{1}$, four edges labeled $a,b,c$ and
$d$ and one disc. The space $X_{\protect\F_{2}}$ consists of one
vertex $v_{2}$ and two loops labeled $z_{1}$ and $z_{2}$. The spaces
$X_{C_{2}}$ and $X_{C_{4}}$ both consist of a single vertex ($v_{3}$
and $v_{4}$, respectively), a single edge (labeled $x$ and $y$,
respectively), and a single $2$-cell: the boundary of the $2$-cell
wraps around the edge twice (respectively, four times). Note that
the latter part of the construction is not well reflected in the figure.
Overall, $\left\{ a,b,c,d,z_{1},z_{2},x,y\right\} $ is a generating
set for $\Gamma$.}
\end{figure}

\subsection{Sub-covers and resolutions}

Covering spaces of $\text{X}_{\Gamma}$ inherit the $CW$-structure
from it (indeed, open cells are covered by disjoint homeomorphic sets),
and we consider $N$-sheeted covering spaces $p\colon\hat{X}\to X_{\Gamma}$
together with a bijection between $p^{-1}\left(o\right)$ and $\left\{ 1,\ldots,N\right\} $.
This yields a bijection between these $N$-covers and $\Hom\left(\Gamma,S_{N}\right)$
(see, for instance, \cite[pp.~68-70]{hatcher2005algebraic}). For
$\varphi\colon\Gamma\to S_{N}$, we denote the corresponding $N$-cover
of $X_{\Gamma}$ by $X_{\varphi}$.
\begin{defn}[Sub-covers]
\label{def:sub-covering} A \textbf{sub-cover} $Y$ of $X_{\Gamma}$
is a (not necessarily connected) sub-complex of a (not necessarily
finite degree) covering space of $X_{\Gamma}$. In particular, a sub-cover
is endowed with the restricted covering map $p\colon Y\to X_{\Gamma}$,
which is an immersion. Denote by $Y|_{G_{i}}\defi p^{-1}\left(X_{G_{i}}\right)$,
$i=1,\ldots,k$, the subcomplex of $Y$ lying above $X_{G_{i}}$,
the subspace of $X_{\Gamma}$ corresponding to the factor $G_{i}$
of $\Gamma$. Let $p_{i}\colon Y_{i}\to X_{\Gamma}$ be sub-covers
for $i=1,2$. A morphism of sub-covers $f\colon Y_{1}\to Y_{2}$ is
a combinatorial morphism of CW-complexes commuting with the restricted
covering maps, namely, such that the following diagram commutes.
\[
\xymatrix{Y_{1}\ar[rr]^{f}\ar[rd]_{p_{1}} &  & Y_{2}\ar[ld]^{p_{2}}\\
 & X_{\Gamma}
}
\]
\end{defn}

This definition extends the definition of a \emph{tiled surface} in
the case where $\Gamma=\Lambda_{g}$ \cite[Def.~3.1]{MPcore}. The
covering space of which $Y$ is a sub-complex is \emph{not} part of
the data attached to the sub-cover: indeed, the same $Y$ can be a
sub-cover of distinct covering spaces. In the following definition
we adapt the terminology from \cite{magee2020asymptotic}, extend
it to our more general setting and add a variant restricted to embeddings.

\begin{defn}[Resolutions]
\label{def:resolution} Let $p\colon Y\to X_{\Gamma}$ be a sub-cover
of $X_{\Gamma}$. A \textbf{resolution} $\R$ of $Y$ is a collection
of morphisms of sub-covers 
\[
\left\{ f\colon Y\to Z_{f}\right\} 
\]
so that every morphism of sub-covers $h\colon Y\to\hat{X}$ to a \emph{full}
covering space $\hat{X}$ of $X_{\Gamma}$ decomposes \uline{uniquely}
as 
\[
Y\stackrel{f}{\to}Z_{f}\stackrel{\overline{h}}{\hookrightarrow}\hat{X},
\]
with $f\in\R$ and $\overline{h}$ an \uline{embedding}.

Similarly, an \textbf{embedding-resolution} $\remb$ of $Y$ is a
collection of \emph{injective }morphisms of sub-covers $\left\{ f\colon Y\hookrightarrow Z_{f}\right\} $
so that every \emph{injective }morphism of sub-covers $h\colon Y\hookrightarrow\hat{X}$
to a \emph{full} covering space $\pi\colon\hat{X}\to X_{\Gamma}$
of $X_{\Gamma}$ decomposes \uline{uniquely} as $Y\stackrel{f}{\hookrightarrow}Z_{f}\stackrel{\overline{h}}{\hookrightarrow}\hat{X}$
with $f\in\remb$ and $\overline{h}$ an \uline{embedding}.
\end{defn}

Many of the resolutions we construct in this paper are the natural
resolutions which consist of all possible \emph{surjective }morphisms
of sub-covers with domain $Y$. When $Y$ is compact, this resolution
is finite. However, we sometimes need more involved resolutions. The
identity map $\left\{ \id\colon Y\to Y\right\} $ constitutes a trivial
embedding-resolution. But again, we will need below more involved
embedding-resolutions.
\begin{defn}[$\mathbb{E}_{Y}$ and $\emb_{Y}$]
\label{def:E_Y and E^emb_Y} Let $p\colon Y\to X_{\Gamma}$ be a
sub-cover of $X_{\Gamma}$, and let $\pi\colon X_{\varphi}\to X_{\Gamma}$
be an $N$-cover of $X_{\Gamma}$ corresponding to a uniformly random
$\varphi\colon\Gamma\to S_{N}$. Denote by $\mathbb{E}_{Y}\left(N\right)$
the expected number of lifts of $p$ to $X_{\varphi}$.
\[
\xymatrix{ & X_{\varphi}\ar[d]^{\pi}\\
Y\ar[r]^{p}\ar@{-->}[ru]^{\mathbb{E}\left[\#f\right]} & X_{\Gamma}
}
\]
Namely
\begin{equation}
\mathbb{E}_{Y}\left(N\right)\defi\mathbb{E}_{\varphi\in\Hom\left(\Gamma,S_{N}\right)}\left|\left\{ f\colon Y\to X_{\varphi}\,\middle|\,\pi\circ f=p\right\} \right|.\label{eq:def of E_Y}
\end{equation}
Similarly, denote by $\emb_{Y}\left(N\right)$ the expected number
of \emph{injective }lifts of $p$ to the random $N$-cover $X_{\varphi}$.
\[
\xymatrix{ & X_{\varphi}\ar[d]^{\pi}\\
Y\ar[r]^{p}\ar@{^{(}-->}[ru]^{\mathbb{E}\left[\#f\right]} & X_{\Gamma}
}
\]
Namely
\begin{equation}
\emb_{Y}\left(N\right)\defi\mathbb{E}_{\varphi\in\Hom\left(\Gamma,S_{N}\right)}\left|\left\{ f\colon Y\hookrightarrow X_{\varphi}\,\middle|\,\pi\circ f=p,~f~\mathrm{is~injective}\right\} \right|.\label{eq:def of E^emb_Y}
\end{equation}
\end{defn}

As reflected in Theorem \ref{thm:E^emb is N^chi and asym expansion}
below, the quantity $\emb_{Y}\left(N\right)$ has nice properties.
In contrast, the quantity $\mathbb{E}_{Y}\left(N\right)$ does not
share these properties, and we therefore study it via resolutions
together with the following obvious lemma. 
\begin{lem}
\label{lem:resolution-sum-of-probabilities}Let $Y$ be a compact
sub-cover of $X_{\Gamma}$. If $\R$ is a finite resolution of $Y$,
then 
\begin{equation}
\mathbb{E}_{Y}\left(N\right)=\sum_{f\in\mathcal{R}}\emb_{Z_{f}}\left(N\right).\label{eq:resolution}
\end{equation}
If $\remb$ is a finite embedding-resolution of $Y$ then 
\begin{equation}
\emb_{Y}\left(N\right)=\sum_{f\in\remb}\emb_{Z_{f}}\left(N\right).\label{eq:emb-resolution}
\end{equation}
\end{lem}

\begin{defn}[Subgroups associated with sub-covers]
\label{def:subgroups corresponding to a sub-covering} Let $p\colon Y\to X_{\Gamma}$
be a compact\footnote{Some of the notions here can be defined for arbitrary, not-necessarily-compact
sub-covers, but we only use them for compact ones.} sub-cover. If $Y$ is \emph{connected}, then $p_{*}\left(\pi_{1}\left(Y\right)\right)\le\pi_{1}\left(X_{\Gamma}\right)=\Gamma$
is a well-defined conjugacy class of f.g.~subgroups of $\Gamma$
we denote by\footnote{The notation $\plab\left(Y\right)$ is meant to hint that we consider
closed cycles in the $1$-skeleton of $Y$ (based at some vertex)
as \emph{labeled} cycles: every cycle represents the element of $\Gamma$
which is spelled by labels on the edges along the cycle. } $\plab\left(Y\right)$. If $y\in p^{-1}\left(o\right)\subseteq Y$
then $\plab\left(Y,y\right)$ is the corresponding particular subgroup
in the conjugacy class $\plab\left(Y\right)$. If $Y$ is not necessarily
connected, let $Y_{1},\ldots,Y_{\ell}$ denote its connected components,
and define
\[
\plab\left(Y\right)\defi\left\{ \plab\left(Y_{1}\right),\ldots,\plab\left(Y_{\ell}\right)\right\} 
\]
to be the \emph{multiset} of conjugacy classes of f.g.~subgroups
of $\Gamma$ corresponding to $Y_{1},\ldots,Y_{\ell}$. Finally, denote
by 
\[
\chigrp\left(Y\right)\defi\chi\left(\plab\left(Y\right)\right)\defi\sum_{i=1}^{\ell}\chi\left(\plab\left(Y_{i}\right)\right)
\]
the sum of Euler characteristics of the subgroups in the multiset. 
\end{defn}

The following theorem is at the heart of this paper.
\begin{thm}
\label{thm:E^emb is N^chi and asym expansion}Let $p\colon Y\to X_{\Gamma}$
be a compact, not necessarily connected, sub-cover of $X_{\Gamma}$.
Let $m=m\left(\Gamma\right)$. Then there are rational numbers $a_{t}=a_{t}\left(Y\right)$
for $t=0,-\frac{1}{m},-\frac{2}{m},-\frac{3}{m},\ldots$ so that 
\begin{equation}
\emb_{Y}\left(N\right)\asyexp N^{\chigrp\left(Y\right)}\cdot\left\{ a_{0}+a_{-1/m}N^{-1/m}+a_{-2/m}N^{-2/m}+\ldots\right\} .\label{eq:asym expansion of E^emb}
\end{equation}
Moreover, $a_{0}$ is a positive integer, so in particular,
\[
\emb_{Y}\left(N\right)=N^{\chigrp\left(Y\right)}\left(a_{0}+O\left(N^{-1/m}\right)\right)=\Theta\left(N^{\chigrp\left(Y\right)}\right).
\]
Furthermore, whenever there are no surface groups involved, $a_{0}=1$,
so
\[
\emb_{Y}\left(N\right)=N^{\chigrp\left(Y\right)}\left(1+O\left(N^{-1/m}\right)\right).
\]
\end{thm}

More details about the value of $a_{0}$ are given in Proposition
\ref{prop:addendum - value of a0}. It is a product of positive integers
determined by the sub-complexes $Y|_{G_{i}}$ of $Y$ lying above
$X_{G_{i}}$ when $G_{i}$ is a surface group. In Section \ref{sec:Sub-coverings of X_G}
we prove Theorem \ref{thm:E^emb is N^chi and asym expansion} for
a compact sub-cover of $X_{G}$ for any factor $G$ of $\Gamma$.
In Section \ref{sec:arbitrary sub-coverings} we complete the proof
of Theorem \ref{thm:E^emb is N^chi and asym expansion} for arbitrary
compact sub-covers of $X_{\Gamma}$. 
\begin{rem}
\label{rem:point of confusion}There is a subtle issue in the central
notion of $\emb_{Y}\left(N\right)$. In case $Y$ has non-trivial
automorphisms, there may be different injective lifts with the same
image in the $N$-cover. We count them separately. To illustrate,
consider the somewhat degenerate case that $\Gamma=C_{2}=\left\langle x\right\rangle $,
$X_{\Gamma}$ is the presentation complex of $\left\langle x\,\middle|\,x^{2}\right\rangle $,
and $Y$ is the graph $\xymatrix@1{\bullet\ar@/^{0.7pc}/[r]_{x}&\bullet\ar@/^{0.7pc}/[l]_x}$.
Every $N$-cover of $X_{\Gamma}$ consists of connected components
corresponding to the trivial subgroup (copies of $Y$ together with
two discs attached, so 2-spheres), and connected components corresponding
to $\Gamma$ (copies of $X_{\Gamma}$). By Theorem \ref{thm:Muller},
a random $N$-cover has in expectation $\sqrt{N}+O\left(1\right)$
copies of $X_{\Gamma}$, and thus $\frac{N-\sqrt{N}}{2}+O\left(1\right)$
copies of the 2-sphere, and thus $\frac{N-\sqrt{N}}{2}+O\left(1\right)$
\emph{disjoint} embeddings of $Y$. However, there are two different
embeddings of $Y$ in every copy of the $2$-sphere, and so $\emb_{Y}\left(N\right)=N-\sqrt{N}+O\left(1\right)$.
This agrees with Theorem \ref{thm:E^emb is N^chi and asym expansion},
as $\plab\left(Y\right)=\left\{ 1\right\} $ and so $\chigrp\left(Y\right)=1$.
See also Remark \ref{rem:Muller - Schlage-Puchta}.
\end{rem}

Below we will repeatedly use the following standard fact from the
theory of covering spaces (e.g., \cite[Prop.~1.33 and 1.34]{hatcher2005algebraic}).
\begin{prop}[Lifting criterion ]
\label{prop:lifting criterion} Let $p\colon(\hat{X},\hat{x}_{0})\to\left(X,x_{o}\right)$
be a covering map and $f\colon\left(Y,y_{0}\right)\to\left(X,x_{0}\right)$
a map with $Y$ path-connected and locally path-connected. Then a
lift $\hat{f}\colon\left(Y,y_{0}\right)\to(\hat{X},\hat{x}_{0})$
of $f$ exists if and only if $f_{*}\left(\pi_{1}\left(Y,y_{0}\right)\right)\le p_{*}(\pi_{1}(\hat{X},\hat{x}_{0}))$.
In this case, the lift is unique. 
\end{prop}

Throughout the proof of Theorem \ref{thm:E^emb is N^chi and asym expansion},
we need the following construction. Let $p\colon Y\to X_{\Gamma}$
be a \emph{connected} compact sub-cover. Choose an arbitrary vertex
$y\in Y$. The fundamental group $\pi_{1}\left(X_{\Gamma},p\left(y\right)\right)$
(isomorphic to $\Gamma$, of course) acts on the universal cover $(\widetilde{X_{\Gamma}},\tilde{x}_{0})\to\left(X_{\Gamma},p\left(y\right)\right)$
by deck transformations. Let $\left(\Upsilon,u\right)\defi\plab\left(Y,y\right)\backslash(\widetilde{X_{\Gamma}},\tilde{x}_{0})$
be the quotient of $\widetilde{X_{\Gamma}}$ by the action of the
subgroup $\plab\left(Y,y\right)$. So $\left(\Upsilon,u\right)$ is
the covering space of $\left(X_{\Gamma},p\left(y\right)\right)$ corresponding
to the subgroup $\plab\left(Y,y\right)$, and $\pi_{1}\left(\Upsilon,u\right)\cong\plab\left(\Upsilon,u\right)=\plab\left(Y,y\right)$
(see \cite[Thm.~1.38]{hatcher2005algebraic}). By Proposition \ref{prop:lifting criterion},
there exists a unique lift $\hat{p}\colon\left(Y,y\right)\to\left(\Upsilon,u\right)$
of $p$.
\begin{defn}[Universal lift]
\label{def:universal lift} Let $p\colon Y\to X_{\Gamma}$ be a compact
sub-cover. If $Y$ is connected, the lift $\hat{p}\colon Y\to\Upsilon=\plab\left(Y\right)\backslash\widetilde{X_{\Gamma}}$
of $p$ is called the\textbf{ universal lift of $Y$}. If $Y$ is
not necessarily connected, let $Y_{1},\ldots,Y_{\ell}$ be its connected
components. The \textbf{universal lift} of $Y$ is the map 
\[
\hat{p}\colon Y\to\plab\left(Y_{1}\right)\backslash\widetilde{X_{\Gamma}}\sqcup\ldots\sqcup\plab\left(Y_{\ell}\right)\backslash\widetilde{X_{\Gamma}}
\]
which maps every connected component of $Y$ to its own connected
cover of $X_{\Gamma}$.
\end{defn}

\begin{lem}
\label{lem:universal lift is injective}The universal lift of any
sub-cover is injective.
\end{lem}

\begin{proof}
It is enough to show injectivity for every connected component of
$Y$ separately (as each is mapped to a different connected component
of the codomain). So we may assume that $Y$ is connected. Let $\hat{p}\colon Y\to\Upsilon$
be its universal lift. By the definition of a sub-cover, there is
an embedding $f\colon Y\hookrightarrow Z$ into a (full) covering
map $\pi\colon Z\to X_{\Gamma}$. Choose an arbitrary vertex $y\in Y$.
Then $\plab\left(\Upsilon,\hat{p}\left(y\right)\right)=\plab\left(Y,y\right)\le\plab\left(Z,f\left(y\right)\right)$,
and so there is a map $\phi\colon\left(\Upsilon,\hat{p}\left(y\right)\right)\to\left(Z,f\left(y\right)\right)$.
The uniqueness of lifts guarantees that $f=\phi\cdot\hat{p},$ and
we get the following commuting diagram. 
\[
\xymatrix{\left(\Upsilon,\hat{p}\left(y\right)\right)\ar[rr]^{\phi} &  & \left(Z,f\left(y\right)\right)\ar[d]^{\pi}\\
\left(Y,y\right)\ar[rr]^{p}\ar@{^{(}->}[rru]^{f}\ar[u]^{\hat{p}} &  & \left(X_{\Gamma},p\left(y\right)\right)
}
\]
The injectivity of $f$ now implies the one of $\hat{p}$.
\end{proof}

\section{Sub-covers of the factors of $\Gamma$\label{sec:Sub-coverings of X_G}}

\subsection{Sub-covers of $X_{G}$ for $G$ a free group\label{subsec:sub-coverings-free-groups}}

Consider a compact sub-cover $Y$ of $X_{\Gamma}$ projecting entirely
into $X_{G}=X_{G_{i}}$ where $G=G_{i}$ is a rank-$r$ free group
with basis $B=\left\{ b_{1},\ldots,b_{r}\right\} $. So $Y$ is a
finite directed graph, not necessarily connected, equipped with a
graph immersion to $X_{G}$, the bouquet with $r$ loops. Equivalently,
$Y$ is a directed finite graph with edges labeled by $b_{1},\ldots,b_{r}$,
and at every vertex, at most one incoming $b_{j}$-edge and at most
one outgoing $b_{j}$-edge for every $j$. We call such a graph \textbf{a
$B$-labeled graph}. Such graphs are closely related to Stallings
core graphs \cite{stallings1983topology} and more generally to multi
core graphs \cite{hanany2020word}, but they may contain leaves and\textbackslash or
isolated vertices. 

It is straight-forward to compute $\emb_{Y}\left(N\right)$, the expected
number of embeddings of $Y$ into a random $N$-sheeted cover of $X_{\Gamma}$,
and there is no need here in any fancy resolution. Recall that $\left(N\right)_{t}\defi N\left(N-1\right)\cdots\left(N-t+1\right)$
denotes the falling factorial.
\begin{lem}
\label{lem:E^emb over free groups}Let $Y$ be a finite $B$-labeled
graph, $v\left(Y\right)$ be the number of vertices in $Y$ and $e_{j}\left(Y\right)$
the number of $b_{j}$-edges. Then for every $N\ge\max_{j}e_{j}\left(Y\right)$,
\[
\emb_{Y}\left(N\right)=\frac{\left(N\right)_{v\left(Y\right)}}{\prod_{j=1}^{r}\left(N\right)_{e_{j}\left(Y\right)}}.
\]
\end{lem}

\begin{proof}
As $Y$ is a sub-cover of $X_{\Gamma}$ sitting exclusively above
a particular component $X_{G}=X_{G_{i}}$ where $G=G_{i}\cong\F_{r}$
is a free group, it is enough to consider random $N$-covers $\hat{X}$
of $X_{G}$. Then $\hat{X}$ is given by $N$ vertices, labeled $1,\ldots,N$,
above the unique vertex $v=v_{i}$ of $X_{G}$. Above the $b_{j}$-loop
at $v$, there are $N$ $b_{j}$-edges in $\hat{X}$, which are given
by a uniformly random permutation $\sigma_{j}\in S_{N}$. The random
permutations $\sigma_{1},\ldots,\sigma_{r}$ are independent. 

The number of possible embeddings of the \emph{vertices} of $Y$ into
$\hat{X}$ is precisely $\left(N\right)_{v\left(Y\right)}$. Any given
embedding of the vertices of $Y$ extends to an embedding of the entire
of $Y$ if and only if for every $j=1,\ldots,r$, the random permutation
$\sigma_{j}$ maps the image of the beginning point of every $b_{j}$-edge
in $Y$ to the image of its endpoint. As every vertex of $Y$ has
at most one incoming $b_{j}$-edge and at most one outgoing $b_{j}$-edge,
there are permutations satisfying these constraints. Finally, the
probability that a uniformly random permutation in $S_{N}$ satisfies
$e_{j}$ valid constraints is precisely 
\[
\frac{\left(N-e_{j}\right)!}{N!}=\frac{1}{\left(N\right)_{e_{j}}}.
\]
\end{proof}
\begin{cor}
\label{cor:E^emb over free group}Theorem \ref{thm:E^emb is N^chi and asym expansion}
holds for sub-covers of $X_{G}$ when $G$ is free. Namely, for a
compact sub-cover $Y$ of $X_{G}$, there are rational numbers $a_{t}=a_{t}\left(Y\right)$
for $t=-1,-2,-3,\ldots$ so that
\begin{equation}
\emb_{Y}\left(N\right)\asyexp N^{\chigrp\left(Y\right)}\cdot\left\{ 1+a_{-1}N^{-1}+a_{-2}N^{-2}+\ldots\right\} .\label{eq:asy-exp-of-E^emb above free}
\end{equation}
\end{cor}

\begin{proof}
If $Y$ is a $B$-labeled graph, then the restricted covering map
$p\colon Y\to X_{G}$ is an immersion of graphs, and therefore $p_{*}$
is injective. In particular, in every connected component $Y_{j}$
of $Y$, we have $\pi_{1}\left(Y_{j}\right)\cong p_{*}\left(\pi_{1}\left(Y_{j}\right)\right)=\plab\left(Y_{j}\right)$,
and so $\chigrp\left(Y_{i}\right)=\chi\left(Y_{i}\right)$. Thus $\chigrp\left(Y\right)=\chi\left(Y\right)=v\left(Y\right)-\sum_{j=1}^{r}e_{j}\left(Y\right)$,
and \eqref{eq:asy-exp-of-E^emb above free} follows from Lemma \ref{lem:E^emb over free groups}.
\end{proof}

\subsection{Sub-covers of $X_{G}$ for $G$ a finite group\label{subsec:sub-coverings-finite-groups}}

Here we prove Proposition \ref{prop:exp-torsion-element} about $\mathbb{E}\left[\fix_{\gamma}\left(N\right)\right]$
for a \emph{torsion} element $\gamma$, as well as the special case
of Theorem \ref{thm:E^emb is N^chi and asym expansion} concerning
a sub-cover $Y$ of $X_{\Gamma}$ projecting entirely into $X_{G}=X_{G_{i}}$
where $G=G_{i}$ is some finite group. Recall that $X_{G}$ is some
finite presentation complex of $G$. For every sub-cover $p\colon Y\to X_{G}$,
define the set 
\begin{equation}
\R_{Y}\defi\left\{ f\colon Y\to Z_{f}\,\middle|\,\begin{gathered}Z_{f}~\mathrm{is~a~(full)~covering~of~}X_{G},~\mathrm{and}\\
f\left(Y\right)~\mathrm{meets~every~connected~component~of}~Z_{f}
\end{gathered}
\right\} ,\label{eq:resolution of finite group sub-coverings}
\end{equation}
where $f$ is a morphism of sub-covers, namely, it commutes with the
immersions into $X_{G}$. We also denote
\begin{equation}
\R_{Y}^{\mathrm{emb}}\defi\left\{ f\colon Y\hookrightarrow Z_{f}\,\middle|\,\begin{gathered}Z_{f}~\mathrm{is~a~(full)~covering~of~}X_{G},~f~\mathrm{is~injective,~and}\\
f\left(Y\right)~\mathrm{meets~every~connected~component~of}~Z_{f}
\end{gathered}
\right\} \subseteq{\cal R}_{Y},\label{eq:injective resolution of finite group sub-coverings}
\end{equation}
Note that there may be distinct elements of $\R_{Y}$ or of ${\cal R}_{Y}^{\mathrm{emb}}$
with the same codomain $Z_{f}$.
\begin{prop}
\label{prop:resolutions over finite groups}Let $G$ be a finite group
and $Y$ a compact sub-cover of $X_{G}$. Then,
\begin{enumerate}
\item the set $\R_{Y}$ from \eqref{eq:resolution of finite group sub-coverings}
is a finite resolution of $Y$, and 
\item the set ${\cal R}_{Y}^{\mathrm{emb}}$ from \eqref{eq:injective resolution of finite group sub-coverings}
is a finite embedding-resolution of $Y$.
\end{enumerate}
\end{prop}

\begin{proof}
For every element $f\colon Y\to Z_{f}$ in $\R_{Y}$, the number of
components in $Z_{f}$ is bounded by the number of components of $Y$,
and because the number of connected covers of $X_{G}$ is finite (equal
to the number of conjugacy classes of subgroups of $G$), we get that
there are finitely many possibilities for $Z_{f}$. As there are finitely
many morphisms of sub-covers between two given compact sub-covers,
we conclude that $\R_{Y}$ is finite, and thus so is its subset $\R_{Y}^{\mathrm{emb}}$.

The set $\R_{Y}$ is a resolution because every morphism $h\colon Y\to\hat{X}$
to a covering space $\pi\colon\hat{X}\to X_{G}$ decomposes uniquely
to a map from $Y$ to the connected components of $\hat{X}$ meeting
$h\left(Y\right)$, followed by the embedding of these components
in $\hat{X}$. The same argument shows that $\R_{Y}^{\mathrm{emb}}$
is an embedding-resolution.
\end{proof}
\begin{prop}
\label{prop:E^emb of full coverings above a finite group}Theorem
\ref{thm:E^emb is N^chi and asym expansion} holds for (full) covers
of $X_{G}$ when $G$ is a finite group. Namely, let $Z$ be a compact
(full) covering space of $X_{G}$. Denote $\mu=\left|G\right|$. Then
there are rational numbers $a_{t}=a_{t}\left(Z\right)$ for $t=-1/\mu,-2/\mu$,$\ldots$
so that
\begin{equation}
\emb_{Z}\left(N\right)\asyexp N^{\chigrp\left(Z\right)}\cdot\left\{ 1+a_{-1/\mu}N^{-1/\mu}+a_{-2/\mu}N^{-2/\mu}+\ldots\right\} .\label{eq:asy expansion of E^emb(Z)}
\end{equation}
Furthermore, $a_{t}=0$ for $0>t>-\frac{1}{2}$, so 
\[
\emb_{Z}\left(N\right)=N^{\chigrp\left(Z\right)}\left(1+O\left(N^{-1/2}\right)\right).
\]
\end{prop}

\begin{proof}
Denote by $v=v\left(Z\right)$ the number of vertices in $Z$. Recall
that $X_{G}$ contains a single vertex and so every $N$-cover of
it contains exactly $N$ vertices. In every embedding $h\colon Z\hookrightarrow\hat{X}$
of $Z$ into an $N$-cover $\hat{X}$ of $X_{G}$, $h\left(Z\right)$
contains $v$ out of the $N$ vertices of $\hat{X}$. Moreover, as
$Z$ is a full cover of $X_{G}$, every embedding of $h\colon Z\hookrightarrow\hat{X}$
into an $N$-cover of $X_{G}$ has the property that $h\left(Z\right)$
and its complement are disconnected. So the embeddings of $Z$ in
all the $N$-covers of $X_{G}$ are in bijection with the embeddings
of the vertices of $Z$ into $\left[N\right]$ along with an arbitrary
$\left(N-v\right)$-cover which ``uses'' the remaining vertices
in $\left[N\right]$. As the number of embeddings of the vertices
of $Z$ in $\left[N\right]$ is $\left(N\right)_{v}$, we obtain
\begin{equation}
\emb_{Z}\left(N\right)=\frac{\left(N\right)_{v}\cdot\left|\Hom\left(G,S_{N-v}\right)\right|}{\left|\Hom\left(G,S_{N}\right)\right|}.\label{eq:first expression for E^emb over finite group}
\end{equation}
By \cite[Thm.~6]{muller1997finite} (stated as Theorem \ref{thm:Muller}
above), we have that $\frac{\left|\Hom\left(G,S_{N}\right)\right|}{\left|\Hom\left(G,S_{N-1}\right)\right|}$
has asymptotic expansion with rational coefficients
\[
\frac{\left|\Hom\left(G,S_{N}\right)\right|}{\left|\Hom\left(G,S_{N-1}\right)\right|}\asyexp N^{1-1/\mu}\cdot\left\{ 1+Q_{-1/\mu}N^{-1/\mu}+Q_{-2/\mu}N^{-2/\mu}+\ldots\right\} .
\]
Moreover, \cite[pp.~552]{muller1997finite} specifies the precise
values of the $\mu+3$ first coefficients $Q_{-1/\mu},\ldots,Q_{-\left(\mu+3\right)/\mu}$
in this asymptotic expansion. In particular, for $1\le\nu\le\mu-1$,
if $\mu-v\nmid\mu$ then $Q_{-\nu/\mu}=0$. Therefore, $Q_{t}=0$
for $0>t>-\frac{1}{2}$. We conclude that the inverse has asymptotic
expansion with rational coefficients:
\[
\frac{\left|\Hom\left(G,S_{N-1}\right)\right|}{\left|\Hom\left(G,S_{N}\right)\right|}\asyexp N^{-1+1/\mu}\cdot\left\{ 1+\beta_{-1/\mu}N^{-1/\mu}+\beta_{-2/\mu}N^{-2/\mu}+\ldots\right\} ,
\]
where, here too, $\beta_{t}=0$ for $0>t>-\frac{1}{2}$. For any $j\in\mathbb{Z}$,
we have 
\begin{eqnarray*}
\frac{\left|\Hom\left(G,S_{N-j-1}\right)\right|}{\left|\Hom\left(G,S_{N-j}\right)\right|} & \asyexp & \left(N-j\right)^{-1+1/\mu}\cdot\left\{ 1+\beta_{-1/\mu}\left(N-j\right)^{-1/\mu}+\beta_{-2/\mu}\left(N-j\right)^{-2/\mu}+\ldots\right\} \\
 & \asyexp & N^{-1+1/\mu}\cdot\left\{ 1+\beta_{-1/\mu}^{\left(j\right)}N^{-1/\mu}+\beta_{-2/\mu}^{\left(j\right)}N^{-2/\mu}+\ldots\right\} 
\end{eqnarray*}
where the second equality follows from Taylor's theorem, applied to
the function $\left(N-j\right)^{t}$ at the point $N$, and the $\beta_{t}^{\left(j\right)}$'s
are rational numbers. Moreover, the Taylor expansion of $\left(N-j\right)^{t}$
at $N$ is of the form $N^{t}+c_{1}N^{t-1}j+c_{2}N^{t-2}j^{2}+\ldots$,
so for $0>t>-1$, $\beta_{t}^{\left(j\right)}=\beta_{t}$. In particular,
$\beta_{t}^{\left(j\right)}=0$ for $0>t>-\frac{1}{2}$. Therefore,
\begin{eqnarray*}
\frac{\left|\Hom\left(G,S_{N-v}\right)\right|}{\left|\Hom\left(G,S_{N}\right)\right|} & = & \frac{\left|\Hom\left(G,S_{N-1}\right)\right|}{\left|\Hom\left(G,S_{N}\right)\right|}\cdot\frac{\left|\Hom\left(G,S_{N-2}\right)\right|}{\left|\Hom\left(G,S_{N-1}\right)\right|}\cdots\frac{\left|\Hom\left(G,S_{N-v}\right)\right|}{\left|\Hom\left(G,S_{N-v+1}\right)\right|}\\
 & \asyexp & N^{-1+1/\mu}\cdot\left\{ 1+\beta_{-1/\mu}^{\left(0\right)}N^{-1/\mu}+\beta_{-2/\mu}^{\left(0\right)}N^{-2/\mu}+\ldots\right\} \cdot\\
 &  & N^{-1+1/\mu}\cdot\left\{ 1+\beta_{-1/\mu}^{\left(1\right)}N^{-1/\mu}+\beta_{-2/\mu}^{\left(1\right)}N^{-2/\mu}+\ldots\right\} \cdot\\
 &  & \vdots\\
 &  & N^{-1+1/\mu}\cdot\left\{ 1+\beta_{-1/\mu}^{\left(v-1\right)}N^{-1/\mu}+\beta_{-2/\mu}^{\left(v-1\right)}N^{-2/\mu}+\ldots\right\} \\
 & \asyexp & N^{-v+v/\mu}\cdot\left\{ 1+\delta_{-1/\mu}N^{-1/\mu}+\delta_{-2/\mu}N^{-2/\mu}+\ldots\right\} ,
\end{eqnarray*}
where the $\delta_{t}$'s are rational constants depending on $G$
and on $v$. Because $\beta_{t}^{\left(j\right)}=0$ for $0>t>-\frac{1}{2}$,
then so does $\delta_{t}=0$ for $0>t>-\frac{1}{2}$. Together with
\eqref{eq:first expression for E^emb over finite group}, this proves
there is an asymptotic expansion for $\emb_{Z}\left(N\right)$ as
in the statement of the proposition and with leading term $N^{v/\mu}$.
It remains to prove that $\chigrp\left(Z\right)=\frac{v}{\mu}$.

Let $Z_{1},\ldots,Z_{\ell}$ denote the connected components of $Z$,
and let $J_{i}\in\plab\left(Z_{i}\right)$ be a subgroup of $G$ inside
the conjugacy class of subgroups corresponding to $Z_{i}$. Then $Z_{i}$
is a $\left[G:J_{i}\right]$-sheeted cover of $X_{G}$, and in particular
has $\left[G:J_{i}\right]$ vertices. Recall that, by definition,
$\chi\left(J_{i}\right)=\frac{1}{\left|J_{i}\right|}$, and note that
$\frac{1}{\left|J_{i}\right|}=\frac{\left[G:J_{i}\right]}{\left|G\right|}=\frac{\left[G:J_{i}\right]}{\mu}$.
Thus
\[
\chigrp\left(Z\right)=\chi\left(J_{1}\right)+\ldots\chi\left(J_{\ell}\right)=\frac{\left[G:J_{1}\right]}{\mu}+\ldots+\frac{\left[G:J_{\ell}\right]}{\mu}=\frac{v}{\mu}.
\]
\end{proof}
\begin{cor}
\label{cor:E^emb above finite}Theorem \ref{thm:E^emb is N^chi and asym expansion}
holds for sub-covers of $X_{G}$ when $G$ is finite. Namely, denoting
$\mu=\left|G\right|$, for a compact sub-cover $Y$ of $X_{G}$, there
are rational numbers $a_{t}=a_{t}\left(Y\right)$ for $t=-1/\mu,-2/\mu,\ldots$
so that
\begin{equation}
\emb_{Y}\left(N\right)\asyexp N^{\chigrp\left(Y\right)}\cdot\left\{ 1+a_{-1/\mu}N^{-1/\mu}+a_{-2/\mu}N^{-2/\mu}+\ldots\right\} .\label{eq:asy-exp-of-E^emb above finite}
\end{equation}
\end{cor}

\begin{proof}
Recall the set ${\cal R}_{Y}^{\mathrm{emb}}$ defined in \eqref{eq:injective resolution of finite group sub-coverings},
which is a finite embedding-resolution of $Y$ by Proposition \ref{prop:resolutions over finite groups}.
By Lemma \ref{lem:resolution-sum-of-probabilities}, 
\[
\emb_{Y}\left(N\right)=\sum_{f\in{\cal R}_{Y}^{\mathrm{emb}}}\emb_{Z_{f}}\left(N\right),
\]
and so we conclude from Proposition \ref{prop:E^emb of full coverings above a finite group}
that $\emb_{Y}\left(N\right)$ has asymptotic expansion $\sum_{t\in\frac{1}{\mu}\mathbb{Z}}a_{t}N^{t}$
with leading term $\max_{f\in{\cal R}_{Y}^{\mathrm{emb}}}\chigrp\left(Z_{f}\right)$.
Consider the universal lift $\hat{p}\colon Y\to\Upsilon$ from Definition
\ref{def:universal lift} (where here $X_{G}$ is in the role of $X_{\Gamma}$,
so $\Upsilon$ is a full cover of $X_{G}$). By definition, $\hat{p}\left(Y\right)$
intersects every component of $\Upsilon$, and by Lemma \ref{lem:universal lift is injective},
$\hat{p}$ is an embedding. Thus $\hat{p}\in\remb_{Y}$. As $\plab\left(\Upsilon\right)\cong\plab\left(Y\right)$,
we get $\chigrp\left(\Upsilon\right)=\chigrp\left(Y\right)$. It remains
to show that for any other element $\hat{p}\ne f\in\remb_{Y}$ we
have $\chigrp\left(Z_{f}\right)\lvertneqq\chigrp\left(Y\right)$.

First, we may reduce to the case where each connected component of
$Y$ is mapped to its own connected component of $Z_{f}$. Indeed,
if there are two distinct components $Y_{1}$ and $Y_{2}$ of $Y$
which are mapped to the same component $Z_{o}$ of $Z_{f}$, we may
reduce to some $f'\in{\cal R}_{Y}^{\mathrm{emb}}$ with $Z_{f'}$
having more connected components by duplicating $Z_{o}$ to two copies
and mapping $Y_{1}$ to one copy and $Y_{2}$ to another copy. Using
the fact that the EC of finite groups is positive, we obtain $\chigrp\left(Z_{f}\right)<\chigrp\left(Z_{f'}\right)$.

So now it is enough to assume that $Y$ is connected and prove that
for $f\in\remb_{Y}$, we have $\chigrp\left(Z_{f}\right)\le\chigrp\left(Y\right)$
with equality if and only if $f$ is the universal lift of $p\colon Y\to X_{G}$.
Choose an arbitrary vertex $y\in Y$ and denote $J=\plab\left(Y,y\right)=\plab\left(\Upsilon,\hat{p}\left(y\right)\right)\le G$.
The existence of $f$ yields that $J\le\plab\left(Z_{f},f\left(y\right)\right)$,
and thus there is a morphism of covering spaces $\tau\colon\left(\Upsilon,\hat{p}\left(y\right)\right)\to\left(Z_{f},f\left(y\right)\right)$.
By the classification of covering spaces (e.g.~\cite[Thm.~1.38]{hatcher2005algebraic}),
$\tau$ is an isomorphism if and only if $J=\plab\left(Z_{f},f\left(y\right)\right)$.
So if $f\ne\hat{p}$ we obtain $J\lvertneqq\plab\left(Z_{f},f\left(y\right)\right)$,
and 
\[
\chigrp\left(Y\right)=\chi\left(J\right)=\frac{1}{\left|J\right|}>\frac{1}{\left|\plab\left(Z_{f}\right)\right|}=\chi\left(\plab\left(Z_{f}\right)\right)=\chigrp\left(Z_{f}\right).
\]
\end{proof}
We can now also prove Proposition \ref{prop:exp-torsion-element}
stating that for any \emph{torsion} element $\gamma\in\Gamma$, we
have $\mathbb{E}\left[\fix_{\gamma}\left(N\right)\right]=N^{1/\left|\gamma\right|}+O\left(N^{1/\left(2\left|\gamma\right|\right)}\right)$.
Along the way we also prove the existence of asymptotic expansion,
as in Theorem \ref{thm:asmptotic expansion}, for torsion elements
of $\gamma$.
\begin{proof}[Proof of Theorem \ref{thm:asmptotic expansion} for torsion elements
and of Proposition \ref{prop:exp-torsion-element}]
 Assume that $G$ is finite, that $\gamma\in G$ and that $\varphi\in\Hom\left(G,S_{N}\right)$
is uniformly random. Let $p\colon(\hat{X}_{\left\langle \gamma\right\rangle },x)\to\left(X_{G},v\right)$
be the connected covering space with $\plab(\hat{X}_{\left\langle \gamma\right\rangle },x)=\left\langle \gamma\right\rangle \le G$.
Consider the $N$-cover $X_{\varphi}$ of $X_{G}$ corresponding to
$\varphi$, with vertices labeled by $\left[N\right]$. Then $\varphi\left(\gamma\right)$
fixes $i\in\left[N\right]$ if and only if there is a lift of $p$
to $\left(X_{\varphi},i\right)$. Thus $\mathbb{E}\left[\fix_{\gamma}\left(N\right)\right]=\mathbb{E}_{\hat{X}_{\left\langle \gamma\right\rangle }}\left(N\right)$.
By Proposition \ref{prop:resolutions over finite groups}, the set
$\R={\cal R}_{\hat{X}_{\left\langle \gamma\right\rangle }}$ from
\eqref{eq:resolution of finite group sub-coverings} is a finite resolution
for $\hat{X}_{\left\langle \gamma\right\rangle }$, and by Lemma \ref{lem:resolution-sum-of-probabilities},
\begin{equation}
\mathbb{E}\left[\fix_{\gamma}\left(N\right)\right]=\mathbb{E}_{\hat{X}_{\left\langle \gamma\right\rangle }}\left(N\right)=\sum_{f\in{\cal R}}\mathbb{E}_{Z_{f}}^{\mathrm{emb}}\left(N\right).\label{eq:E=00005Bfix=00005D for torsion as sum of E^emb}
\end{equation}
Theorem \ref{thm:asmptotic expansion} for torsion elements, namely,
the fact that $\mathbb{E}\left[\fix_{\gamma}\left(N\right)\right]$
has asymptotic expansion as in \eqref{eq:asym expansion of E=00005Bfix=00005D},
now follows from \eqref{eq:E=00005Bfix=00005D for torsion as sum of E^emb}
together with Proposition \ref{prop:E^emb of full coverings above a finite group}.

Note that the identity map $\mathrm{id}\colon\hat{X}_{\left\langle \gamma\right\rangle }\to\hat{X}_{\left\langle \gamma\right\rangle }$
belongs to ${\cal R}$. As $\chi\left(\left\langle \gamma\right\rangle \right)=\frac{1}{\left|\gamma\right|}$,
this element of ${\cal R}$ satisfies
\[
\emb_{\hat{X}_{\left\langle \gamma\right\rangle }}\left(N\right)=N^{1/\left|\gamma\right|}\left(1+O\left(N^{-1/2}\right)\right)=N^{1/\left|\gamma\right|}+O\left(N^{1/\left|\gamma\right|-1/2}\right)
\]
by Proposition \ref{prop:E^emb of full coverings above a finite group}.
Note that $\frac{1}{\left|\gamma\right|}-\frac{1}{2}\le\frac{1}{2\left|\gamma\right|}$.
It is left to show that for every other element $\id\ne f\in{\cal R}$,
$\chigrp\left(Z_{f}\right)\le\frac{1}{2\left|\gamma\right|}$. Indeed,
as $\hat{X}_{\left\langle \gamma\right\rangle }$ is connected, so
is $Z_{f}$ for every $f\in{\cal R}$. As $f$ is a lift of $p$ but
$f\ne\id$, we have $\left\langle \gamma\right\rangle \lvertneqq\plab\left(Z_{f},f\left(x\right)\right)$.
Thus 
\[
\chigrp\left(Z_{f}\right)=\frac{1}{\left|\plab\left(Z_{f},f\left(x\right)\right)\right|}=\frac{1}{\left|\gamma\right|\cdot\left[\plab\left(Z_{f},f\left(x\right)\right):\left\langle \gamma\right\rangle \right]}\le\frac{1}{2\left|\gamma\right|}.
\]
\end{proof}
\begin{rem}
\label{rem:another version for torsion elements}In Section \ref{sec:Introduction}
we claimed that torsion elements $\gamma\in\Gamma$ satisfy also \eqref{eq:leading order of E=00005Bfix=00005D for torsion and non-torsion},
namely, that
\begin{equation}
\mathbb{E}\left[\fix_{\gamma}\left(N\right)\right]=N^{1/\left|\gamma\right|}\left(1+O\left(N^{-1/m}\right)\right).\label{eq:another version for torsion elements}
\end{equation}
Indeed, if $\gamma$ is conjugated into the finite group $G$ (one
of the factors of $\Gamma$), then $m\left(\Gamma\right)\ge m\left(G\right)=\left|G\right|$,
and we may thus assume that $m=\left|G\right|$. If $\left|\gamma\right|=\left|G\right|$
then $\left\langle \gamma\right\rangle =G$, namely, $\gamma$ does
not belong to any proper subgroup of $G$. Thus, the only element
of the resolution $\R$ from the last proof is $\mathrm{id}\colon\hat{X}_{\left\langle \gamma\right\rangle }\to\hat{X}_{\left\langle \gamma\right\rangle }$,
and $\mathbb{E}\left[\fix_{\gamma}\left(N\right)\right]=\emb_{\hat{X}_{\left\langle \gamma\right\rangle }}\left(N\right)=N^{1/\left|\gamma\right|}\left(1+O\left(N^{-1/2}\right)\right)$.
This yields \eqref{eq:another version for torsion elements} as $m\ge2$.
Finally, if $\left|\gamma\right|\lvertneqq\left|G\right|$, then $\left|\gamma\right|\le\frac{m}{2}$,
and \eqref{eq:another version for torsion elements} follows immediately
from Proposition \ref{prop:exp-torsion-element}.
\end{rem}

\begin{example}
\label{exa:x^2  in C_4}Let $G=C_{4}=\left\langle x\right\rangle $
be the cyclic group of size $4$ generated by $x$, and consider the
element $x^{2}$. There are two subgroups containing $x^{2}$: $\left\langle x\right\rangle $
and $\left\langle x^{2}\right\rangle $, with corresponding coverings
spaces $\hat{X}_{\left\langle x\right\rangle }$ and $\hat{X}_{\left\langle x^{2}\right\rangle }$.
The computations appearing above together with the some precise values
of coefficients from \cite[p.~552]{muller1997finite}, yield 
\begin{eqnarray*}
\emb_{\hat{X}_{\left\langle x\right\rangle }} & \asyexp & N^{1/4}\cdot\left\{ 1-\frac{1}{4}N^{-1/2}-\frac{1}{4}N^{-3/4}+\ldots\right\} \\
\emb_{\hat{X}_{\left\langle x^{2}\right\rangle }} & \asyexp & N^{1/2}\cdot\left\{ 1-\frac{1}{2}N^{-1/2}-\frac{1}{2}N^{-3/4}+\ldots\right\} ,
\end{eqnarray*}
so 
\[
\mathbb{E}\left[\fix_{x^{2}}\left(N\right)\right]\asyexp N^{1/2}+N^{1/4}-\frac{1}{2}-\frac{3}{4}N^{-1/4}+\ldots.
\]
\end{example}

\begin{rem}
\label{rem:Muller - Schlage-Puchta}Some of the results of this subsection
\ref{subsec:sub-coverings-finite-groups} also follow from \cite{muller2010statistics}.
Let $p\colon Z\to X_{G}$ be a \emph{connected} (full) cover with
$m_{i}$ vertices, and let $m=\left|G\right|$, so $\chigrp\left(Z\right)=\frac{m_{i}}{m}$.
Then \cite[Lem.~4]{muller2010statistics} states that $\frac{\emb_{Z}\left(N\right)-N^{m_{i}/m}}{N^{m_{i}/\left(2m\right)}}$
converges in distribution to a standard Gaussian ${\cal N}\left(0,1\right)$.
Note that the statement of \cite[Lem.~4]{muller2010statistics} wrongly
implies that what is being counted is the number of \emph{disjoint}
copies of $Z$ in a random $N$-cover, whereas what is actually being
counted there is $\emb_{Z}\left(N\right)$, namely the number of disjoint
copies times $\left|\mathrm{Aut}\left(Z\right)\right|$, the number
of automorphisms of $Z$ as a covering map. See also Remark \ref{rem:point of confusion}.
\end{rem}

\subsection{Sub-covers of $X_{G}$ for $G$ a surface group\label{subsec:sub-coverings-surface-groups}}

We now assume that $G=\Lambda_{g}=\left\langle a_{1},b_{1},\ldots,a_{g},b_{g}\,\middle|\,\left[a_{1},b_{1}\right]\cdots\left[a_{g},b_{g}\right]\right\rangle $
with $g\ge2$. Recall from Section \ref{subsec:The-graph-of-spaces}
that $X_{G}$ is an orientable surface of genus $g$ endowed with
a CW-structure of a single vertex $v$, $2g$ edges labeled $a_{1},\ldots,b_{g}$,
and a single $2$-cell, which we think of as a $4g$-gon as its boundary
is attached to a cycle of $4g$ edges. A sub-cover of $X_{G}$ is
also called a \emph{tiled surface }in \cite{MPcore,magee2020asymptotic}.
See also \cite[Prop.~3.3]{MPcore} for an intrinsic definition of
a tiled surface.

We now introduce some further terminology from \cite{MPcore,magee2020asymptotic}.
The definitions are laconic as they are only used in order to state
some results from these two papers. Let $Y\subseteq Z$ be a sub-cover
which is a subcomplex of the (full) covering space $p\colon Z\to X_{G}$.
As $Y$ is embedded in a surface, we may take a small closed regular
neighborhood of $Y$ in $Z$ and obtain the ``thick version'' of
$Y$ which is a surface, possibly with boundary. The thick version
of $Y$, which we sometimes denote by $\mathbb{Y}$, is a feature
of $Y$ as a sub-cover, and does not depend on the particular $Z$
it is embedded in -- see \cite[Sec.~3.1]{MPcore}. We write $\partial Y$
for the boundary of the thick version of $Y$. This boundary is a
finite collection of cycles. We pick an orientation on every boundary
component (see below) to obtain a \textbf{boundary cycle} of $Y$,
and using the edge-labels along a boundary cycle, it corresponds to
some cyclic word in the generators of $G$. 

Every full cover $Z$ of $X_{G}$ consists of vertices, directed edges
labeled by $a_{1},\ldots,b_{g}$, and $4g$-gons. The cycle around
every $4g$-gon reads the relation -- the cyclic word $R=\left[a_{1},b_{1}\right]\cdots\left[a_{g},b_{g}\right]$.
A boundary cycle of a sub-cover $Y$ is always oriented so that if
$Y$ is embedded in the full cover $Z$, the cycle reads successive
segments of the boundaries of the neighboring $4g$-gons (in $Z\setminus Y$)
with the orientation of each $4g$-gon coming from $\left[a_{1},b_{1}\right]\cdots\left[a_{g},b_{g}\right]$
(and not from the inverse word). 

If a boundary cycle of a sub-cover $Y$ contains a subword of $R$
of length $>\frac{1}{2}\left|R\right|=2g$, then in \emph{every} full
cover $Z$ in which $Y$ is embedded, one may shorten the total boundary
of $Y$ by annexing the $4g$-gon neighboring this subword. In this
sense the boundary of $Y$ is not ``reduced''. We call a subword
of $R$ of length $\ge2g+1$ a \textbf{long block}. 

There are further, more involved cases involving a sequence of a few
consecutive $4g$-gons where $\partial Y$ is not reduced. For example,
if $g=2$ and $\partial Y$ contains the subword $a_{1}b_{1}a_{1}^{-1}b_{1}^{-1}b_{2}^{-1}a_{1}b_{1}b_{1}a_{1}^{-1}b_{1}^{-1}a_{2}$
then there are three consecutive octagons neighboring this subword,
and annexing them strictly reduces the total length of $\partial Y$.
Such subwords are called \textbf{long chains} -- see \cite[Sec.~3.2]{MPcore}
for the precise definition.\textbf{ }This leads to the following definition.
\begin{defn}[Boundary reduced]
\cite[Def. 4.1]{MPcore} Let $G=\Lambda_{g}$ with $g\ge2$. A sub-cover
$Y$ of $X_{G}$ is called \textbf{boundary reduced}, or \textbf{BR}
for short, if $\partial Y$ contains no long blocks nor long chains.
\end{defn}

If $\partial Y$ contains a subword which constitutes half of the
relation $R$, called a \textbf{half-block}, then in every full cover
$Z$ in which $Y$ is embedded, the neighboring $4g$-gon can be annexed
to $Y$ without increasing the total length of $\partial Y$. Likewise,
there are cycles called \textbf{half-chains }so that annexing the
sequence of consecutive $4g$-gons along them does not increase the
length of the boundary. Again, see \cite[Sec.~3.2]{MPcore} for the
precise definitions. 
\begin{defn}[Strongly boundary reduced]
\cite[Def. 4.2]{MPcore} Let $G=\Lambda_{g}$ with $g\ge2$. A sub-cover
$Y$ of $X_{G}$ is called \textbf{strongly boundary reduced}, or
\textbf{SBR} for short, if $\partial Y$ contains no half-blocks nor
half-chains.
\end{defn}

As explained in \cite[Sec.~4]{MPcore}, every SBR sub-cover is, in
particular, BR. The case of Theorem \ref{thm:E^emb is N^chi and asym expansion}
dealing with sub-covers of $X_{G}$ (where $G=\Lambda_{g}$ is a surface
group), crucially relies on the following results from \cite{magee2020asymptotic}.
\begin{thm}
\cite{magee2020asymptotic}\label{thm:E^emb of BR and SBR } Let $Y$
be a compact sub-cover of $X_{G}$ where $G=\Lambda_{g}$ with $g\ge2$. 
\begin{enumerate}
\item If $Y$ is BR, there are rational number $a_{t}=a_{t}\left(Y\right)$
for $t=0,-1,-2,\ldots$ with $a_{0}>0$, so that
\begin{equation}
\emb_{Y}\left(N\right)\asyexp N^{\chigrp\left(Y\right)}\cdot\left\{ a_{0}+a_{-1}N^{-1}+a_{-2}N^{-2}+\ldots\right\} .\label{eq:E^emb of BR}
\end{equation}
\item \label{enu:E^emb for SBR}If $Y$ is moreover SBR, then $a_{0}\left(Y\right)=1$.
\end{enumerate}
\end{thm}

Although they probably should have been, these results are not written
explicitly in \cite{magee2020asymptotic}. However, they follow immediately
from the results therein. In fact, as explained in \cite[Sec.~1.6 and 5.1]{magee2020asymptotic},
the results of that paper immediately give \eqref{eq:E^emb of BR}
with $\chigrp\left(Y\right)$ replaced with $\chi\left(Y\right)$.
But then \cite[Lem.~5.6]{MPcore} shows that if $Y$ is compact and
BR, then $\chigrp\left(Y\right)=\chi\left(Y\right)$.
\begin{thm}
\cite[Thm.~2.14]{magee2020asymptotic}\label{thm:resolution for a compact tiled surface}
Let $Y$ be a compact sub-cover of $X_{G}$ where $G=\Lambda_{g}$
and let $\chi_{0}\in\mathbb{Z}$. Then $Y$ admits a finite resolution
${\cal R}={\cal R}\left(Y,\chi_{0}\right)$ such that for every $f\colon Y\to Z_{f}$
in ${\cal R}$, the following properties holds:

$\left(i\right)$ the sub-cover $Z_{f}$ is compact and BR, 

$\left(ii\right)$ if $\chigrp\left(Z_{f}\right)\ge\chi_{0}$, then
$Z_{f}$ is SBR, and 

$\left(iii\right)$ the image of $f$ meets every connected component
of $Z_{f}$.
\end{thm}

The original statement of \cite[Thm.~2.14]{magee2020asymptotic} states
the second condition as $\chi\left(Z_{f}\right)\ge\chi_{0}$, but
as mentioned above, for compact BR sub-covers, $\chi\left(Z_{f}\right)=\chigrp\left(Z_{f}\right)$.
Part $\left(iii\right)$ is not mentioned in ibid, but it follows
from the specific construction of $\R$ in \cite[Def.~2.13]{magee2020asymptotic}.
\begin{cor}
\label{cor:injective resolution for a compact tiled surface} Let
$Y$ be compact sub-cover of $X_{G}$ where $G=\Lambda_{g}$ and let
$\chi_{0}\in\mathbb{Z}$. Then $Y$ admits a finite embedding-resolution
${\cal R}^{\mathrm{emb}}={\cal R}^{\mathrm{emb}}\left(Y,\chi_{0}\right)$
for the \emph{injective} lifts of $Y$ to (full) covers of $X_{G}$,
and with the same three properties as in Theorem \ref{thm:resolution for a compact tiled surface}.
\end{cor}

\begin{proof}
Take the subset of ${\cal R}\left(Y,\chi_{0}\right)$ from Theorem
\ref{thm:resolution for a compact tiled surface} consisting of all
injective morphisms. 
\end{proof}
Let $\tsg$ be the universal cover of the genus-$g$ orientable closed
surface $\Sigma_{g}$, endowed with the CW-complex structure pulled-back
from $X_{G}\cong\Sigma_{g}$. For every subgroup $J\le\Lambda_{g}$,
the corresponding covering space is $J\backslash\tsg$ (see \cite[Example 3.5]{MPcore}).

\begin{lem}
\label{lem:in resolution of Y chigrp<=00003Dchigroup(Y)} If $f\colon Y\hookrightarrow Z$
is an embedding of compact sub-covers of $X_{G}$ with $G=\Lambda_{g}$
such that $f\left(Y\right)$ meets every component of $Z$, then $\chigrp\left(Z\right)\le\chigrp\left(Y\right)$.
\end{lem}

\begin{proof}
Let $Z_{1},\ldots,Z_{s}$ be the connected components of $Z$ with
$z_{j}\in Z_{j}$ some vertex. Denote $H_{j}=\plab\left(Z_{j},z_{j}\right)$,
$\Upsilon_{j}=H_{j}\backslash\tsg$ and $\Upsilon=\Upsilon_{1}\sqcup\ldots\sqcup\Upsilon_{s}$.
Then the universal lift of $Z$ has codomain $\Upsilon$ and $Z$
is embedded in $\Upsilon$ by Lemma \ref{lem:universal lift is injective}.
We may think of $f$ as an embedding of $Y$ inside $\Upsilon$. As
$\plab\left(Z_{i}\right)=\plab\left(\Upsilon_{i}\right)$, it is enough
to prove that $\chigrp\left(\Upsilon\right)\le\chigrp\left(Y\right)$.

Consider the thick part $\mathbb{Y}$ of $f\left(Y\right)$ in $\Upsilon$,
with $\mathbb{Y}_{1},\ldots,\mathbb{Y}_{\ell}$ its connected components,
and denote by $C_{1},\ldots,C_{q}$ the connected components of the
complement $\Upsilon-\mathbb{Y}$. We denote by $\overline{C_{i}}$
the closure of the component $C_{i}$, and the fact it is a component
in the complement of $\mathbb{Y}$ and not of $f\left(Y\right)$ guarantees
that $\overline{C_{i}}^{\circ}=C_{i}$. 

We may assume that none of the $\mathbb{Y}_{i}$'s and none of the
$\overline{C_{j}}$'s are discs. Indeed, if some $\overline{C_{i}}$
is a disc, then we can replace $\mathbb{Y}$ with $\mathbb{Y}\cup C_{i}$:
this does not change $\plab\left(Y\right)$ nor $\chigrp\left(Y\right)$.
If any $\mathbb{Y}_{i}$ is a disc, then it neighbors a single $\overline{C_{j}}$.
Assume that $\mathbb{Y}_{i}\subseteq\Upsilon_{t}$. If $\mathbb{Y}_{i}\cup\overline{C_{j}}=\Upsilon_{t}$,
we may reduce to the case where this part is ignored completely, for
$H_{t}\le\Lambda_{g}$ and so $\chigrp\left(Y_{i}\right)=1\ge\chi\left(H_{t}\right)$.
If there are additional parts in $\Upsilon_{t}$, we may reduce to
the case where we remove $\mathbb{Y}_{i}$ from $\mathbb{Y}$ and
replace $\overline{C_{j}}$ with $\mathbb{Y}_{i}\cup\overline{C_{j}}$,
for then $\chigrp\left(Y\right)$ is decreased by one and $\chigrp\left(\Upsilon\right)$
does not change.

We obtained a decomposition of the space $\Upsilon$ to a graph of
spaces with vertex-spaces $\mathbb{Y}_{1},\ldots\mathbb{Y}_{\ell},\overline{C_{1}},\ldots,\overline{C_{q}}$
and all edge groups isomorphic to $\mathbb{Z}$ (every edge connects
some $\mathbb{Y}_{i}$ with some $\overline{C}_{j}$ and corresponds
to some boundary component of $\mathbb{Y}$). As the vertex-spaces
are not-a-disc surfaces and are embedded in hyperbolic surfaces, then
by the classification of surfaces they have non-trivial fundamental
groups. Furthermore, in every connected surface $S$ with boundary
which is not a disc, the cyclic fundamental group of every boundary
component is \emph{embedded} in $\pi_{1}\left(S\right)$. Thus, all
edge-groups (which are infinite cyclic) are embedded in the corresponding
vertex groups. By Bass-Serre theory of graph of groups, this means
that $\pi_{1}\left(\mathbb{Y}_{i}\right)$ is embedded in $H_{t}$
whenever $\mathbb{Y}_{i}\subseteq\Upsilon_{t}$. If $\pi\colon\Upsilon_{t}\to X_{G}$
is the covering map, then $\pi_{*}\colon\pi_{1}\left(\Upsilon_{t}\right)\to\pi_{1}\left(X_{G}\right)$
is injective, which yields that so is $\pi_{*}\circ f_{*}\colon\pi_{1}\left(\mathbb{Y}_{i}\right)\to\pi_{1}\left(X_{G}\right)$.
Thus $\plab\left(\mathbb{Y}_{i}\right)\cong\pi_{1}\left(\mathbb{Y}_{i}\right)$
and $\chigrp\left(Y_{i}\right)=\chi\left(\pi_{1}\left(\mathbb{Y}_{i}\right)\right)$.
Finally, because all edge groups are $\mathbb{Z}$ and all vertex
groups of the graph of spaces are non-trivial groups with non-positive
EC, we get
\begin{eqnarray}
\chigrp\left(Z\right) & = & \chi\left(H_{1}\right)+\ldots+\chi\left(H_{s}\right)=\chi\left(\pi_{1}\left(\Upsilon\right)\right)=\nonumber \\
 & = & \sum_{i=1}^{\ell}\chi\left(\pi_{1}\left(\mathbb{Y}_{i}\right)\right)+\sum_{j=1}^{q}\chi\left(\pi_{1}\left(\overline{C_{i}}\right)\right)-\sum_{e~\mathrm{edge~of~graph~of~spaces}}\chi\left(\mathbb{Z}\right)\\
 & = & \chigrp\left(Y\right)+\sum_{i=1}^{q}\chi\left(\pi_{1}\left(\overline{C_{i}}\right)\right)\le\chigrp\left(Y\right).\label{eq:chi of graph of spaces-1}
\end{eqnarray}
\end{proof}
We can now extend \eqref{eq:E^emb of BR} to arbitrary sub-covers
of $X_{G}$.
\begin{cor}
\label{cor:E^emb of arbitrary tiled surface}Theorem \ref{thm:E^emb is N^chi and asym expansion}
holds for sub-covers of $X_{G}$ when $G=\Lambda_{g}$ is a surface
group. Namely, for every compact sub-cover $Y$ of $X_{G}$ there
are rational numbers $a_{t}=a_{t}\left(Y\right)$ for $t=0,-1,-2,\ldots$
so that 
\begin{equation}
\emb_{Y}\left(N\right)\asyexp N^{\chigrp\left(Y\right)}\cdot\left\{ a_{0}+a_{-1}N^{-1}+a_{-2}N^{-2}+\ldots\right\} ,\label{eq:asy-exp-of-E^emb above surface}
\end{equation}
where $a_{0}\in\mathbb{Z}_{\ge1}$ is a positive integer.
\end{cor}

\begin{proof}
Let $Y$ be an arbitrary compact sub-cover of $X_{G}$. Set $\chi_{0}=\chigrp\left(Y\right)$
and let $\remb{\cal =\remb}\left(Y,\chi_{0}\right)$ be a finite embedding-resolution
as in Corollary \ref{cor:injective resolution for a compact tiled surface}.
By Lemma \ref{lem:in resolution of Y chigrp<=00003Dchigroup(Y)},
$\chigrp\left(Z_{f}\right)\le\chigrp\left(Y\right)$ for every $f\in\remb$,
and as $\emb_{Y}\left(N\right)=\sum_{f\in\remb}\emb_{Z_{f}}\left(N\right)$,
it follows from Theorem \ref{thm:E^emb of BR and SBR } that $\emb_{Y}\left(N\right)$
admits an asymptotic expansion as in \eqref{eq:asy-exp-of-E^emb above surface},
with some $a_{0}\in\mathbb{Q}_{\ge0}$. As every $f\in\remb$ with
$\chigrp\left(Z_{f}\right)\ge\chi_{0}=\chigrp\left(Y\right)$ is SBR,
we get from Theorem \ref{thm:E^emb of BR and SBR }\eqref{enu:E^emb for SBR}
that each such $f$ contributes 1 to $a_{0}$ and so, in fact, $a_{0}\in\mathbb{Z}_{\ge0}$.
It is thus left to show that there is an element of $\remb$ with
$\chigrp\left(Z_{f}\right)=\chigrp\left(Y\right)$.

Let $\hat{p}\colon Y\hookrightarrow\Upsilon$ be the universal lift
from Definition \ref{def:universal lift}, which is injective by Lemma
\ref{lem:universal lift is injective}. By the definition of an embedding-resolution,
this embedding $\hat{p}$ decomposes as 
\[
Y\stackrel{f}{\hookrightarrow}Z_{f}\stackrel{}{\hookrightarrow}\Upsilon,
\]
for some $f\in\remb$. Of course, $Z_{f}$ has the same number of
connected components as $Y$ (and $\Upsilon$). For each connected
component $Y_{i}$ of $Y$ with $y_{i}\in Y_{i}$ a vertex, we have 

\[
H_{i}\defi\plab\left(Y_{i},y_{i}\right)\le\plab\left(Z_{f},f\left(y_{i}\right)\right)\le\plab\left(\Upsilon,\hat{p}\left(y_{i}\right)\right)=H_{i}
\]
and so $\plab\left(Z_{f},f\left(y_{i}\right)\right)=H_{j}$. In particular,
$\chigrp\left(Z_{f}\right)=\chi\left(H_{1}\right)+\ldots+\chi\left(H_{\ell}\right)=\chigrp\left(Y\right)$.
\end{proof}
\begin{example}
Figure \ref{fig:BR with a_0=00003D2} illustrates two different SBR
sub-covers $Z_{1}$ and $Z_{2}$ in a possible resolution of a particular
(BR) sub-cover $Y$. One of them is a torus with one boundary component,
while the other is a pair of pants. In this example, $\plab\left(Y\right)\cong\F_{2}$
and $\chigrp\left(Y\right)=-1$. Both $Z_{1}$ and $Z_{2}$ have,
too, $\chi\left(Z_{1}\right)=\chi\left(Z_{2}\right)=-1$. In fact,
in an embedding-resolution $\remb$ of $Y$ which contains $Z_{1}$
and $Z_{2}$, they must be the only elements of $EC$ $-1$. This
shows that $a_{0}\left(Y\right)=2$, namely, 
\[
\emb_{Y}\left(N\right)\asyexp N^{-1}\cdot\left\{ 2+a_{-1}N^{-1}+a_{-2}N^{-2}+\ldots\right\} .
\]
\end{example}

\begin{figure}
\begin{centering}
\includegraphics[viewport=0bp 0bp 599.718bp 323.821bp,scale=0.7]{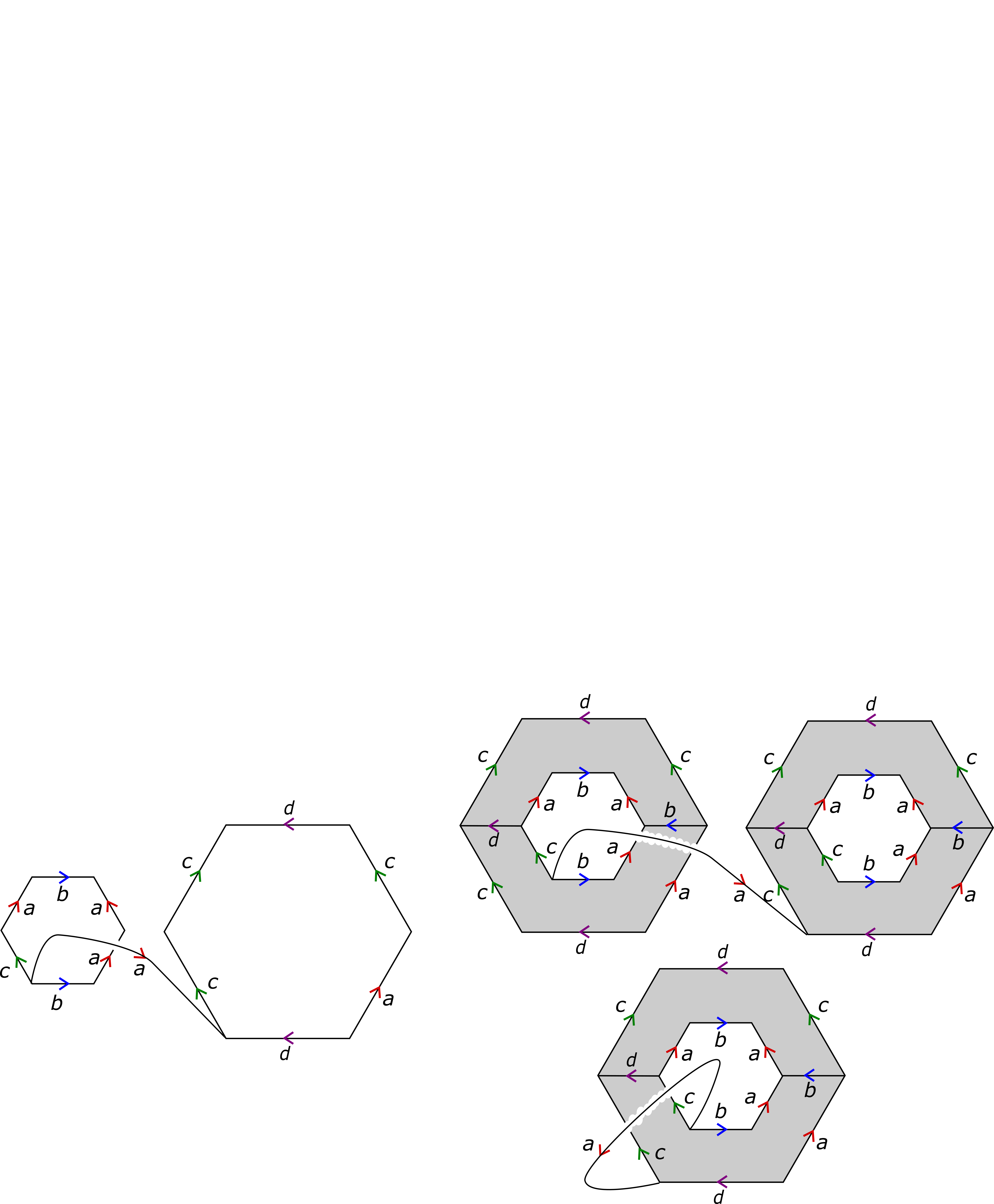}
\par\end{centering}
\caption{\label{fig:BR with a_0=00003D2}On the left hand side is a BR sub-cover
$Y$ of $X_{G}$ where $G=\Lambda_{2}=\left\langle a,b,c,d\,\middle|\,\left[a,b\right]\left[c,d\right]\right\rangle $.
This sub-cover consists of two hexagons connected by an additional
edge, and it satisfies $\chi\left(Y\right)=\protect\chigrp\left(Y\right)=-1$.
On the right there are two distinct SBR sub-covers in which $Y$ is
embedded, and which can serve as part of an embedding-resolution of
$Y$. Both of these have $\chi=\protect\chigrp=-1$ and $\protect\plab\protect\cong\protect\F_{2}$,
yet the upper one has $\protect\plab=\protect\plab\left(Y\right)$,
while the bottom one has $\protect\plab$ which is an HNN-extension
of $\protect\plab\left(Y\right)$.}
\end{figure}

To end this section, we characterize sub-covers $Y$ where $a_{0}\left(Y\right)=1$.
The characterization is stated in Proposition \ref{prop:no matching boundary cycles =00003D=00003D> a_0=00003D1},
and uses the following lemma (which could have fit in the paper \cite{MPcore}
better than the current one).
\begin{lem}
\label{lem:SBR does not bound an annulus}Let $Y\hookrightarrow Z$
be an SBR sub-cover $Y$ embedded in a full cover $Z$. Then there
is no boundary component of $Y$ bounding a disc in $Z\setminus Y$,
nor is there a pair of boundary components of $Y$ bounding an annulus
in $Z\setminus Y$.
\end{lem}

\begin{proof}
If some boundary cycle $\mathcal{C}$ of $Y$ bounds a disc, then
$\mathcal{C}$ spells a word which is equal to the trivial word in
$\Lambda_{g}$. By the classical results of Dehn \cite{Dehn}, $\mathcal{C}$
must contain a long block, contradicting the assumption that $Y$
is SBR.

Now assume that $\mathcal{C}_{1}$ and $\C_{2}$ are two boundary
cycles of $Y$ bounding an annulus of $Z\setminus Y$. They both represent
the same free-homotopy class in $Z$, and they are not null-homotopic
(otherwise we get once again a contradiction to \cite{Dehn}). One
of the key features of SBR sub-covers such as $Y$ is that given a
non-nullhomotopic loop $\C$ in its $1$-skeleton $Y^{\left(1\right)}$,
one can greedily shorten $\C$ by replacing a long block along some
$4g$-gon with its complement on the other side of this $4g$-gon.
Then, any two shortest representatives of the free homotopy class
of $\C$ can be obtained one from the other by ``half-block switches''
or a ``half-chain switch'' (see \cite[Sec.~4]{MPcore}). All these
switches take place inside $Y$. We conclude that $\C_{1}$ and $\C_{2}$
are freely-homotopic \emph{inside $Y$}, which means that $Y$ is
topologically an annulus, and $Z$ a genus-1 torus (or a Klein bottle).
This contradicts the fact that $Z$ is a covering space of a genus-$g$
surface with $g\ge2$.
\end{proof}
\begin{defn}[Matching boundary cycles]
\label{def:matching bdy cycles} We say that two different boundary
cycles $\gamma_{1}$ and $\gamma_{2}$ of a sub-cover $Y$ are \textbf{matching}
if $\left(i\right)$ there is an embedding $f\colon Y\hookrightarrow Z$
into a (full) cover $Z$ of $X_{G}$ such that one of the connected
components of $Z-f\left(Y\right)$ is an annuls bounded by $\gamma_{1}$
and $\gamma_{2}$, and $\left(ii\right)$ $\gamma_{1}$ (and $\gamma_{2}$)
do not represent the trivial element of $\Lambda_{g}$.
\end{defn}

\begin{prop}
\label{prop:no matching boundary cycles =00003D=00003D> a_0=00003D1}
Let $G=\Lambda_{g}$ and let $p\colon Y\to X_{G}$ be a compact sub-cover.
Then in the asymptotic expansion \eqref{eq:asy-exp-of-E^emb above surface},
$a_{0}\left(Y\right)=1$ if and only if $Y$ does \emph{not} admit
matching boundary cycles. 
\end{prop}

\begin{proof}
Let $f\colon Y\hookrightarrow Z$ be an embedding of $Y$ into a full
cover of $X_{G}$ with $f\left(Y\right)$ meeting every component
of $Z$. As in the proof of Lemma \ref{lem:in resolution of Y chigrp<=00003Dchigroup(Y)},
denote by $\overline{C_{1}},\ldots,\overline{C_{q}}$ the connected
components of $Z-\mathbb{Y}$. By that same proof, $\chigrp\left(Z\right)=\chigrp\left(Y\right)$
if and only if $\left(i\right)$ every connected component $Y_{i}$
of $Y$ with trivial $\plab\left(Y_{i}\right)$ is embedded in its
own connected component $Z_{j}$ of $Z$ with $\plab\left(Z_{j}\right)=1$,
and $\left(ii\right)$ every $\overline{C_{t}}$ is either a disc
or an annulus.

Now let $\remb=\remb\left(Y,\chigrp\left(Y\right)\right)$ be the
embedding-resolution from Corollary \ref{cor:injective resolution for a compact tiled surface}.
Let $\hat{f}\in\remb$ be the element taking part in the decomposition
of the universal lift $\hat{p}\colon Y\hookrightarrow\Upsilon$ from
Definition \ref{def:universal lift}. As in the proof of Corollary
\ref{cor:E^emb of arbitrary tiled surface}, $\plab\left(Z_{\hat{f}}\right)=\plab\left(Y\right)$
and $\chigrp\left(Z_{\hat{f}}\right)=\chigrp\left(Y\right)$. So this
$\hat{f}$ contributes $1$ to the coefficient $a_{0}$ from \eqref{eq:asy-exp-of-E^emb above surface}.
Note that if a pair of matching boundary cycles is realized in some
embedding of $Y$ in a full cover, then $\plab$ of the codomain strictly
contains that of $Y$ (it contains a non-trivial amalgamated product
or HNN extension of $\plab\left(Y\right)$). In particular, $\hat{p}\colon Y\hookrightarrow\Upsilon$
does not realize any pair of matching boundary cycles. 

Now assume that $a_{0}\ge2,$ namely, that there exists another element
$\hat{f}\ne g\in\remb$ with $\chigrp\left(Z_{g}\right)=\chigrp\left(Y\right)$.
Let $\overline{g}\colon Y\to\Upsilon_{g}$ be the composition of $g$
with the universal lift of $Z_{g}$ (we let $\Upsilon_{g}$ denote
the codomain of this lift) . By the uniqueness in the definition of
a resolution, $Z_{g}$ does not embed into $\Upsilon$ in a way compatible
with the universal lift $\hat{p}$. So for some component $\left(Y_{i},y_{i}\right)$
of $Y$, we must have $\plab\left(\Upsilon_{g},f\left(y_{i}\right)\right)=\plab\left(Z_{g},f\left(y_{i}\right)\right)\gneqq\plab\left(Y_{i},y_{i}\right)$.
But this can only happen, by the first paragraph of this proof applied
to $\overline{g}$, if some $\overline{C_{i}}$ is an annulus which
does not border any components of $Y$ with trivial $\plab$. This
precisely means that $\overline{g}$ realizes some pair of matching
boundary cycles.

Conversely, if $Y$ admits a pair of matching boundary cycles, we
may consider the embedding $f\colon Y\hookrightarrow Z$ from Definition
\ref{def:matching bdy cycles} that realizes this pair. Let $C$ be
the connected component of $Z-f\left(Y\right)$ which is an annulus
bounded by the matching pair. Let $Y'=f\left(Y\right)\cup C\subseteq Z$.
Then $\chigrp\left(Y'\right)=\chigrp\left(Y\right)$. Let $h'\colon Y'\hookrightarrow\Upsilon'$
be the universal lift of $Y'$, and $j\colon Y\hookrightarrow\Upsilon'$
the resulting embedding of $Y$ in $\Upsilon'$. Then by the definition
of $\remb$, $j$ decomposes through some $g\in\remb$, 
\[
Y\stackrel{g}{\hookrightarrow}Z_{g}\hookrightarrow\Upsilon'.
\]
By Lemma \ref{lem:SBR does not bound an annulus}, $Z_{g}$ must contain
$Y'$. Thus $\plab\left(Y,y\right)\lvertneqq\plab\left(Z_{g},j\left(y\right)\right)$
for any vertex $y\in Y$. But $\plab\left(Y,y\right)=\plab\left(Z_{\hat{f}},\hat{f}\left(y\right)\right)$,
and so $g\ne\hat{f}$. As $\chigrp\left(Z_{g}\right)=\chigrp\left(Y\right)$,
we obtain that $a_{0}\left(Y\right)\ge2$.
\end{proof}
\begin{cor}
\label{cor:Y with a_0=00003D1}In the following cases, a compact sub-cover
$Y$ of $X_{G}$ satisfies $a_{0}\left(Y\right)=1$:
\begin{enumerate}
\item \label{enu:a_0=00003D1 if p_1(Y) is trivial}$\plab\left(Y_{i}\right)$
is trivial for every connected component $Y_{i}$ of $Y$.
\item \label{enu:a_0=00003D1 if Y is annulus}$Y$ is a single topological
annulus.
\item \label{enu:a_0=00003D1 if Y is a union of isomorphic annuli}Y is
a disjoint union of several copies of the same topological annulus.
\item \label{enu:a_0=00003D1 if no bdy cycles are conjugate}No two different
boundary cycles $\gamma_{1}$ and $\gamma_{2}$ of $Y$ satisfy that
$\gamma_{1}$ is conjugate to $\gamma_{2}$ or to $\gamma_{2}^{-1}$.
\item \label{enu:a_0=00003D1 if Y is a union of several isomorphic annuli, no two are conjugates}Y
is a disjoint union of topological annuli, where every two are either
identical or have non-conjugate boundary cycles.
\item \label{enu:a_0=00003D1 if Y is-SBR.}$Y$ is SBR.
\end{enumerate}
\end{cor}

\begin{proof}
Any matching pair of boundary cycles consists of boundary cycles corresponding
to a non-trivial element of $\Lambda_{g}$, so part \ref{enu:a_0=00003D1 if p_1(Y) is trivial}
follows immediately from Proposition \ref{prop:no matching boundary cycles =00003D=00003D> a_0=00003D1}.
If an annulus has a matching pair of boundary cycles, then by definition,
it can be embedded in a genus-one torus, which is impossible as a
torus cannot cover $X_{G}$, and part \ref{enu:a_0=00003D1 if Y is annulus}
follows. 

Now assume that $Y$ is a disjoint union of several copies of $A$,
where the thick version $\mathbb{A}$ of $A$ is an annulus with boundary
cycles $\gamma_{1}$ and $\gamma_{2}$. If $\plab\left(A\right)=1$
we reduce to part \ref{enu:a_0=00003D1 if p_1(Y) is trivial}, so
assume otherwise. Note that $A$ has a mirror symmetry swapping $\gamma_{1}$
and $\gamma_{2}$ if and only if it is a 1-dimensional simple cycle.
Assume towards contradiction that $Y$ admits a pair of matching boundary
cycles. If this pair involves one copy of $\gamma_{1}$ and one of
$\gamma_{2}$, we can use the same annulus bounded between them to
connect the two boundary components of the same copy of $A$ and thus
obtain a torus which is a legitimate covering space of $X_{G}$ (see
\cite[Prop.~4.3]{MPcore}), which, as before, is impossible. If the
matching involves two copies of $\gamma_{1}$ then, as no non-trivial
element of $\Lambda_{g}$ is conjugate to its inverse, these two copies
of $\gamma_{1}$ bound an annulus so that they have matching orientations.
If $A$ has mirror symmetry, we return to the previous case where
the matching involves $\gamma_{1}$ and $\gamma_{2}$. Otherwise,
we get that on the same covering space of $X_{G}$ we have two copies
of $A$ with opposite orientations, which is impossible (for example,
the cyclic order of the half-edges at every vertex is determined by
the edge-labels alone -- see \cite[Prop.~3.4]{MPcore}). This shows
part \ref{enu:a_0=00003D1 if Y is a union of isomorphic annuli}.

Part \ref{enu:a_0=00003D1 if no bdy cycles are conjugate} is immediate
from Proposition \ref{prop:no matching boundary cycles =00003D=00003D> a_0=00003D1}
and the fact that the two boundary components of an annulus inside
a cover of $X_{G}$ must represent conjugates in $\Lambda_{g}$. Part
\ref{enu:a_0=00003D1 if Y is a union of several isomorphic annuli, no two are conjugates}
follows from combining the arguments of parts \ref{enu:a_0=00003D1 if Y is a union of isomorphic annuli}
and \ref{enu:a_0=00003D1 if no bdy cycles are conjugate}. Finally,
part \ref{enu:a_0=00003D1 if Y is-SBR.} is Theorem \ref{thm:E^emb of BR and SBR }\eqref{enu:E^emb for SBR}.
\end{proof}
Let us stress that part \ref{enu:a_0=00003D1 if Y is-SBR.} also falls
under the content of Proposition \ref{prop:no matching boundary cycles =00003D=00003D> a_0=00003D1}.
Indeed, if $Y$ is SBR and is embedded in a full cover $Z$, then
no connected component of $Z\setminus Y$ is an annulus bounded by
two non-nullhomotopic cycles of $Z$, by Lemma \ref{lem:SBR does not bound an annulus}.

\section{Sub-covers of $X_{\Gamma}$: expectations and asymptotic expansion
\label{sec:arbitrary sub-coverings}}

We can now prove Theorem \ref{thm:E^emb is N^chi and asym expansion}
for an arbitrary compact sub-cover $Y$ of $X_{\Gamma}$. Namely,
for $m=m\left(\Gamma\right)$, we show that there are rational numbers
$a_{t}=a_{t}\left(Y\right)$ for $t=0,-\frac{1}{m},-\frac{2}{m},-\frac{3}{m},\ldots$
so that 
\begin{equation}
\emb_{Y}\left(N\right)\asyexp N^{\chigrp\left(Y\right)}\cdot\left\{ a_{0}+a_{-1/m}N^{-1/m}+a_{-2/m}N^{-2/m}+\ldots\right\} ,\label{eq:asy exp ofarbitrary  E^emb, again}
\end{equation}
with $a_{0}\in\mathbb{Z}_{\ge1}$. 

Moreover, we need to show that whenever there are no surface groups
involved, $a_{0}=1$. We show a bit more. Recall from Definition \ref{def:sub-covering}
that $Y|_{G_{i}}$ denotes the subcomplex of $Y$ sitting above $X_{G_{i}}$
for $i=1,\ldots,k$. Also recall that Corollaries \ref{cor:E^emb over free group},
\ref{cor:E^emb above finite} and \ref{cor:E^emb of arbitrary tiled surface}
already established Theorem \ref{thm:E^emb is N^chi and asym expansion}
for sub-covers of $X_{G}$ where $G$ is any single factor of $\Gamma$. 
\begin{prop}[Addendum to Theorem \ref{thm:E^emb is N^chi and asym expansion}]
\label{prop:addendum - value of a0} For a compact sub-cover $Y$
of $X_{\Gamma}$, let $a_{0}^{\left(i\right)}\in\mathbb{Z}_{\ge1}$
denote the leading coefficient (the coefficient of $N^{\chigrp\left(Y|_{G_{i}}\right)}$)
in the asymptotic expansion of $\emb_{Y|_{G_{i}}}\left(N\right)$.
Then 
\[
a_{0}\left(Y\right)=\prod_{i=1}^{k}a_{0}^{\left(i\right)}.
\]
In particular, $a_{0}\left(Y\right)=1$ if and only if, whenever $G_{i}$
is a surface group, the subcomplex $Y|_{G_{i}}$ does not admit matching
pairs of boundary cycles. 
\end{prop}

We will need the following lemma. Recall that $o$ is the basepoint
of $X_{\Gamma}$ and $e_{i}$ is the edge connecting $o$ to $X_{G_{i}}$.
\begin{lem}
\label{lem:plab(Y) same as graph of groups} For any compact sub-cover
$p\colon Y\to X_{\Gamma}$ we have 
\begin{equation}
\chigrp\left(Y\right)=\left|p^{-1}\left(o\right)\right|+\sum_{i=1}^{k}\left(\chigrp\left(Y|_{G_{i}}\right)-\left|p^{-1}\left(e_{i}\right)\right|\right).\label{eq:chigrp(y) from chigrp of Y^(i)}
\end{equation}

\end{lem}

\begin{proof}
We may assume that $Y$ is connected: the general case follows immediately.
We embed $Y$ in a larger sub-cover $Z$, where for every $i=1,\ldots,k$,
$Y|_{G_{i}}$ is embedded in a space $Z|_{G_{i}}$ according to the
following rules:
\begin{itemize}
\item if $G_{i}$ is free, $Z|_{G_{i}}=Y|_{G_{i}}$,
\item if $G_{i}$ is finite, $f|_{G_{i}}\colon Y|_{G_{i}}\hookrightarrow Z|_{G_{i}}$
is the universal lift (Definition \ref{def:universal lift}) of $Y|_{G_{i}}$
as a sub-cover of $X_{G_{i}}$ (in particular, $Z|_{G_{i}}$ is a
compact full cover of $X_{G_{i}}$),
\item if $G_{i}$ is a surface group, $f|_{G_{i}}\colon Y|_{G_{i}}\hookrightarrow Z|_{G_{i}}$
is the element of $\remb_{Y|_{G_{i}}}=\remb\left(Y|_{G_{i}},\chigrp\left(Y|_{G_{i}}\right)\right)$
from Corollary \ref{cor:injective resolution for a compact tiled surface}
through which the universal lift of $Y|_{G_{i}}$ factors. 
\end{itemize}
We let $Z$ be the union of the $Z|_{G_{i}}$'s together with $Y-\bigcup Y|_{G_{i}}$
(attached in the obvious manner), and $f\colon Y\hookrightarrow Z$
be the embedding obtained from the identity on $Y-\bigcup Y|_{G_{i}}$
and $f|_{G_{i}}$ on $Y|_{G_{i}}$.

For every vertex $y$ in $Y|_{G_{i}}$, we claim that 
\begin{equation}
\plab\left(Y|_{G_{i}},y\right)=\plab\left(Z|_{G_{i}},f\left(y\right)\right)\cong\pi_{1}\left(Z|_{G_{i}},f\left(y\right)\right).\label{eq:plab(Y) isom pi_1(Z)}
\end{equation}
Indeed, in the free case \eqref{eq:plab(Y) isom pi_1(Z)} is trivial.
In the finite case the first equality in \eqref{eq:plab(Y) isom pi_1(Z)}
follows from the definition of the universal lift and the second one
from that $Z|_{G_{i}}$ is a full cover. Finally, in the surface case,
the same argument gives \eqref{eq:plab(Y) isom pi_1(Z)} with $Z|_{G_{i}}$
replaced with the codomain of the universal lift $\Upsilon$ of the
connected component of $y$ in $Y|_{G_{i}}$, but this implies \eqref{eq:plab(Y) isom pi_1(Z)}
as $Z|_{G_{i}}$ is BR and so its embedding in $\Upsilon$ is $\pi_{1}$-injective
\cite[Cor.~4.11]{MPcore}. 

Now $Z$ has the structure of a graph of spaces with the edge-spaces
being the ordinary edges in $\bigcup_{i=1}^{k}p^{-1}\left(e_{i}\right)$.
Fix a vertex $y_{o}\in p^{-1}\left(o\right)$. The sub-covering map
$\phi\colon Z\to X_{\Gamma}$ induces an embedding on the fundamental
group 
\[
\phi_{*}\colon\pi_{1}\left(Z,y_{o}\right)\hookrightarrow\pi_{1}\left(X_{\Gamma},o\right)=\Gamma.
\]
Indeed, every non-trivial element $g\in\pi_{1}\left(Z,y_{o}\right)$
can be described by an irreducible combinatorial path in the $1$-skeleton
of $Z$ based at $y_{o}$: this is a closed path where at each vertex-space
it may ``accumulate'' an element of that vertex group, and if the
path backtracks, the element of the vertex-group in the middle must
be non-trivial. But then the $\phi$-image of this path is irreducible
and thus non-trivial in $X_{\Gamma}$ by \eqref{eq:plab(Y) isom pi_1(Z)}.

Finally, the embedding $f\colon Y\hookrightarrow Z$ induces a \emph{surjective
}homomorphism $f_{*}\colon\pi_{1}\left(Y,y_{o}\right)\twoheadrightarrow\pi_{1}\left(Z,y_{o}\right)$:
this follows again from \eqref{eq:plab(Y) isom pi_1(Z)}. As $p_{*}=\phi_{*}\circ f_{*}$,
we conclude that $\plab\left(Y,y_{o}\right)=p_{*}\left(\pi_{1}\left(Y,y_{o}\right)\right)\cong\pi_{1}\left(Z,y_{o}\right)$.
Hence $\chigrp\left(Y\right)=\chi\left(\pi_{1}\left(Z\right)\right)$,
and the latter is equal to the right hand side of \eqref{eq:chigrp(y) from chigrp of Y^(i)}. 
\end{proof}

\begin{proof}[Proof of Theorem \ref{thm:E^emb is N^chi and asym expansion} and
of Proposition \ref{prop:addendum - value of a0}]
 Let $p\colon Y\to X_{\Gamma}$ denote the sub-covering map. Denote
by $\nu_{o}=\left|p^{-1}\left(o\right)\right|$ the number of vertices
above the vertex $o\in X_{\Gamma}$. For $i=1,\ldots,k$ denote by
$\nu_{i}=v\left(Y|_{G_{i}}\right)$ the number of vertices in $Y|_{G_{i}}$,
and by $\varepsilon_{i}=\left|p^{-1}\left(e_{i}\right)\right|$ the
number of edges projecting to the edge $e_{i}\in X_{\Gamma}$ which
connects $o$ and $X_{G_{i}}$. In an $N$-cover $\hat{X}$ of $X_{\Gamma}$
in our model, the vertices above $o$ are labeled by $\left[N\right]=\left\{ 1,\ldots,N\right\} $.
Every other vertex $u$ in $\hat{X}$ is a neighbor (in the $1$-skeleton
of $\hat{X}$) of exactly one vertex $u'$ in the fiber above $o$,
and we label $u$ by the same label from $\left[N\right]$ as $u'$.

Let $q\colon Z\to X_{G_{i}}$ be a sub-cover of some $X_{G_{i}}$.
Every embedding of the vertices of $Z$ to a cover of $X_{G_{i}}$
can be extended to an embedding of $Z$ in at most one way. Because
we identified the vertices of an $N$-cover of $X_{G_{i}}$ with $\left[N\right]$,
such an embedding of the vertices of $Z$ into an $N$-cover is an
embedding into $\left[N\right]$. Let $p\left(Z\right)$ be the probability
that a given embedding of the \emph{vertices} of $Z$ to $\left[N\right]$
can be extended to an embedding of $Z$ to a random $N$-cover with
vertices $\left[N\right]$. By symmetry, $p\left(Z\right)$ is independent
of the embedding of vertices, so 
\begin{equation}
\emb_{Z}\left(N\right)=\left(N\right)_{v\left(Z\right)}\cdot p\left(Z\right).\label{eq:e^emb vs p(emb)}
\end{equation}
Consider an arbitrary embedding $f$ of the $\nu_{o}$ vertices $p^{-1}\left(o\right)$
of $Y$ into $\left[N\right]$, out of the $\left(N\right)_{\nu_{o}}$
possible ones. For every $i=1,\ldots,k$, the embedding $f$ determines
the embedding of the $\varepsilon_{i}$ vertices of $Y|_{G_{i}}$
incident to edges projecting to $e_{i}$, so there are 
\[
\left(N-\varepsilon_{i}\right)_{\nu_{i}-\varepsilon_{i}}=\frac{\left(N\right)_{\nu_{i}}}{\left(N\right)_{\varepsilon_{i}}}
\]
possible extensions of the embedding $f$ to an embedding of the $\nu_{i}$
vertices of $Y|_{G_{i}}$. Because of the independence of the random
$N$-covers of every $X_{G_{i}}$, we obtain that
\begin{eqnarray*}
\emb_{Y}\left(N\right) & = & \left(N\right)_{\nu_{o}}\cdot\prod_{i=1}^{k}\left[\frac{\left(N\right)_{\nu_{i}}}{\left(N\right)_{\varepsilon_{i}}}p\left(Y|_{G_{i}}\right)\right]\\
 & \stackrel{\eqref{eq:e^emb vs p(emb)}}{=} & \left(N\right)_{\nu_{o}}\left[\prod_{i=1}^{k}\frac{1}{\left(N\right)_{\varepsilon_{i}}}\right]\left[\prod_{i=1}^{k}\emb_{Y|_{G_{i}}}\left(N\right)\right].
\end{eqnarray*}
We already know that each term $\emb_{Y|_{G_{i}}}\left(N\right)$
admits an asymptotic expansion with exponents in $\frac{1}{m\left(G_{i}\right)}\mathbb{Z}$
(where $m\left(G_{i}\right)=1$ for torsion-free group and $m\left(G_{i}\right)=\left|G_{i}\right|$
if $G_{i}$ is finite). Of course, the product of these asymptotic
expansions gives an asymptotic expansion with exponents in $\frac{1}{m\left(\Gamma\right)}\mathbb{Z}$.
Together with the terms $\left(N\right)_{\nu_{o}}\cdot\prod_{i=1}^{k}\frac{1}{\left(N\right)_{\varepsilon_{i}}}$
we get an asymptotic expansion as in \eqref{eq:asy exp ofarbitrary  E^emb, again},
with leading coefficient $a_{0}\defi\prod_{i}a_{0}^{\left(i\right)}$,
and with leading exponent $\nu_{o}-\sum_{i}\varepsilon_{i}+\sum_{i}\chigrp\left(Y^{\left(i\right)}\right)$,
which is equal to $\chigrp\left(Y\right)$ by Lemma \ref{lem:plab(Y) same as graph of groups}.
The final statement of Proposition \ref{prop:addendum - value of a0}
now follows from Proposition \ref{prop:no matching boundary cycles =00003D=00003D> a_0=00003D1}.
\end{proof}

\section{The limit distribution and asymptotic expansion of $\protect\fix_{\gamma}\left(N\right)$\label{sec:limit-distr-and-asymptotic-expansion of fix}}

In this section we prove Theorems \ref{thm:limit expectation of fix}
and \ref{thm:limit distribution of fix} about the limit distribution
of $\fix_{\gamma}\left(N\right)$ as $N\to\infty$ for non-torsion
$\gamma\in\Gamma$, and Theorem \ref{thm:asmptotic expansion} about
the asymptotic expansion of $\mathbb{E}\left[\fix_{\gamma}\left(N\right)\right]$
for arbitrary $\gamma\in\Gamma$. For these results, we consider a
natural sub-cover $p_{\gamma}\colon Y_{\gamma}\to X_{\Gamma}$ such
that for every $\varphi\in\Hom\left(\Gamma,S_{N}\right)$, the number
of fixed points of $\varphi\left(\gamma\right)$ is equal to the number
of lifts of $p_{\gamma}$ to $X_{\varphi}$, the corresponding $N$-cover
of $X_{\Gamma}$. We then proceed by applying Theorem \ref{thm:E^emb is N^chi and asym expansion}
to a natural resolution of $Y_{\gamma}$. All the results in this
paper are immediate for the trivial element of $\Gamma$, so we may
assume that $\gamma\ne1$.

Fix $1\ne\gamma\in\Gamma$. We may write $\gamma$ in its canonical
form as 
\begin{equation}
\gamma=h_{1}h_{2}\ldots h_{\ell\left(\gamma\right)},\label{eq:canoncial form of gamma}
\end{equation}
where $\ell=\ell\left(\gamma\right)\in\mathbb{Z}_{\ge1}$, $h_{j}\in G_{i_{j}}\setminus\left\{ 1\right\} $
and $i_{j+1}\ne i_{j}$. 

\subsubsection*{Constructing $p_{\gamma}\colon Y_{\gamma}\to X_{\Gamma}$ when $\ell\left(\gamma\right)\ge2$}

Assume that $\ell\ge2$. We may further assume without loss of generality
that $\gamma$ is cyclically reduced, namely, that $i_{\ell}\ne i_{1}$.
Indeed, replacing $\gamma$ with a conjugate does not alter any of
the local statistics of a $\gamma$-random permutation or the quantities
appearing in our results (such as $\left|\H_{\gamma}\right|$ or the
integers $t$, $\alpha_{1},\ldots,\alpha_{t}$ and $\beta_{1},\ldots,\beta_{t}$
from Theorem \ref{thm:limit distribution of fix}). 

We want to construct $p_{\gamma}\colon Y_{\gamma}\to X_{\Gamma}$
as a sub-cover with $\plab\left(Y_{\gamma}\right)$ and with the extra
property that whenever $h_{j}=h_{s}$ or $h_{j}=h_{s}^{-1}$, the
part in $Y_{\gamma}$ representing $h_{j}$ is identical, up to inversion,
to the part representing $h_{s}$. For every $j$, unless $h_{j}^{2}=1$,
we do it by expressing $h_{j}$ as a word $w_{j}$ expressing $h_{j}$
in the generators of $G_{i_{j}}$ from Section \ref{subsec:The-graph-of-spaces},
so that $w_{j}$ is a shortest word representing $h_{j}$, and $w_{j}=w_{s}$
or $w_{j}=w_{s}^{-1}$ if $h_{j}=h_{s}$ or $h_{j}=h_{s}^{-1}$, respectively.
Let $T_{j}$ be a sub-cover which is a path of length $\left|w_{j}\right|$
and such that the edges are all mapped to $X_{G_{i_{j}}}$ according
to the word $w_{j}$. Denote by $\iota\left(T_{j}\right)$ the initial
vertex and $\omega\left(T_{j}\right)$ the terminal vertex. 

If $h_{j}^{2}=1$, as we cannot have a word which is equal to its
inverse (as free words), we set $T_{j}$ to be a copy of the 1-skeleton
of the universal cover of $X_{G_{i_{j}}}$. Denote by $\iota\left(T_{j}\right)$
one arbitrary vertex of $T_{j}$ and by $\omega\left(T_{j}\right)$
another vertex so that every path from $\iota\left(T_{j}\right)$
to $\omega\left(T_{j}\right)$ descends in $X_{G_{i_{j}}}$ to a closed
path representing the element $h_{j}$. The case where $h_{j}=x$
and $G_{i_{j}}\cong C_{2}=\left\langle x\right\rangle $ is illustrated
in Figure \ref{fig:resolution}.

Finally, construct $Y_{\gamma}$ from $T_{1},\ldots,T_{\ell}$ as
follows. For every $j=1,\ldots,\ell$, connect $\omega\left(T_{j}\right)$
with $\iota\left(T_{j+1}\right)$ by a path of length two, with the
first edge mapping to $e_{j}$, the middle vertex mapping to $o$
and the second edge mapping to $e_{j+1}$ (here, if $j=\ell$, we
think of $j+1$ as $1$). Denote by $y$ the vertex mapped to $o$
in the middle of the path from $T_{\ell}$ to $T_{1}$. See the top-left
part of Figure \ref{fig:resolution}.
\begin{rem}
There are, in fact, more canonical constructions of the $T_{j}$'s
when $G=G_{i_{j}}$ is finite of a surface group. If $G$ is finite,
one could take the full universal cover of $X_{G}$ for arbitrary
$h_{j}\in G$ (or its 1-skeleton, as we did above for involutions).
If $G$ is a surface group, one could take as above a shortest word
$w_{j}$ representing $h_{j}$, consider a path representing $w_{j}$,
but then let $T_{j}$ be the ``SBR-closure'' of this path, in the
sense of \cite[Def.~4]{MPcore}. Such constructions are canonical
and do not depend on the choice of $w_{j}$. However, they make the
natural resolution of $Y_{\gamma}$ unnecessarily larger. For example,
in the example in Figure \ref{fig:resolution}, if we take a full
cover for $T_{j}$ for the element $y\in C_{4}=\left\langle y\right\rangle $
instead of a single edge, we will get five more elements in the resolution.
\end{rem}

\subsubsection*{Constructing $p_{\gamma}\colon Y_{\gamma}\to X_{\Gamma}$ when $\ell\left(\gamma\right)=1$}

Now assume that $\ell=1$ and let $G=G_{i_{1}}$. If $G$ is finite
(namely, if $\gamma$ is a torsion element), we let $Y_{\gamma}$
be the full cover of $X_{G}$ corresponding to the subgroup $\left\langle \gamma\right\rangle $
(so that $\plab\left(Y_{\gamma}\right)$ is the conjugacy class of
$\left\langle \gamma\right\rangle $). 

If $G$ is free or a surface group, then there is a unique non-power
$\gamma_{0}\in G$ so that $\gamma=\gamma_{0}^{~q}$ with $q\in\mathbb{Z}_{\ge1}$,
and the cyclic subgroups containing $\gamma$ are precisely\footnote{\label{fn:cyclic subgroups containing gamma in free and surface groups}This
fact is standard, but let us explain it for completeness: when $G$
is either free or a surface group with the generators from Section
\ref{subsec:The-graph-of-spaces}, and $w$ is any word in the generators
that is a shortest representative of its conjugacy class, then the
concatenation of $n$ copies of $w$ is also shortest in its conjugacy
class (this is immediate for free groups and follows from \cite[Lem.~2.11 ]{BirmanSeries}
for surface groups). Therefore, there is a maximal $q\in\mathbb{Z}_{\ge1}$
so that $\gamma$ has a $q$-th root in $G$. Finally, any two elements
in a free or surface group generate a free subgroup, so every two
roots of $\gamma$ belong to the same cyclic subgroup.} $\left\langle \gamma_{0}^{~j}\right\rangle $ for $1\le j|q$. As
above, we may assume by conjugating $\gamma$ if needed, that $\gamma_{0}$
is represented by some word $w_{0}$ in the generators of $G$ which
is a shortest representative of any element in the conjugacy class
of $\gamma_{0}$. We then define $w=w_{1}$ to be the concatenation
of $q$ copies of $w_{0}$. This $w$ represents $\gamma$, and is
shortest among all words representing elements in the conjugacy class
of $\gamma$: this is trivial if $G$ is free, and follows from \cite[Lem.~2.11]{BirmanSeries}
if $G$ is a surface group. We now define the sub-cover $p_{\gamma}\colon Y_{\gamma}\to X_{\Gamma}$.
Let $Y_{\gamma}$ be a cycle subdivided into $\left|w\right|$ edges.
Some vertex $y$ is mapped to the base point $v$ of $X_{G}$, and
the remaining edges are mapped by $p_{\gamma}$ to $X_{G}$ according
to the word $w$. 
\begin{rem}
It is not a priori clear that the map $p_{\gamma}\colon Y_{\gamma}\to X_{\Gamma}$
we constructed is a sub-cover. Although this is indeed true, we can
bypass the proof by saying that if this is not the case, we replace
$p_{\gamma}\colon Y_{\gamma}\to X_{\Gamma}$ with a sub-cover by lifting
it to the connected covering space of $X_{\Gamma}$ corresponding
to $\left\langle \gamma\right\rangle $ and taking the image of the
lift with the restricted covering map.
\end{rem}

\subsubsection*{A resolution of $p_{\gamma}\colon Y_{\gamma}\to X_{\Gamma}$ and
the proof of Theorem \ref{thm:asmptotic expansion}}

Consider the 'natural' resolution of the sub-cover $p_{\gamma}\colon Y_{\gamma}\to X_{\Gamma}$:
\[
\R_{\gamma}\defi\left\{ f\colon Y_{\gamma}\twoheadrightarrow Z_{f\,}\,\middle|\,f~\mathrm{is~a~surjective~morhpism~of~sub\textnormal{-}covers}\right\} .
\]
This is indeed a resolution as every morphism decomposes uniquely
to a surjective one composed with an injective one. It is finite as
$Y_{\gamma}$ is compact. Figure \ref{fig:resolution} illustrates
such a resolution.

\begin{figure}
\begin{centering}
\includegraphics[viewport=0bp 0bp 600.124bp 323.904bp,scale=0.7]{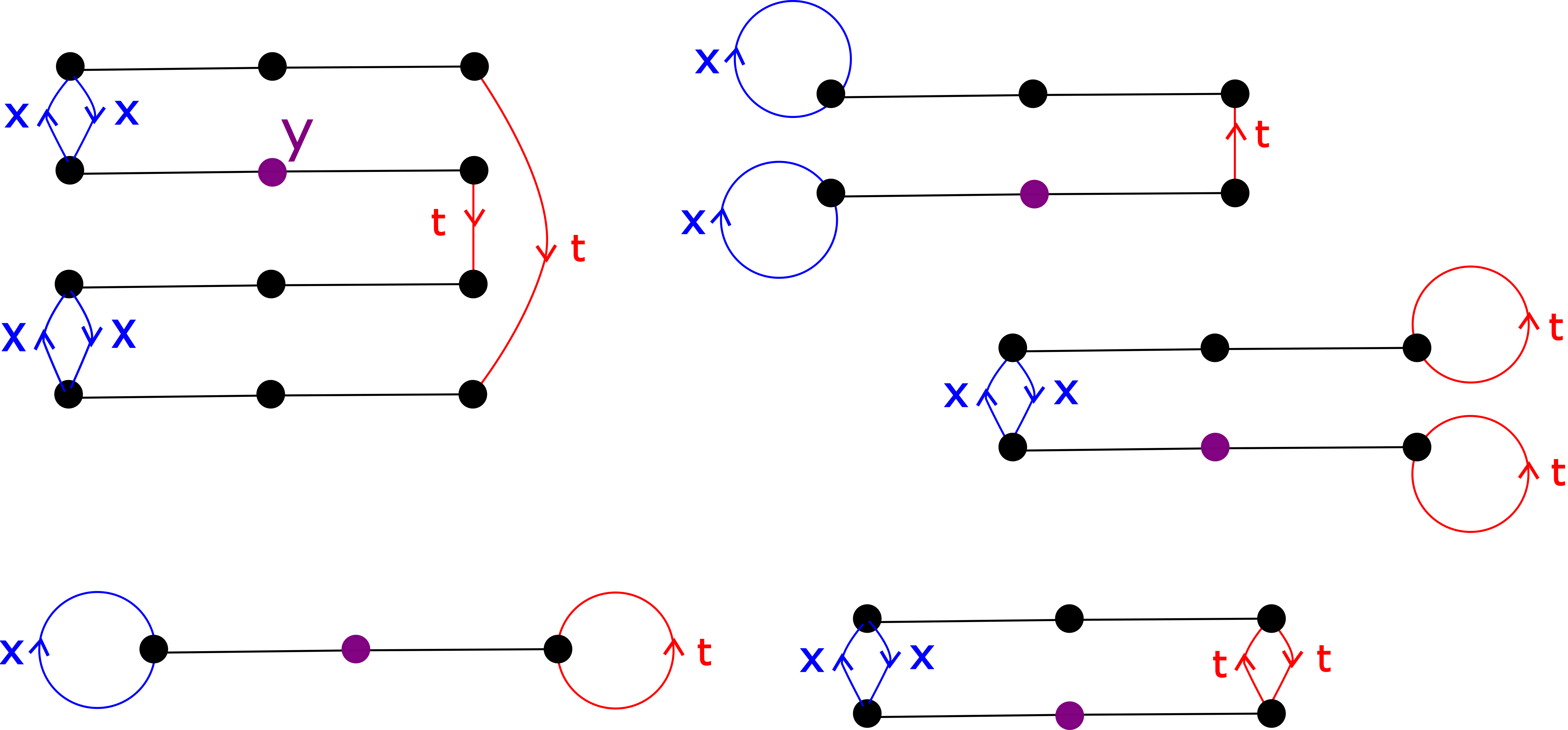}
\par\end{centering}
\caption{\label{fig:resolution}This figure illustrates $p_{\gamma}\colon Y_{\gamma}\to X_{\Gamma}$
and its resolution $\protect\R_{\gamma}$ for $\gamma=xtxt^{-1}\in C_{2}*C_{4}=\left\langle x\right\rangle *\left\langle t\right\rangle $.
The sub-cover $Y_{\gamma}$ is in the upper left part. The resolution
$\protect\R_{\gamma}$ consists of the identity map $Y_{\gamma}\to Y_{\gamma}$
together with four additional surjective morphisms of sub-covers.
In each one of them, the image of the purple vertex $y$ of $Y_{\gamma}$
is purple as well. Two of the elements in this resolution --- the
top two --- have $\protect\chigrp=0$.}
\end{figure}

\begin{proof}[Proof of Theorem \ref{thm:asmptotic expansion}]
 Let $X_{\varphi}$ be the $N$-cover of $X_{\Gamma}$ corresponding
to the uniformly random $\varphi\colon\Gamma\to S_{N}$. Recall that
vertices in every fiber of $X_{\varphi}$ are in a given bijection
with $\left[N\right]$ (this is by definition for the vertices above
$o$, and we label every other vertex in the same label as its $o$-fiber
neighbor). In the correspondence between $\Hom\left(\Gamma,S_{N}\right)$
and $N$-covers of $X_{\Gamma}$, the fixed points of $\varphi\left(\gamma\right)$
are precisely the elements $i$ in $\left[N\right]$ so that $\gamma\in\plab\left(X_{\varphi},i\right)$.
By Proposition \ref{prop:lifting criterion}, the number of fixed
points of $\varphi\left(\gamma\right)$ is thus precisely the number
of lifts of $p_{\gamma}\colon Y_{\gamma}\to X_{\Gamma}$ to $X_{\varphi}$.
Namely,
\begin{equation}
\mathbb{E}\left[\fix_{\gamma}\left(N\right)\right]=\mathbb{E}_{Y_{\gamma}}\left(N\right)=\sum_{f\in\R_{\gamma}}\emb_{Z_{f}}\left(N\right),\label{eq:fix_gamma and E_Y}
\end{equation}
where the last equality is by Lemma \ref{lem:resolution-sum-of-probabilities}.
Theorem \ref{thm:asmptotic expansion} now follows from Theorem \ref{thm:E^emb is N^chi and asym expansion}.
\end{proof}

\subsubsection*{The proof of Theorem \ref{thm:limit expectation of fix}}

We now turn to proving Theorem \ref{thm:limit expectation of fix}.
Now $\gamma$ is a non-torsion element, so either $\ell\left(\gamma\right)=1$
and $G=G_{i_{1}}$ is a free group or a surface group, or $\ell\left(\gamma\right)\ge2$.
We need to show that $\mathbb{E}\left[\fix_{\gamma}\left(N\right)\right]\underset{N\to\infty}{\to}\left|\H_{\gamma}\right|$
where $\H_{\gamma}=\left\{ H\le\Gamma\,\middle|\,\chi\left(H\right)=0~\mathrm{and}~H\ni\gamma\right\} $.
Every morphism $f\colon Y_{\gamma}\to Z_{f}$ in the resolution $\R_{\gamma}$
satisfies $\plab\left(Z_{f},f\left(y\right)\right)\ni\gamma$. (Here,
if $\ell\left(\gamma\right)=1$, we consider $\plab\left(Z_{f},f\left(y\right)\right)\le\pi_{1}\left(X_{\Gamma},v_{i_{1}}\right)$,
where $v_{i_{1}}$ is the vertex in $X_{G_{i_{1}}}$. We identify
this group with $\Gamma$ by conjugating with the edge $e_{i_{1}}$.)
In particular, $\plab\left(Z_{f},f\left(y\right)\right)$ is an infinite
subgroup of $\Gamma$ and therefore $\chigrp\left(Z_{f}\right)\le0$
(see the discussion following Definition \ref{def:EC}). Denote 
\[
\R_{\gamma}^{0}\defi\left\{ f\in\R_{\gamma}\,\middle|\,\chigrp\left(Z_{f}\right)=0\right\} .
\]
Theorem \ref{thm:E^emb is N^chi and asym expansion} and \eqref{eq:fix_gamma and E_Y}
now yield that 
\begin{equation}
\lim_{N\to\infty}\mathbb{E}\left[\fix_{\gamma}\left(N\right)\right]=\sum_{f\in\R_{\gamma}^{0}}a_{0}\left(Z_{f}\right),\label{eq:limit expectation as sum of resolution with chi=00003D0}
\end{equation}
where $a_{0}\left(Z_{f}\right)$ is the positive integer from Theorem
\ref{thm:E^emb is N^chi and asym expansion}. Consider the map 
\begin{eqnarray*}
\Psi\colon\R_{\gamma}^{0} & \to & \H_{\gamma}\\
f & \mapsto & H_{f}\defi\plab\left(Z_{f,}f\left(y\right)\right).
\end{eqnarray*}
Theorem \ref{thm:limit expectation of fix} will be proved by showing
that $\Psi$ is a bijection and that $a_{0}\left(Z_{f}\right)=1$
for all $f\in\R_{\gamma}^{0}$.

First, we show that $\Psi$ is injective. Let $f_{1},f_{2}\in\R_{\gamma}^{0}$
with $\plab\left(Z_{f_{1}},f_{1}\left(y\right)\right)=\plab\left(Z_{f_{2}},f_{2}\left(y\right)\right)$.
Then $\left(Z_{f_{1}},f_{1}\left(y\right)\right)$ and $\left(Z_{f_{2}},f_{2}\left(y\right)\right)$
have universal lifts to the \emph{same }full covering $\hat{p}_{i}\colon\left(Z_{f_{i}},f_{i}\left(y\right)\right)\to\left(\Upsilon,u\right)$
for $i=1,2$. Both are injective by Lemma \ref{lem:universal lift is injective},
and so $f_{1}$ and $f_{2}$ coincide with (the surjective part in
the decomposition of) the morphism $f\colon\left(Y,y\right)\to\left(\Upsilon,u\right)$,
and are thus identical.

For the remainder of the proof, we need the following lemma. Recall
that in any morphism $f\colon\left(Y,y\right)\to\left(Z,z\right)$
from a sub-cover to a full cover, $\plab\left(Y,y\right)\le\plab\left(Z,z\right)$.
\begin{lem}
\label{lem:Y_gamma to EC 0 is pi1-onto}If $f\colon Y_{\gamma}\to Z$
is a morphism where $Z$ is a connected full cover of $X_{\Gamma}$
with $\chigrp\left(Z\right)=0$, then it is $\pi_{1}$-surjective,
namely, $\plab\left(f\left(Y_{\gamma}\right),f\left(y\right)\right)=\plab\left(Z,f\left(y\right)\right)$.
\end{lem}

\begin{proof}
Denote by $H=\plab\left(Z,f\left(y\right)\right)\le\Gamma$. By assumption,
$H\cong\mathbb{Z}$ or $H\cong C_{2}*C_{2}$. The existence of the
morphism $f$ guarantees that $\gamma\in H$. 

Assume first that $H\cong\mathbb{Z}$. By the discussion above, there
is a unique non-power $\gamma_{0}\in\Gamma$ and $q\in\mathbb{Z}_{\ge1}$
so that $\gamma=\gamma_{0}^{~q}$ (if $\ell\left(\gamma\right)\ge2$,
then $\gamma_{0}=h_{1}h_{2}\cdots h_{\ell/q}$ and is the shortest
period in $\gamma$), and $H=\left\langle \gamma_{0}^{~d}\right\rangle $
for some $d|q$. By the way we constructed $Y_{\gamma}$, the first
$\frac{d\ell}{q}$ $T_{j}$'s (if $\ell\left(\gamma\right)\ge2$)
or the first $\frac{d\left|w\right|}{q}$ letters of $w$ (in case
$\ell\left(\gamma\right)=1$) represent $\gamma_{0}^{~j}$, and thus
the $f$-image in $Z$ of this part of $Y_{\gamma}$, contains a loop
at $f\left(y\right)$ representing $\gamma_{0}^{~j}$. We obtain $\plab\left(f\left(Y\right),f\left(y\right)\right)\ge\left\langle \gamma_{0}^{~j}\right\rangle =H$.

Now assume that $H\cong C_{2}*C_{2}$. Since $H\ni\gamma$, we are
in the case $\ell=\ell\left(\gamma\right)\ge2$. The cover $Z$, as
all covers of $X_{\Gamma}$, is a graph of spaces itself (e.g., \cite[Sec.~3]{scott1979topological}),
with trivial edge-spaces which are the preimages of $e_{1},\ldots,e_{k}$
and vertex-spaces which are the connected components of $Z|_{G_{i}}$
for $i=1,\ldots,k$ and the vertices in the fiber above $o$. Furthermore,
any decomposition of $H$ as a free product of indecomposable groups
is of the form $C_{2}*C_{2}$ (e.g., \cite[Thm.~3.5]{scott1979topological}).
Hence as a graph of spaces, $Z$ has no cycles (it is a tree of spaces),
two of its vertex spaces have $\pi_{1}\cong C_{2}$ and all remaining
vertex spaces have trivial fundamental groups. Denote by $v_{1}$
and $v_{2}$ the two vertex spaces in $Z$ with $\pi_{1}\cong C_{2}$.
Note that for every $j=1,\ldots,\ell$, the subspace $T_{j}$ of $Y_{\gamma}$
is mapped to a vertex-space $v$ of $Z$, and $\iota\left(T_{j}\right)$
and $\omega\left(T_{j}\right)$ may be mapped to the same point in
$v$ only if $v=v_{1}$ or $v=v_{2}$. So, if we follow the images
of the $T_{j}$'s one by one, we may ``backtrack'' only in $v_{1}$
or $v_{2}$. Overall, as $Y_{\gamma}$ closes up and $Z$ is a tree
of spaces, we must backtrack (cyclically) in at least two different
vertex spaces. Thus, we must backtrack in $v_{1}$ and in $v_{2}$
(at least once in each of them), read the non-trivial element in the
fundamental group in each of them, and traverse the entire path between
them. Thus $\plab\left(f\left(Y_{\gamma}\right),f\left(y\right)\right)\ge H$.
\end{proof}
We can now complete the proof of Theorem \ref{thm:limit expectation of fix}.
\begin{proof}[Proof of Theorem \ref{thm:limit expectation of fix}]
 It remains to show that $\Psi$ is surjective and that $a_{0}\left(Z_{f}\right)=1$
for all $f\in\R_{\gamma}^{0}$. Let $H\in{\cal H}_{\gamma}$ and $\left(\Upsilon_{H},u\right)$
the corresponding connected cover of $X_{\Gamma}$. By Lemma \ref{lem:Y_gamma to EC 0 is pi1-onto},
the morphism $f\colon\left(Y_{\gamma},y\right)\to\left(\Upsilon_{H},u\right)$
satisfies $\plab\left(f\left(Y_{\gamma}\right),f\left(y\right)\right)=H$.
Hence its 'surjective part' $\overline{f}\colon Y_{\gamma}\to f\left(Y_{\gamma}\right)$
is an element of ${\cal R}_{\gamma}^{0}$ with $\Psi\left(\overline{f}\right)=H$.
So $\Psi$ is surjective.

If $\ell\left(\gamma\right)\ge2$, then all vertex spaces in $\Upsilon_{H}$
have fundamental groups trivial or $C_{2}$. So all vertex spaces
of $\Upsilon_{H}$ projecting to $X_{G_{i}}$ where $G_{i}$ is a
surface group must be trivial. By Corollary \ref{cor:Y with a_0=00003D1}\eqref{enu:a_0=00003D1 if p_1(Y) is trivial},
$a_{0}\left(f\left(Y_{\gamma}\right)|_{G_{i}}\right)=1$ in this case,
by Theorem \ref{thm:E^emb is N^chi and asym expansion} $a_{0}\left(f\left(Y_{\gamma}\right)|_{G_{i}}\right)=1$
whenever $G_{i}$ is not a surface group, and Proposition \ref{prop:addendum - value of a0}
now yields that $a_{0}\left(f\left(Y_{\gamma}\right)\right)=1$.

Finally, if $\ell\left(\gamma\right)=1$ and $G=G_{i_{1}}$ is a free
or surface group\footnote{This case reduces to the results about free groups and surface groups
due to \cite{nica1994number} and \cite{magee2020asymptotic}, respectively.
We prove it here for completeness.} and $H=\langle\gamma_{0}^{~j}\rangle$, then by our choice of the
word $w=w_{1}$ above, $f\left(Y_{\gamma}\right)$ is a simple cycle:
this is trivial for $G$ a free group, and for $G$ a surface group,
$\left(\Upsilon_{H},u\right)$ is the space $\left\langle \gamma_{0}^{j}\right\rangle \backslash\tsg$,
where $w_{0}$ is a simple cycle in the $1$-skeleton (see the discussion
in \cite[Sec.~4 and 5]{MPcore}). So Proposition \ref{prop:addendum - value of a0}
and Corollary \ref{cor:Y with a_0=00003D1}\eqref{enu:a_0=00003D1 if Y is annulus}
yield that $a_{0}\left(f\left(Y_{\gamma}\right)\right)=a_{0}\left(f\left(Y_{\gamma}\right)|_{G}\right)=1$.
\end{proof}
\begin{cor}
\label{cor:H_gamma is finite}The set $\H_{\gamma}$ is finite for
every non-torsion $\gamma\in\Gamma$.
\end{cor}

\begin{proof}
This follows from the fact that $\Psi$ gives a bijection between
the finite set $\R_{\gamma}^{0}$ and $\H_{\gamma}$.
\end{proof}
We also record here the following lemma, which  we need for the proof
of Theorem \ref{thm:limit distribution of fix}. 
\begin{lem}
\label{lem:conjuage subgroups of EC=00003D0 correspond to identical sub-covers}Let
$H_{1},H_{2}\in\H_{\gamma}$ be conjugate and for $i=1,2$, $f_{i}\colon Y_{\gamma}\to Z_{i}$
the corresponding morphisms in $\R_{\gamma}^{0}$. Then $Z_{1}$ and
$Z_{2}$ are identical.
\end{lem}

(Of course, if $H_{1}\ne H_{2}$, then $f_{1}$ and $f_{2}$ map the
basepoint $y$ to two different vertices of $Z_{1}=Z_{2}$.)
\begin{proof}
If $\ell\left(\gamma\right)=1$, then there are no (non-trivial) conjugates
of $H_{1}=\langle\gamma_{0}^{~j}\rangle$ containing $\gamma$, so
the Lemma is vacuous. So assume that $\ell\left(\gamma\right)\ge2$.
Let $\Upsilon$ be the connected cover of $X_{\Gamma}$ corresponding
to the conjugacy class of $H_{1}$ and $H_{2}$, with basepoints $u_{1}$
and $u_{2}$ corresponding to $H_{1}$ and $H_{2}$, respectively.
Then $f_{i}$ is the surjective part of the morphism $\overline{f_{i}}\colon\left(Y_{\gamma},y\right)\to\left(\Upsilon,u_{i}\right)$.
Consider $\Upsilon$ as a graph of spaces. As explained in the proof
of Lemma \ref{lem:Y_gamma to EC 0 is pi1-onto}, the image of $\overline{f_{i}}$
goes precisely through the vertex-spaces and edge-spaces in the ``core''
of $\Upsilon$ (this is a cycle in the graph if $H_{1}\cong\mathbb{Z}$
or a path between the two non-$\pi_{1}$-trivial vertex spaces if
$H_{1}\cong C_{2}*C_{2}$). Our choice of the $T_{j}$'s -- that
$T_{j}$ and $T_{s}$ are identical or inverse of one another if so
are the corresponding $h_{j}$ and $h_{s}$ -- guarantees that the
precise image in every vertex space of the core is identical, and
so, indeed, $\overline{f_{1}}\left(Y_{\gamma}\right)=\overline{f_{2}}\left(Y_{\gamma}\right)$.
\end{proof}

\subsubsection*{The method of moments}

We now turn to prove Theorem \ref{thm:limit distribution of fix},
which describes the limit distribution of $\fix_{\gamma}\left(N\right)$
as $N\to\infty$ for every fixed non-torsion element $\gamma\in\Gamma$.
Our proof is based on the method of moments. Some of the steps follow
parallel steps in \cite[Sec.~4]{linial2010word}.

A probability distribution $\mu$ on $\mathbb{R}$ is said to be \emph{determined
by its moments} if it has finite moments $\alpha_{r}=\int_{-\infty}^{\infty}x^{r}\mu$$\left(dx\right)$
of all orders, and $\mu$ is the only probability measure with these
moments.
\begin{thm}[{Method of moments, e.g., \cite[Thm.~30.2]{billingsley1995probability}}]
\label{thm:method of moments} Let $X$ and $X_{N}$ $\left(N\in\mathbb{Z}_{\ge1}\right)$
be random variables, and suppose that the distribution of $X$ is
determined by its moments, that the $X_{N}$ have moments of all order,
and that $\lim_{N\to\infty}\mathbb{E}\left[X_{N}^{~r}\right]=\mathbb{E}\left[X^{r}\right]$
for every $r\in\mathbb{Z}_{\ge1}$. Then
\[
X_{N}\stackrel{\mathrm{dis}}{\to}X,
\]
where $\stackrel{\mathrm{dis}}{\to}$ denotes convergence in distribution.
\end{thm}

\begin{thm}[{Sufficient condition for $\mu$ to be determined by its moments, e.g.,
\cite[Thm.~30.1]{billingsley1995probability}}]
\label{thm:condition to be determined by moments} Let $\mu$ be
a probability measure on $\mathbb{R}$ having finite moments $\alpha_{r}=\int_{-\infty}^{\infty}x^{r}\mu$$\left(dx\right)$
of all orders. If the power series $\sum_{r}\alpha_{r}\frac{t^{r}}{r!}$
has a positive radius of convergence, then $\mu$ is determined by
its moments.
\end{thm}

Recall that Theorem \ref{thm:limit distribution of fix} states that
$\fix_{\gamma}\left(N\right)$ converges in distribution to $\sum_{i}\alpha_{i}\beta_{i}Z_{1/\beta_{i}}$
--- a finite linear combination of independent Poisson-distributed
random variables with coefficients from $\mathbb{Z}_{\ge1}$. We first
record the standard fact that such a sum is determined by its moments. 
\begin{lem}
\label{lem:linear combination of Poisson is determined by its moments}Let
$Z_{1},\ldots,Z_{t}$ be independent Poisson-distributed random variables
with parameters $\lambda_{1},\ldots,\lambda_{t}>0$, respectively,
and let\footnote{The assumption that $c_{1},\ldots,c_{t}\ge1$ is not material: it
only somewhat simplifies the notation in the proof and it holds anyway
in the case we use.} $c_{1},\ldots,c_{t}\ge1$. Then the distribution of $\sum_{i=1}^{t}c_{i}Z_{i}$
is determined by its moments.
\end{lem}

\begin{proof}
For a single Poisson distribution with parameter $\lambda$, the power
series from Theorem \ref{thm:condition to be determined by moments}
is $\sum_{r}\alpha_{r}\frac{t^{r}}{r!}=e^{\lambda\left(e^{t}-1\right)}$
\cite[Eq.~(21.22) and (21.27)]{billingsley1995probability} and, in
particular, converges for all $t$. The sum $Z_{1}+\ldots+Z_{t}$
is Poisson with parameter $\lambda_{1}+\ldots+\lambda_{t}$ and, in
particular, the corresponding power series converges for all $t$.
Let $c\defi\max\left\{ c_{1},\ldots,c_{t}\right\} $. Then $\sum Z_{i}\le\sum c_{i}Z_{i}\le c\sum Z_{i}$.
In particular, if the $r$-th moment of $\sum Z_{i}$ is $\alpha_{r}$
and of $\sum c_{i}Z_{i}$ is $\beta_{r}$, then $\alpha_{r}\le\beta_{r}\le c^{r}\alpha_{r}$.
Consequently, the series $\sum_{r}\beta_{r}\frac{t^{r}}{r!}$ has
radius of convergence that is $\ge\frac{1}{c}$ that of the series
$\sum\alpha_{r}\frac{t^{r}}{r!}$. But the latter converges for all
real $t$, hence so does $\sum\beta_{r}\frac{t^{r}}{r!}$. By Theorem
\ref{thm:condition to be determined by moments} we conclude that
$\sum c_{i}Z_{i}$ is determined by its moments. 
\end{proof}

\subsubsection*{The proof of Theorem \ref{thm:limit distribution of fix}}

Recall the statement of Theorem \ref{thm:limit distribution of fix}:
$H_{1},\ldots,H_{t}$ are representatives of the conjugacy classes
of subgroups represented in $\H_{\gamma}=\left\{ H\le\Gamma\,\middle|\,\gamma\in H~\mathrm{and}~\chi\left(H\right)=0\right\} $,
$\alpha_{i}=\left|\left\{ \H_{\gamma}\cap H_{i}^{\Gamma}\right\} \right|$
and $\beta_{i}=\left[N_{\Gamma}\left(H_{i}\right)\colon H_{i}\right]$.
We need to show that as $N\to\infty$
\[
\fix_{\gamma}\left(N\right)\stackrel{\mathrm{dis}}{\to}\sum_{i=1}^{t}\alpha_{i}\beta_{i}Z_{1/\beta_{i}},
\]
where $Z_{1/\beta_{1}},\ldots,Z_{1/\beta_{t}}$ are independent Poisson
random variables with parameters $\frac{1}{\beta_{1}},\ldots,\frac{1}{\beta_{t}}$,
respectively. For every $N$, the random variable $\fix_{\gamma}\left(N\right)$
is finitely supported and so has finite moments. By Theorem \ref{thm:method of moments}
and Lemma \ref{lem:linear combination of Poisson is determined by its moments},
it is enough to prove that for every $r\in\mathbb{Z}_{\ge1}$ we have
\begin{equation}
\mathbb{E}\left[\left(\fix_{\gamma}\left(N\right)\right)^{r}\right]\underset{N\to\infty}{\longrightarrow}\mathbb{E}\left[\left(\sum_{i=1}^{t}\alpha_{i}\beta_{i}Z_{1/\beta_{i}}\right)^{r}\right].\label{eq:goal-in-terms-of-moments}
\end{equation}

Recall that for every $\varphi\in\Hom\left(\Gamma,S_{N}\right)$,
the number of fixed points of $\varphi\left(\gamma\right)$ is equal
to the number of lifts of $p_{\gamma}\colon Y_{\gamma}\to X_{\Gamma}$
to $X_{\varphi}$. Similarly, define
\[
Y_{\gamma}^{\sqcup r}\defi\underbrace{Y_{\gamma}\sqcup\ldots\sqcup Y_{\gamma}}_{r~\mathrm{times}}.
\]
Then $\left(\fix\left(\varphi\left(\gamma\right)\right)\right)^{r}$
is equal to $\left(\#~\mathrm{lifts~of}~p_{\gamma}\right)^{r}$, which
is equal to the number of lifts to $X_{\varphi}$ of 
\[
p_{\gamma}^{\sqcup r}\colon Y_{\gamma}^{\sqcup r}\to X_{\Gamma},
\]
where $p_{\gamma}^{~\sqcup r}$ restricts to $p_{\gamma}$ on each
of the $r$ disjoint copies of $Y_{\gamma}$. So
\begin{equation}
\mathbb{E}\left[\left(\fix_{\gamma}\left(N\right)\right)^{r}\right]=\mathbb{E}_{Y_{\gamma}^{\sqcup r}}\left(N\right)=\sum_{f\in\R_{\gamma,r}}\emb_{Z_{f}}\left(N\right),\label{eq:r-th moment by lifts}
\end{equation}
where $\R_{\gamma,r}$ is the standard resolution of $p_{\gamma}^{\sqcup r}$:
\[
\R_{\gamma,r}\defi\left\{ f\colon Y_{\gamma}^{\sqcup r}\twoheadrightarrow Z_{f\,}\,\middle|\,f~\mathrm{is~a~surjective~morhpism~of~sub\textnormal{-}covers}\right\} .
\]
As for the case $r=1$, if we let 
\[
\R_{\gamma,r}^{0}\defi\left\{ f\in\R_{\gamma,r}\,\middle|\,\chigrp\left(Z_{f}\right)=0\right\} ,
\]
then 
\begin{equation}
\lim_{N\to\infty}\mathbb{E}\left[\fix_{\gamma}\left(N\right)^{r}\right]=\sum_{f\in\R_{\gamma,r}^{0}}a_{0}\left(Z_{f}\right).\label{eq:limit of moments of fix with resolution}
\end{equation}
Let $f_{1},\ldots,f_{t}\in\R_{\gamma}^{0}$ be the morphisms $f_{i}\colon Y_{\gamma}\to Z_{i}=Z_{f_{i}}$
corresponding through the bijection $\Psi$ to $H_{1},\ldots,H_{t}$
from the statement of Theorem \ref{thm:limit distribution of fix},
respectively. Namely, $f_{i}$ is onto and $\plab\left(Z_{i},f_{i}\left(y\right)\right)=H_{i}$. 
\begin{lem}
\label{lem:every chi=00003D0 element of the resolution is a disjoint union of copies of representatives}
For every $f\in\R_{\gamma,r}^{0}$, the sub-cover $Z_{f}$ is a disjoint
union of copies of $Z_{1},\ldots,Z_{t}$. 
\end{lem}

\begin{proof}
Let $f\in\R_{\gamma,r}^{0}$. Every connected component of $Z_{f}$
contains $\gamma$ in its fundamental group, up to conjugation. Hence
every connected component has non-positive $\chigrp$. We conclude
that they all have $\chigrp=0$.

Let $C$ be a connected component of $Z_{f}$. Consider the universal
lift $\Upsilon$ of $C$. Let $g\colon Y_{\gamma}\to C$ be the restriction
of $f$ to one of the connected components of $Y_{\gamma}^{\sqcup r}$
mapped to $C$. By Lemma \ref{lem:Y_gamma to EC 0 is pi1-onto}, the
map $g\colon Y_{\gamma}\to\Upsilon$ is $\pi_{1}$-surjective. This
shows that all components of $Y_{\gamma}^{\sqcup r}$ mapped to $C$
have $\pi_{1}$-images which are identical, up to conjugation. By
Lemma \ref{lem:conjuage subgroups of EC=00003D0 correspond to identical sub-covers}
$C$ is identical to one of the $Z_{i}$'s ($1\le i\le t$).
\end{proof}

\begin{proof}[Proof of Theorem \ref{thm:limit distribution of fix}]
 Recall \eqref{eq:goal-in-terms-of-moments}: it is enough to prove
that $\mathbb{E}\left[\left(\fix_{\gamma}\left(N\right)\right)^{r}\right]\underset{N\to\infty}{\longrightarrow}\mathbb{E}\left[\left(\sum_{i=1}^{t}\alpha_{i}\beta_{i}Z_{1/\beta_{i}}\right)^{r}\right]$
for every $r\in\mathbb{Z}_{\ge1}$. By \eqref{eq:limit of moments of fix with resolution},
$\lim_{N\to\infty}\mathbb{E}\left[\fix_{\gamma}\left(N\right)^{r}\right]=\sum_{f\in\R_{\gamma,r}^{0}}a_{0}\left(Z_{f}\right)$.
By Lemma \ref{lem:every chi=00003D0 element of the resolution is a disjoint union of copies of representatives},
for every $f\in\R_{\gamma,r}^{0}$, the sub-cover $Z_{f}$ is a disjoint
union of copies of $Z_{1},\ldots,Z_{t}$. By Corollary \ref{cor:Y with a_0=00003D1}\eqref{enu:a_0=00003D1 if Y is a union of several isomorphic annuli, no two are conjugates}
and Proposition \ref{prop:addendum - value of a0}, $a_{0}\left(Z_{f}\right)=1$
for all $f\in\R_{\gamma,r}^{0}$. With Lemma \ref{lem:every chi=00003D0 element of the resolution is a disjoint union of copies of representatives}
we now obtain 
\[
\lim_{N\to\infty}\mathbb{E}\left[\fix_{\gamma}\left(N\right)^{r}\right]=\sum_{\substack{r_{1}+\ldots+r_{t}=r\\
r_{i}\ge0
}
}\binom{r}{r_{1}~\ldots~r_{t}}\prod_{i=1}^{t}\#\left\{ \mathrm{surjective~maps~}Y_{\gamma}^{\sqcup r_{i}}\twoheadrightarrow Z_{i}^{\sqcup s_{i}},~0\le s_{i}\le r_{i}\right\} .
\]
By assumption, the variables $Z_{1/\beta_{1}},\ldots,Z_{1/\beta_{t}}$
are independent, and so 
\[
\mathbb{E}\left[\left(\sum_{i=1}^{t}\alpha_{i}\beta_{i}Z_{1/\beta_{i}}\right)^{r}\right]=\sum_{\substack{r_{1}+\ldots+r_{t}=r\\
r_{i}\ge0
}
}\binom{r}{r_{1}~\ldots~r_{t}}\prod_{i=1}^{t}\mathbb{E}\left[\left(\alpha_{i}\beta_{i}Z_{1/\beta_{i}}\right)^{r_{i}}\right],
\]
so it is enough to show that for all $r\in\mathbb{Z}_{\ge0}$ and
$i\in\left\{ 1,\ldots,t\right\} $, we have 
\begin{equation}
\#\left\{ \mathrm{surjective~maps~}Y_{\gamma}^{\sqcup r}\twoheadrightarrow Z_{i}^{\sqcup s},~0\le s\le r\right\} =\mathbb{E}\left[\left(\alpha_{i}\beta_{i}Z_{1/\beta_{i}}\right)^{r}\right].\label{eq:equality for a given Z_i}
\end{equation}
Both sides equal $1$ when $r=0$, so assume that $r\ge1$. Recall
that $\left\{ \begin{array}{c}
r\\
s
\end{array}\right\} $ denotes a Stirling number of the second kind, and is equal to the
number of ways to partition a set of $r$ objects into $s$ non-empty
subsets. The left hand side of \eqref{eq:equality for a given Z_i}
is equal to 
\begin{equation}
\sum_{s=1}^{r}\#\left\{ \mathrm{surjective~maps~}Y_{\gamma}^{\sqcup r}\twoheadrightarrow Z_{i}^{\sqcup s}\right\} =\left(\alpha_{i}\beta_{i}\right)^{r}\sum_{j=1}^{r}\left\{ \begin{array}{c}
r\\
s
\end{array}\right\} \cdot\frac{1}{\beta_{i}^{~s}}.\label{eq:lhs}
\end{equation}
Indeed, a surjective map $Y_{\gamma}^{\sqcup r}\twoheadrightarrow Z_{i}^{\sqcup s}$
is determined by a partition of the $r$ copies of $Y_{\gamma}$ into
$s$ non-empty subsets. For each subset, we map one element $\left(Y_{\gamma},y\right)$
to one of $\alpha_{i}$ non-isomorphic possible base points $u$ in
$Z_{i}$ so that $\plab\left(Z_{i},u\right)\ni\gamma$. As $\beta_{i}$
is the number of automorphisms of $Z_{i}$, we get that each remaining
element of the subset now has $\alpha_{i}\beta_{i}$ possibilities
for the image-vertex of $y$. Together, these images of $y$ completely
determine the map $Y_{\gamma}^{\sqcup r}\twoheadrightarrow Z_{i}^{\sqcup s}$,
and the total number of options is $\left(\alpha_{i}\beta_{i}\right)^{r-s}\cdot\alpha_{i}^{~s}=\left(\alpha_{i}\beta_{i}\right)^{r}\cdot\beta_{i}^{~-s}$.

On the other hand, the right hand side of \eqref{eq:equality for a given Z_i}
is $\left(\alpha_{i}\beta_{i}\right)^{r}\mathbb{E}\left[\left(Z_{1/\beta_{i}}\right)^{r}\right]$,
and it is a standard fact about the moments of Poisson variables that
\[
\mathbb{E}\left[\left(Z_{\lambda}\right)^{r}\right]=\sum_{s=1}^{r}\left\{ \begin{array}{c}
r\\
s
\end{array}\right\} \cdot\lambda^{s}.
\]
\end{proof}

\section{Asymptotic independence and statistics of small cycles\label{sec:Asymptotic-independence-and-small-cycles}}

In this section we prove the remaining results: Theorem \ref{thm:asymptotic independence}
giving a precise condition on when $\fix_{\gamma_{1}}\left(N\right)$
and $\fix_{\gamma_{2}}\left(N\right)$ are asymptotically independent
for non-torsion $\gamma_{1},\gamma_{2}\in\Gamma$, and Theorem \ref{thm:asymptotic number of cyclces of bounded size}
about the statistics of cycles of bounded size. These two results
are proven along similar lines to the results proven in Section \ref{sec:limit-distr-and-asymptotic-expansion of fix},
so we stress mostly some crucial points that did not appear in the
previous results.

\subsubsection*{The multivariate method of moments}

In both results we need the following classical extension of Theorem
\ref{thm:method of moments}:
\begin{thm}[{Multivariate method of moments, e.g., \cite[Exer.~30.6]{billingsley1995probability} }]
\label{thm:multivariate method of moments} Let $X^{\left(1\right)},\ldots,X^{\left(p\right)}$
and $X_{N}^{\left(1\right)},\ldots,X_{N}^{\left(p\right)}$ $\left(N\in\mathbb{Z}_{\ge1}\right)$
be random variables, and suppose that the distribution of $X^{\left(1\right)},\ldots,X^{\left(p\right)}$
on $\mathbb{R}^{p}$ is determined by its moments (see \cite[Exer.~30.5]{billingsley1995probability}
for the definition), that the $X_{N}^{\left(i\right)}$ have moments
of all order, and that 
\[
\lim_{N\to\infty}\mathbb{E}\left[\left(X_{N}^{~\left(1\right)}\right)^{r_{1}}\cdots\left(X_{N}^{~\left(p\right)}\right)^{r_{p}}\right]=\mathbb{E}\left[\left(X^{\left(1\right)}\right)^{r_{1}}\cdots\left(X^{\left(p\right)}\right)^{r_{p}}\right]
\]
 for every $r_{1},\ldots,r_{p}\in\mathbb{Z}_{\ge0}$. Then
\[
\left(X_{N}^{\left(1\right)},\ldots,X_{N}^{\left(p\right)}\right)\stackrel{\mathrm{dis}}{\to}\left(X^{\left(1\right)},\ldots,X^{\left(p\right)}\right).
\]
In particular, if $X^{\left(1\right)},\ldots,X^{\left(p\right)}$
are independent, then $X_{N}^{\left(1\right)},\ldots,X_{N}^{\left(p\right)}$
are asymptotically independent. 
\end{thm}

\subsubsection*{Proof of Theorem \ref{thm:asymptotic independence}}

Recall that we are given two non-torsion elements $\gamma_{1},\gamma_{2}\in\Gamma$,
and we need to show that the following three conditions are equivalent:
$\left(i\right)$ $\fix_{\gamma_{1}}\left(N\right)$ and $\fix_{\gamma_{2}}\left(N\right)$
are asymptotically independent, $\left(ii\right)$ $\gamma_{1}$ and
$\gamma_{2}$ cannot be both conjugated into the same EC-zero subgroup
of $\Gamma$, and $\left(iii\right)$ $\mathbb{E}\left[\fix_{\gamma_{1}}\left(N\right)\cdot\fix_{\gamma_{2}}\left(N\right)\right]=\mathbb{E}\left[\fix_{\gamma_{1}}\left(N\right)\right]\cdot\mathbb{E}\left[\fix_{\gamma_{2}}\left(N\right)\right]+O\left(N^{-1/m}\right)$.
\begin{proof}[Proof of Theorem \ref{thm:asymptotic independence} ]
 We start by proving $\left(ii\right)\Longrightarrow\left(i\right)$.
So we assume that $\gamma_{1}$ and $\gamma_{2}$ cannot be both conjugated
into the same EC-zero subgroup of $\Gamma$. For $j=1,2$, denote
by $Y_{j}$ a random variable distributed as the linear combination
of Poissons from Theorem \ref{thm:limit distribution of fix} corresponding
to $\fix_{\gamma_{j}}\left(N\right)$. By Theorem \ref{thm:multivariate method of moments},
it is enough to show that 
\begin{equation}
\lim_{N\to\infty}\mathbb{E}\left[\left(\fix_{\gamma_{1}}\left(N\right)\right)^{r_{1}}\left(\fix_{\gamma_{2}}\left(N\right)\right)^{r_{2}}\right]=\mathbb{E}\left[\left(Y_{1}\right)^{r_{1}}\right]\cdot\mathbb{E}\left[\left(Y_{2}\right)^{r_{2}}\right]\label{eq:goal-thm-1.12}
\end{equation}
for every $r_{1},r_{2}\in\mathbb{Z}_{\ge0}$. Let $\R_{\gamma_{1},r_{1},\gamma_{2},r_{2}}$
be the natural resolution of $p\colon Y_{\gamma_{1}}^{\sqcup r_{1}}\sqcup Y_{\gamma_{2}}^{\sqcup r_{2}}$,
and $\R_{\gamma_{1},r_{1},\gamma_{2},r_{2}}^{0}$ the subset of morphisms
$f\in\R_{\gamma_{1},r_{1},\gamma_{2},r_{2}}$ with $\chigrp\left(Z_{f}\right)=0$.
As in the proof of Theorem \ref{thm:limit distribution of fix}, 
\[
\lim_{N\to\infty}\mathbb{E}\left[\left(\fix_{\gamma_{1}}\left(N\right)\right)^{r_{1}}\left(\fix_{\gamma_{2}}\left(N\right)\right)^{r_{2}}\right]=\sum_{f\in\R_{\gamma_{1},r_{1},\gamma_{2},r_{2}}^{0}}a_{0}\left(Z_{f}\right)=\left|\R_{\gamma_{1},r_{1},\gamma_{2},r_{2}}^{0}\right|.
\]
The assumption on $\gamma_{1}$ and $\gamma_{2}$ guarantees that
in every $f\in\R_{\gamma_{1},r_{1},\gamma_{2},r_{2}}^{0}$, the copies
of $Y_{\gamma_{1}}$ and those of $Y_{\gamma_{2}}$ are mapped to
disjoint connected components of $Z_{f}$. Thus 
\[
\left|\R_{\gamma_{1},r_{1},\gamma_{2},r_{2}}^{0}\right|=\left|\R_{\gamma_{1},r_{1}}^{0}\right|\cdot\left|\R_{\gamma_{2},r_{2}}^{0}\right|=\mathbb{E}\left[\left(Y_{1}\right)^{r_{1}}\right]\cdot\mathbb{E}\left[\left(Y_{2}\right)^{r_{2}}\right].
\]

The implication $\left(i\right)\Longrightarrow\left(iii\right)$:
from the very definition of asymptotic independence it follows that
$\mathbb{E}\left[\fix_{\gamma_{1}}\left(N\right)\cdot\fix_{\gamma_{2}}\left(N\right)\right]=\mathbb{E}\left[\fix_{\gamma_{1}}\left(N\right)\right]\cdot\mathbb{E}\left[\fix_{\gamma_{2}}\left(N\right)\right]+o_{N}\left(1\right)$.
Applying Theorem \ref{thm:E^emb is N^chi and asym expansion} to $p_{\gamma_{1}}\colon Y_{\gamma_{1}}\to X_{\Gamma}$,
to $p_{\gamma_{2}}\colon Y_{\gamma_{2}}\to X_{\Gamma}$ and to $p_{\gamma_{1}\sqcup\gamma_{2}}\colon Y_{\gamma_{1}}\sqcup Y_{\gamma_{2}}\to X_{\Gamma}$,
shows that the error term is $O\left(N^{-1/m}\right)$.

The implication $\left(iii\right)\Longrightarrow\left(ii\right)$:
Finally, assume that $\mathbb{E}\left[\fix_{\gamma_{1}}\left(N\right)\cdot\fix_{\gamma_{2}}\left(N\right)\right]=\mathbb{E}\left[\fix_{\gamma_{1}}\left(N\right)\right]\cdot\mathbb{E}\left[\fix_{\gamma_{2}}\left(N\right)\right]+O\left(N^{-1/m}\right)$.
As in the proof of Theorem \ref{thm:limit expectation of fix}, 
\[
\lim_{N\to\infty}\mathbb{E}\left[\fix_{\gamma_{1}}\left(N\right)\cdot\fix_{\gamma_{2}}\left(N\right)\right]=\left|\R_{\gamma_{1}\sqcup\gamma_{2}}^{0}\right|,
\]
the number of elements $f$ in the natural resolution of $p_{\gamma_{1}\sqcup\gamma_{2}}$
with $\chigrp\left(Z_{f}\right)=0$. Inside $\R_{\gamma_{1}\sqcup\gamma_{2}}^{0}$
there are all those morphisms in which $Y_{\gamma_{1}}$ and $Y_{\gamma_{2}}$
are mapped to two different connected components of $Z_{f}$. The
number of such elements is $\left|\R_{\gamma_{1}}^{0}\right|\cdot\left|\R_{\gamma_{1}}^{0}\right|$.
By the assumption in $\left(iii\right)$, there are no further elements
in $\R_{\gamma_{1}\sqcup\gamma_{2}}^{0}$. 

Assume towards contradiction that $\gamma_{1}$ and $\gamma_{2}$
are both conjugate into the same EC-zero subgroup $H\le\Gamma$. In
particular, $\ell\left(\gamma_{1}\right)\ge2$ if and only if $\ell\left(\gamma_{2}\right)\ge2$.
We may assume that $Y_{\gamma_{1}}$ and $Y_{\gamma_{2}}$ were constructed
in a coordinated manner: if $\ell\left(\gamma_{1}\right)\ge2$, then
the $T_{j}$'s constituting both $Y_{\gamma_{1}}$ and $Y_{\gamma_{2}}$
are identical whenever the corresponding elements of the canonical
forms are identical or inverse of one another, and if $\ell\left(\gamma_{1}\right)=1$,
then we take $\gamma_{1}$ and $\gamma_{2}$ to be powers of the same
$\gamma_{0}$. 

But then there is a map of $Y_{\gamma_{1}}\sqcup Y_{\gamma_{2}}$
into the connected cover $\Upsilon_{H}$ corresponding to $H$. The
images of both $Y_{\gamma_{1}}$ and $Y_{\gamma_{2}}$ in $\Upsilon_{H}$
are identical, and the surjective part of this morphism constitutes
another element of $\R_{\gamma_{1}\sqcup\gamma_{2}}^{0}$, a contradiction.
\end{proof}

\subsubsection*{Proof of Theorem \ref{thm:asymptotic number of cyclces of bounded size}}

We will need the following lemma. Recall that for non-torsion $\gamma\in\Gamma$
and any $L\in\mathbb{Z}_{\ge1}$, we let $\H_{\gamma,L}$ mark the
set of EC zero subgroups of $\Gamma$ containing $\gamma^{L}$ but
not any smaller power of $\gamma$.
\begin{lem}
\label{lem:different powers of gamma are not conjugate in EC 0} Let
$L_{1}\ne L_{2}$ be positive integer. Then no subgroup in $\H_{\gamma,L_{1}}$
is conjugate to a subgroup in $\H_{\gamma,L_{2}}$.
\end{lem}

\begin{proof}
As mentioned in the proof of Lemma \ref{lem:Y_gamma to EC 0 is pi1-onto},
for every non-torsion $\gamma\in\Gamma$ there is a unique non-power
$\gamma_{0}$ and $q\in\mathbb{Z}_{\ge1}$ with $\gamma=\gamma_{0}^{~q}$.
We first show the claim of the lemma is true for subgroup isomorphic
to $\mathbb{Z}$. The subgroups isomorphic to $\mathbb{Z}$ in $\bigcup_{L}\H_{\gamma,L}$
are precisely $\left\{ \left\langle \gamma_{0}^{~j}\right\rangle \,\middle|\,j\in\mathbb{Z}_{\ge1}\right\} $,
no distinct two of which are conjugate one to the other. Moreover,
for every $j\in\mathbb{Z}_{\ge1}$, $\left\langle \gamma_{0}^{~j}\right\rangle $
belongs to a single $\H_{\gamma,L}$: exactly the $L$ satisfying
that it is the smallest positive integer with $j\mid qL$.

For subgroups isomorphic to $C_{2}*C_{2}$, the argument is similar,
as we now explain. Let $H\le\Gamma$ with $H\cong C_{2}*C_{2}$. If
$\gamma^{L}\in H$ for some $L\in\mathbb{Z}_{\ge1}$, then $\ell\left(\gamma\right)\ge2$.
Let $f\colon Y_{\gamma^{L}}\to Z_{f}$ be the element of $\R_{\gamma^{L}}^{0}$
corresponding to $H$. By the analysis in the proof of Lemma \ref{lem:Y_gamma to EC 0 is pi1-onto},
as a graph of spaces, $Z_{f}$ is a path of vertex-spaces with trivial
groups, between two vertex-spaces representing order two subgroups.
Assume that the path, excluding the vertex-spaces at the two ends,
consists of $s$ ``vertex spaces'', namely, it spells out an element
$\delta\in\Gamma$ with $\ell\left(\delta\right)=s$. Assume that
$H=\plab\left(Z_{f},u\right)$ for some vertex $u$. Because $\gamma$
is cyclically reduced, the closed path at $u$ corresponding to $\gamma^{L}$
starts by leaving $u$ to one direction (say, to the right), and ends
by arriving to $u$ from the other direction (say, from the left).
Thus $\ell\left(\gamma^{L}\right)=L\cdot\ell\left(\gamma\right)=Lq\cdot\ell\left(\gamma_{0}\right)$
is equal to some multiple of $\left(2+2s\right)$. So knowing that
$H\in\H_{\gamma,L}$ for some $L$, we may find $L$ simply as the
smallest positive integer $L$ satisfying that $2+2s\mid Lq\cdot\ell\left(\gamma_{0}\right)$.
Any conjugate of $H$ has the same parameter $s$ (it is the same
graph-of-spaces, only, possibly, with a different basepoint), so if
it belongs to any $\H_{\gamma,L}$, it must belong to the same $\H_{\gamma,L}$
as $H$ does. 
\end{proof}

\begin{proof}[Proof of Theorem \ref{thm:asymptotic number of cyclces of bounded size}]
 We begin with the first part of the theorem: that $\mathbb{E}\left[\cyc_{\gamma,L}\left(N\right)\right]=\frac{1}{L}\left|\H_{\gamma,L}\right|+O\left(N^{-1/m}\right)$.
Note that 
\[
\fix_{\gamma^{L}}\left(N\right)=\sum_{1\le d|L}d\cdot\cyc_{\gamma,d}\left(N\right),
\]
so the this part of the theorem follows from Theorem \ref{thm:limit expectation of fix}
by a simple induction on $L$. Indeed, when $L=1$ this is precisely
Theorem \ref{thm:limit expectation of fix}. For general $L$, we
get by induction that
\begin{eqnarray*}
L\cdot\mathbb{E}\left[\cyc_{\gamma,L}\left(N\right)\right] & = & \mathbb{E}\left[\fix_{\gamma^{L}}\left(N\right)\right]-\sum_{1\le d<L,d|L}d\cdot\mathbb{E}\left[\cyc_{\gamma,d}\left(N\right)\right]\\
 & = & \left|\H_{\gamma^{L}}\right|-\sum_{1\le d<L,d|L}\left|\H_{\gamma,d}\right|+O\left(N^{-1/m}\right)=\left|\H_{\gamma,L}\right|+O\left(N^{-1/m}\right).
\end{eqnarray*}
For the second part of Theorem \ref{thm:asymptotic number of cyclces of bounded size},
recall that $H_{1},\ldots,H_{t}$ are \emph{representatives of the
conjugacy classes of subgroups }represented in ${\cal H}_{\gamma,L}$,
and $\alpha_{1},\ldots,\alpha_{t}$ and $\beta_{1},\ldots,\beta_{t}$
are defined analogously to their definition in Theorem \ref{thm:limit distribution of fix}.
By Lemma \ref{lem:different powers of gamma are not conjugate in EC 0},
$H_{1},\ldots,H_{t}$ can be taken to be a subset of the representatives
of $\H_{\gamma^{L}}$ from Theorem \ref{thm:limit distribution of fix},
which are not conjugate to any element of $\H_{\gamma^{L'}}$ with
$L'<L$. 

This part of Theorem \ref{thm:asymptotic number of cyclces of bounded size}
states that $\cyc_{\gamma,L}\left(N\right)\stackrel[N\to\infty]{\mathrm{dis}}{\longrightarrow}\frac{1}{L}\sum_{i=1}^{t}\alpha_{i}\beta_{i}Z_{1/\beta_{i}}$.
As in the first part, this can be deduced from Theorem \ref{thm:limit distribution of fix}
applied to $\fix_{\gamma^{L}}\left(N\right)$ by a simple induction
on $L$. Indeed, 
\begin{eqnarray*}
L\cdot\cyc_{\gamma,L}\left(N\right) & = & \fix_{\gamma^{L}}\left(N\right)-\sum_{1\le d<L,d|L}d\cdot\cyc_{\gamma,d}\left(N\right)
\end{eqnarray*}
and applying Theorem \ref{thm:limit distribution of fix} to $\fix_{\gamma^{L}}\left(N\right)$
and the induction hypothesis on $d\cdot\cyc_{\gamma,d}\left(N\right)$
for $1\le d<L,d|L$, we obtain the result.

Finally, the third part of Theorem \ref{thm:asymptotic number of cyclces of bounded size}
states that $\cyc_{\gamma,1}\left(N\right),\cyc_{\gamma,2}\left(N\right),\ldots,\cyc_{\gamma,L}\left(N\right)$
are asymptotically independent, and that for $L_{1}\ne L_{2}$ we
have $\mathbb{E}\left[\cyc_{\gamma,L_{1}}\left(N\right)\cdot\cyc_{\gamma,L_{2}}\left(N\right)\right]=\mathbb{E}\left[\cyc_{\gamma,L_{1}}\left(N\right)\right]\cdot\mathbb{E}\left[\cyc_{\gamma,L_{2}}\left(N\right)\right]+O\left(N^{-1/m}\right)$.
The argument here is the same as in the proof of Theorem \ref{thm:asymptotic independence},
where Lemma \ref{lem:different powers of gamma are not conjugate in EC 0}
replaces the assumption in part \ref{enu:gamma1 and gamma2 are not conjugate in a EC=00003D0 subgroup}
of Theorem \ref{thm:asymptotic independence}.
\end{proof}

\section{Open questions\label{sec:Open-questions}}

There are several questions the current paper raises. We discuss here
two we find most appealing.

\subsubsection*{The leading term of $\mathbb{E}\left[\protect\fix_{\gamma}\left(N\right)\right]-1$}

Recall that when $\Gamma$ is a free group and $1\ne\gamma\in\Gamma$,
Theorem \ref{thm:limit expectation of fix}, which is originally due
to \cite{nica1994number} in this case, says that $\mathbb{E}\left[\fix_{\gamma}\left(N\right)\right]=d\left(q\right)+O\left(N^{-1}\right)$,
where $q\in\mathbb{Z}_{\ge1}$ is maximal so that $\gamma$ is a $q$-th
power, and $d\left(q\right)$ the number of positive divisors of $q$.
In particular, when $\gamma$ is a non-power, then $\mathbb{E}\left[\fix_{\gamma}\left(N\right)\right]=1+O\left(N^{-1}\right)$.
But, in fact, much more is known. For $\gamma$ in a free group $\Gamma$,
denote 
\begin{equation}
\chimax\left(\gamma\right)\defi\max\left\{ \chi\left(H\right)\,\middle|\,\gamma\in H\le\Gamma,~\gamma~\mathrm{non\textnormal{-}primitive~in}~H\right\} ,\label{eq:pi(gamma)}
\end{equation}
and let $\crit\left(\gamma\right)$ denote the number of subgroups
attaining the maximum from \eqref{eq:pi(gamma)}. Then \cite[Thm.~1.8]{PP15}
states that for every $\gamma\in\Gamma$, 
\[
\mathbb{E}\left[\fix_{\gamma}\left(N\right)\right]-1=\left|\crit\left(\gamma\right)\right|\cdot N^{\chimax\left(\gamma\right)}\left(1+O\left(N^{-1}\right)\right).
\]
Notice that this estimate is true for proper powers as well and even
for the identity element. We conjecture that the same phenomenon is
true for the family of groups considered in this paper.
\begin{conjecture}
\label{conj:pi-conjecture}In the notation of Assumption \ref{assu:Gamma},
let $\gamma\in\Gamma$ be a non-torsion element. Denote 
\[
\chimax\left(\gamma\right)\defi\max\left\{ \chi\left(H\right)\,\middle|\,\begin{gathered}\gamma\in H\le\Gamma,~\mathrm{and}\\
\left\langle \gamma\right\rangle ~\mathrm{is~not~a}~\mathrm{free~factor~}\mathrm{of}~H
\end{gathered}
\right\} ,
\]
and let $\crit\left(\gamma\right)$ denote the number of subgroups
$H\le\Gamma$ satisfying the conditions in the definition of $\chimax\left(\gamma\right)$
with $\chi\left(H\right)=\chimax\left(\gamma\right)$. Then
\[
\mathbb{E}\left[\fix_{\gamma}\left(N\right)\right]-1=\left|\crit\left(\gamma\right)\right|\cdot N^{\chimax\left(\gamma\right)}\left(1+O\left(N^{-1/m}\right)\right).
\]
\end{conjecture}

The fact that $\crit\left(\gamma\right)$ is finite can be shown using
the techniques of the current paper. By Theorem \ref{thm:limit expectation of fix},
this conjecture is true for any element $\gamma$ with $\left|\H_{\gamma}\right|\ge1$.
Here are a few other examples illustrating the conjecture:
\begin{itemize}
\item Let $\Gamma=\Lambda_{2}=\left\langle a,b,c,d\,\middle|\,\left[a,b\right]\left[c,d\right]\right\rangle $
be the genus-2 surface group. Consider $\gamma=a$. It is possible
to obtain the following estimate: 
\[
\mathbb{E}\left[\fix_{a}\left(N\right)\right]=1+\frac{1}{N^{2}}+\frac{2}{N^{3}}+\frac{10}{N^{4}}+O\left(\frac{1}{N^{5}}\right).
\]
It seems that $a$ is primitive in every free subgroup of $\Gamma$
containing it, so the only subgroups containing it not inside a proper
free factor are the finite-index subgroups, which are all surface
groups. Among these, $\Gamma$ itself has maximal Euler characteristic:
$\chi\left(\Gamma\right)=2-2g=-2$. So $\chimax\left(\gamma\right)=-2$
and $\crit\left(\gamma\right)=\left\{ \Gamma\right\} $. This agrees
with the conjecture.
\item Let $\Gamma=\Lambda_{2}=\left\langle a,b,c,d\,\middle|\,\left[a,b\right]\left[c,d\right]\right\rangle $
again, and consider $\gamma=\left[a,b\right]$. Using a computer,
Michal Buran carried out a computation showing that most likely
\[
\mathbb{E}\left[\fix_{\left[a,b\right]}\left(N\right)\right]=1+\frac{2}{N}+O\left(\frac{1}{N^{2}}\right).
\]
This seems to agree with the conjecture as $\left[a,b\right]$ is
a non-primitive element in two free subgroups of Euler characteristic
$-1$:$\left\langle a,b\right\rangle $ and $\left\langle c,d\right\rangle $.
So $\chimax\left(\gamma\right)=-1$, and most likely $\crit\left(\gamma\right)=\left\{ \left\langle a,b\right\rangle ,\left\langle c,d\right\rangle \right\} $.
\item Now consider $\Gamma=C_{3}*C_{3}*C_{3}=\left\langle x\right\rangle *\left\langle y\right\rangle *\left\langle z\right\rangle $,
and let $\gamma=xyz$. The resolution $\R_{\gamma}$ from Section
\ref{sec:limit-distr-and-asymptotic-expansion of fix} contains five
elements, corresponding to the subgroups $\left\langle \gamma\right\rangle ,\left\langle x,yz\right\rangle ,\left\langle xz,xyx^{-1}\right\rangle ,\left\langle xy,z\right\rangle ,\left\langle x,y,z\right\rangle $.
It is possible to show that any critical subgroup of $\gamma$ must
be found inside this collection. But $\left\langle \gamma\right\rangle $
is not critical by definition, $\left\langle x,yz\right\rangle \cong C_{3}*\mathbb{Z}$
has also the decomposition $\left\langle x\right\rangle *\left\langle xyz\right\rangle $
so $\gamma$ belongs to a proper free factor. Similarly, $\left\langle xz,xyx^{-1}\right\rangle =\left\langle xyz\right\rangle *\left\langle xyx^{-1}\right\rangle $
and $\left\langle xy,z\right\rangle =\left\langle xyz\right\rangle *\left\langle z\right\rangle $.
This leaves us with $\left\langle x,y,z\right\rangle $ where $\left\langle x,y,z\right\rangle $
does not seem to belong to a proper free factor. We conclude that
most likely, $\chimax\left(\gamma\right)=\chi\left(\left\langle x,y,z\right\rangle \right)=\chi\left(\Gamma\right)=-1$,
and $\crit\left(\gamma\right)=\left\{ \Gamma\right\} $. We thus expect
that $\mathbb{E}\left[\fix_{\gamma}\left(N\right)\right]-1=N^{-1}+O\left(N^{-4/3}\right)$.\\
Using the results of \cite{muller1997finite}, we may compute the
leading terms of $\emb_{Y}\left(N\right)$ for these five sub-covers.
We get the following. For $\left\langle \gamma\right\rangle $ we
get $1-3N^{-2/3}+O\left(N^{-4/3}\right)$; For $\left\langle x,yz\right\rangle $
we get $N^{-2/3}+O\left(N^{-4/3}\right)$, and by symmetry, the same
leading term apply to $\left\langle xz,xyx^{-1}\right\rangle $ and
to $\left\langle xy,z\right\rangle $; For $\left\langle x,y,z\right\rangle $
we get $N^{-1}+O\left(N^{-5/3}\right)$. Overall the coefficients
of $N^{-2/3}$ cancel out and we get $\mathbb{E}\left[\fix_{\gamma}\left(N\right)\right]=1+N^{-1}+O\left(N^{-4/3}\right)$,
which agrees with the conjecture. 
\end{itemize}

\subsubsection*{Scope of the phenomena described in this paper}

We are curious as to what extent the results of this paper can be
generalized to a larger family of groups. As noted in Remark \ref{rem:Z^2},
Theorem \ref{thm:limit expectation of fix} does not hold when $\Gamma=\mathbb{Z}^{2}$.
Other results in this paper, such as Theorem \ref{thm:E^emb is N^chi and asym expansion},
require that $\chi\left(H\right)\le1$ for any subgroup of $\Gamma$:
indeed, the number of (injective) lifts of a connected sub-cover to
an arbitrary $N$-cover is bounded from above by $N$. Still, we are
positive that the results of this paper apply to groups not covered
by Assumption \ref{assu:Gamma}. For example, we suspect they are
true for fundamental groups of non-orientable surfaces of negative
Euler characteristic, and more generally to all Fuchsian groups. They
may also hold for general amalgams of finite groups. What it the widest
generality? Do some of the results, e.g., the asymptotic expansion
of Theorem \ref{thm:asmptotic expansion}, apply nonetheless to groups
such as $\mathbb{Z}^{2}$?

\bibliographystyle{alpha}
\bibliography{random-permutations-from-free-products}

\noindent Doron Puder, School of Mathematical Sciences, Tel Aviv University,
Tel Aviv, 6997801, Israel\\
\texttt{doronpuder@gmail.com}~\\

\noindent Tomer Zimhoni, School of Mathematical Sciences, Tel Aviv
University, Tel Aviv, 6997801, Israel\\
\texttt{tomerzimhoni@mail.tau.ac.il}
\end{document}